\documentclass{amsart}

\usepackage{amsmath,amssymb,amsfonts,amsthm,mathtools,mathrsfs}
\usepackage[margin=2.5cm]{geometry}
\usepackage[expansion=false]{microtype}
\usepackage{xcolor}
\usepackage[hidelinks,hypertexnames=false]{hyperref}

\definecolor{darkblue}{RGB}{30,65,115}
\hypersetup{
  colorlinks=true,
  linkcolor=darkblue,
  citecolor=darkblue,
  urlcolor=darkblue,
  pdftitle={Topological string blowup equations via stable pairs},
  pdfauthor={Lutian Zhao}
}

\newtheorem{theorem}{Theorem}
\newtheorem{proposition}{Proposition}
\newtheorem{lemma}{Lemma}
\newtheorem{corollary}{Corollary}
\newtheorem{conjecture}{Conjecture}
\theoremstyle{definition}
\newtheorem{definition}{Definition}
\numberwithin{equation}{section}

\newcommand{\C}{\mathbb C}
\newcommand{\Q}{\mathbb Q}
\newcommand{\Z}{\mathbb Z}
\newcommand{\cO}{\mathcal O}
\newcommand{\cF}{\mathcal F}
\newcommand{\cV}{\mathcal V}
\newcommand{\cW}{\mathcal W}
\newcommand{\Tot}{\operatorname{Tot}}
\newcommand{\End}{\operatorname{End}}
\newcommand{\Hom}{\operatorname{Hom}}
\newcommand{\PE}{\operatorname{PE}}
\newcommand{\ch}{\operatorname{ch}}
\newcommand{\rk}{\operatorname{rk}}
\newcommand{\M}{\mathfrak M}
\newcommand{\K}{\mathbb K}
\newcommand{\eps}{\varepsilon}
\newcommand{\PT}{\mathrm{PT}}
\newcommand{\llbracket}{\mathopen{[\![}}
\newcommand{\rrbracket}{\mathclose{]\!]}}

\title[Topological string blowup equations via stable pairs]{Topological string blowup equations via stable pairs}
\author{Lutian Zhao}
\address{Kavli Institute for the Physics and Mathematics of the Universe, University of Tokyo, Kashiwa, Chiba 277--8583, Japan}
\email{lutian.zhao@ipmu.jp}
\date{}
\subjclass[2020]{14N35, 14D21, 14C05, 17B37}
\keywords{stable pairs, framed sheaves, equivariant K-theory, blowup formulas, local Hirzebruch surfaces, local \(\mathbb P^2\)}

\begin{document}

\allowdisplaybreaks
\raggedbottom
\setlength{\emergencystretch}{2em}

\begin{abstract}
The blowup equations of Huang, Sun and Wang are bilinear relations satisfied by the refined topological string partition function of a local Calabi--Yau 3-fold.  We prove the blowup equations for the local Hirzebruch surfaces \(\operatorname{Tot}_{\mathbb F_\ell}K_{\mathbb F_\ell}\), \(0\leq\ell\leq2\), as identities of torus-equivariant symmetrized \(K\)-theoretic stable pair invariants; for \(\ell=2\) the invariants are localized indices.  The proof identifies the stable pair vertex sum, after division by the fibre contribution, with the equivariant Euler characteristic of \((\det\mathcal V)^\ell\) on the moduli space of framed rank \(2\) sheaves on \(\mathbb P^2\).  The blowup formulas of Nakajima--Yoshioka for framed sheaves then yield the unity and vanishing equations.

For local \(\mathbb P^2\) we obtain the blowup equations conditionally on two explicitly stated conjectures.  We also state the Huang--Sun--Wang conjecture for general local Calabi--Yau 3-folds in the language of stable pairs, and formulate stable pair conjectures for the local rational elliptic surface involving the \(E_8\) lattice.
\end{abstract}

\maketitle

\setcounter{tocdepth}{1}
\tableofcontents

\setcounter{section}{-1}
\section{Introduction}

\subsection{Overview}
Let \(Y\) be a nonsingular quasi-projective 3-fold, and let \(\beta\in H_2(Y,\Z)\) be a nonzero class.  A stable pair on \(Y\) is a section
\[
 \cO_Y\xrightarrow{s}F
\]
of a pure \(1\)-dimensional sheaf \(F\) with proper support and \(0\)-dimensional cokernel.  The moduli space \(P_n(Y,\beta)\) of stable pairs with
\[
 [F]=\beta,\qquad \chi(F)=n
\]
carries a perfect obstruction theory \cite{PandharipandeThomasStablePairs}, which is symmetric when \(K_Y\) is trivial; if \(Y\) is a toric Calabi--Yau 3-fold, the \(K\)-theoretic stable pair invariants may be calculated by virtual localization, and the outcome is expressed by the stable pair vertex \cite{PandharipandeThomasVertex,KuhnLiuThimm}.

We use the vertex to prove blowup equations for the local Hirzebruch surfaces
\[
 \mathbb F_\ell=\mathbb P_{\mathbb P^1}(\cO\oplus\cO(\ell)),
 \qquad Y_\ell=\Tot_{\mathbb F_\ell}K_{\mathbb F_\ell},
 \qquad 0\leq\ell\leq2.
\]
Taki identified the refined vertex series of the \(SU(N)\) geometries with the \(K\)-theoretic Nekrasov partition function \cite{TakiRefined}; Huang, Sun and Wang formulated blowup equations for the refined topological string on an arbitrary local Calabi--Yau 3-fold and computed the local Hirzebruch cases \cite{HuangSunWang}; Grassi and Gu used blowup equations to prove BPS relations for a family of toric 3-folds \cite{GrassiGu}.  Our main result, Theorem~\ref{thm:Hirzebruch-main} below, establishes the local Hirzebruch equations as identities among stable pair invariants.

The vertex computation behind Proposition~\ref{prop:Hirzebruch-PT-ADHM} is the rank \(2\) case of a calculation of Taki \cite{TakiRefined}, and reading it as a statement about stable pair invariants requires the square root branches of \cite{Arbesfeld,KuhnLiuThimm}, the sign \((-1)^\ell\) on the base class, the normalization by the fibre contribution, and a formal series reading of the degree \(0\) factors.  For \(\ell=2\) the moduli spaces need not be proper, so the invariants are localized indices.  The \(j=1\) equations for \(\mathbb F_1\) carry a phase from the fourth root of the counting variable.

Throughout the paper, \(\mathrm i\) denotes a square root of \(-1\).  For a row vector the superscript \({}^t\) is the transpose; for a partition \(\lambda\), \(\lambda^t\) is the conjugate partition.

\subsection{Local Hirzebruch surfaces}
Here \(\mathbb P(E)\) parametrizes \(1\)-dimensional subspaces of the fibres of \(E\).  Let \(B\subset\mathbb F_\ell\) be the torus-invariant section with \(B^2=-\ell\), and let \(F\) be the fibre class.\footnote{When \(\ell=0\), either of the two torus-invariant sections may be taken as \(B\).}  The effective curve classes form the monoid
\[
 \operatorname{NE}(\mathbb F_\ell)
 =\Z_{\geq0}B\oplus\Z_{\geq0}F.
\]
Let \(t,q\) be independent refinement parameters, with fixed square roots.  For \(\beta=d_BB+d_FF\), we write \(Q^\beta=Q_B^{d_B}Q_F^{d_F}\) in the formal power series ring \(\Q(t^{1/2},q^{1/2})[[Q_B,Q_F]]\).  The variables \(Q_B,Q_F\) are the Novikov coordinates of the basis \(B,F\); no numerical values or convergence conditions are imposed on them.  Let \(Q_B=e^{-T_B}\) and \(Q_F=e^{-T_F}\), and let
\begin{equation}
 T_M=T_B+\left(\frac\ell2-1\right)T_F.
\label{eq:intro-fixed-coordinate}
\end{equation}
The compact divisor \(\mathbb F_\ell\subset Y_\ell\) translates the coordinates by
\[
 (T_F,T_B)\longmapsto\bigl(T_F+2a,\,T_B+(2-\ell)a\bigr),
\]
and \(T_M\) is invariant.

Our sign convention weights the class \(d_BB+d_FF\) by
\begin{equation}
 \bigl((-1)^\ell Q_B\bigr)^{d_B}Q_F^{d_F}.
\label{eq:intro-signed-Novikov-specialization}
\end{equation}

\subsection{Blowup equations}
Let \(\eps_1,\eps_2\) be the additive equivariant parameters, and let
\[
 \mathbf t=(T_F,T_M)^t,\qquad \mathbf C=(2,0)^t.
\]
An integer vector \(\mathbf r=(r_F,r_M)^t\) prescribes a translation in the two logarithmic coordinates.  Two such vectors are equivalent,
\begin{equation}
 \mathbf r\sim\mathbf r'
 \quad\Longleftrightarrow\quad
 \mathbf r'-\mathbf r\in2\mathbf C\Z;
\label{eq:Hirzebruch-r-equivalence}
\end{equation}
replacing \(\mathbf r\) by an equivalent vector only reindexes the lattice sum below.  The vectors we use are indexed by pairs \((j,d)\) in the sets
\begin{equation}
\begin{aligned}
 \mathcal I_0&=
 \bigl(\{0\}\times\{-1,0,1,2,3\}\bigr)
 \cup\bigl(\{1\}\times\{0,1,2\}\bigr),\\
 \mathcal I_1&=
 \bigl(\{0\}\times\{0,1,2,3\}\bigr)
 \cup\bigl(\{1\}\times\{0,1,2\}\bigr),\\
 \mathcal I_2&=
 \bigl(\{0\}\times\{0,1,2,3,4\}\bigr)
 \cup\bigl(\{1\}\times\{0,1,2\}\bigr),
\end{aligned}
\label{eq:intro-index-sets}
\end{equation}
via
\begin{equation}
\mathbf r_{\ell;j,d}
 =\begin{pmatrix}2j\\ \ell(1-j)+2-2d\end{pmatrix}.
\label{eq:Hirzebruch-main-shift}
\end{equation}
Here \(j\in\{0,1\}\) selects a coset \(\Z+j/2\) of the lattice, and \(d\) is an integer which shifts the instanton counting variable.

Theorem~\ref{thm:Hirzebruch-main}, stated after the normalization and scalar coefficients have been defined, proves the stable pair blowup identity for these shifts.  Its primary meaning is coefficientwise: divide by the unshifted degree \(0\) factor and expand in the signed base counting variable.  Each coefficient is then a finite sum of rational functions.  An analytic identity follows on domains where compatible continuations and normally convergent sums exist.

Following Huang--Sun--Wang, an equation is a \emph{unity equation} when its right side is a nonzero scalar multiple of the unshifted partition function, and a \emph{vanishing equation} when its right side is \(0\).  Up to \eqref{eq:Hirzebruch-r-equivalence}, the unity equations have representatives
\[
\begin{aligned}
 \ell=0,2:\quad&
 (0,4),(0,2),(0,0),(0,-2),(0,-4),(2,2),(2,-2),\\
 \ell=1:\quad&
 (0,3),(0,1),(0,-1),(0,-3),(2,2),(2,-2).
\end{aligned}
\]
For every \(\ell\), the class of \((2,0)\) gives a vanishing equation.

\subsection{Framed sheaves}
Proposition~\ref{prop:Hirzebruch-PT-ADHM} identifies the stable pair vertex sum over the four charts of \(\mathbb F_\ell\), taken in a preferred slope limit and divided by the fibre class contribution, term by term with the generating series
\[
 \sum_{n\geq0}\mathfrak q^n
 \chi_{\mathbf T_{\mathrm{fr}}}^{\mathrm{loc}}
 \bigl(\M(2,n),(\det\cV)^\ell\bigr)
\]
of equivariant Euler characteristics on the moduli space \(\M(2,n)\) of framed rank \(2\) sheaves on \(\mathbb P^2\).  Here \(\cV\) is the tautological bundle.

The \(K\)-theoretic blowup formulas of Nakajima--Yoshioka \cite{NakajimaYoshiokaII,NakajimaYoshiokaPerverse}, the two extreme cases established by Bershtein--Shchechkin \cite{BershteinShchechkin}, and their consequences derived by Shchechkin \cite{Shchechkin} are collected in Theorem~\ref{thm:finite-NY-levels}.  Translating the shifts of the framed sheaf variables into translations of \((T_F,T_M)\) gives Theorem~\ref{thm:Hirzebruch-main}.  Corollary~\ref{cor:Hirzebruch-exact-HSW} rewrites the result in the integral coordinates \((T_F,T_B)\) and verifies the parity condition and the normalization of \cite{HuangSunWang}.

The case \(\ell=1\) is related to the open/closed generating-series correspondence of Gr\"afnitz, Ruddat, Zaslow and Zhou \cite[Theorem~5.3]{GraefnitzRuddatZaslowZhou}.  Their Corollaries~5.5 and~5.7 identify individual invariants at winding \(1\) and \(2\).  We do not use this correspondence in the proof, nor assert an identification of individual invariants at higher winding.

\subsection{Local \texorpdfstring{\(\mathbb P^2\)}{P2}}
The local \(\mathbb P^2\) partition function is reached from the local \(\mathbb F_1\) geometry by flopping the curve \(B\subset Y_1\), whose normal bundle in \(Y_1\) is \(\cO_{\mathbb P^1}(-1)\oplus\cO_{\mathbb P^1}(-1)\), and then taking the large-volume limit of the flopped curve while keeping the hyperplane coordinate fixed.  We obtain the blowup equations of local \(\mathbb P^2\) only conditionally, on two conjectures which we state separately.  Conjecture~\ref{conj:P2-vertex-comparison} identifies the stable pair series with an explicit vertex sum.  We have not been able to prove it.  The natural route is to cut the two toric curves through the third fixed point and apply a degeneration formula, and the \(K\)-theoretic degeneration formula for stable pairs, with control of the square roots, is not available in the form required.  Conjecture~\ref{conj:P2-asymptotic-resummation} is the analytic input needed to send the framing parameter of \(Y_1\) to \(0\), and once it is granted the rest of the argument, in Appendix~\ref{app:u-zero}, is unconditional.  Under these hypotheses we obtain three blowup equations (Theorem~\ref{thm:P2-blowup-main}), a finite recursion determining all invariants (Corollary~\ref{cor:P2-finite-coefficients}), and, after an explicit change of sign conventions, agreement of the first coefficients with the tables of Choi--Katz--Klemm \cite{ChoiKatzKlemm}.

\subsection{Conjectures}
Section~\ref{sec:general-HSW} states the general Huang--Sun--Wang blowup conjecture for a local Calabi--Yau 3-fold in a form adapted to stable pairs, with the lattice data of \cite{HuangSunWang} written in the present notation.  Section~\ref{sec:rational-elliptic-conjecture} formulates stable pair conjectures for the local rational elliptic surface.  There the curve classes are decomposed using the \(E_8\) lattice, the equations are indexed by the \(240\) roots and the zero vector, and the coefficients are theta functions.

\subsection{Plan of the paper}
Section~\ref{sec:PT-refined-vertex} reviews the symmetrized stable pair invariants, the vertex, and the preferred slope limits.  Section~\ref{sec:ADHM} concerns framed sheaves on \(\mathbb P^2\), and Section~\ref{sec:localization-ADHM} carries out the localization on \(Y_\ell\) and proves Proposition~\ref{prop:Hirzebruch-PT-ADHM}.  Section~\ref{sec:NY} proves Theorem~\ref{thm:Hirzebruch-main}.  Section~\ref{sec:stable-pair-blowup} sets up the flop and large-volume transition to local \(\mathbb P^2\) and Section~\ref{sec:P2-blowup} derives its blowup equations, with the analytic part of the argument in Appendix~\ref{app:u-zero}.  Sections~\ref{sec:general-HSW} and \ref{sec:rational-elliptic-conjecture} contain the conjectures.

\subsection*{Acknowledgements}
I thank Min-xin Huang, Kaiwen Sun and Xin Wang for introducing me to this problem, and Jinwon Choi, Sheldon Katz, Hiraku Nakajima and Weite Pi for helpful discussions.  I thank Helge Ruddat and Rubik Poghossian for useful feedbacks.  This work was supported by JSPS KAKENHI Grant Number JP25K17226.

\section{Stable pairs and the vertex}
\label{sec:PT-refined-vertex}

\subsection{Novikov rings}
Let \(R\) be a commutative \(\Q\)-algebra, and let \(\Gamma\) be a commutative monoid with zero.  Fix an additive map
\[
 \deg:\Gamma\to\Z_{\geq0}
\]
with \(\deg(\gamma)>0\) for \(\gamma\neq0\) and with \(\deg^{-1}([0,N])\) finite for every \(N\).  The degree completed monoid algebra is
\begin{equation}
 R\llbracket\Gamma\rrbracket_{\deg}
 :=\varprojlim_N
 R[\Gamma]/(Q^\gamma:\deg(\gamma)>N).
\label{eq:completed-monoid-algebra-general}
\end{equation}
Elements are written \(\sum_{\gamma\in\Gamma}a_\gamma Q^\gamma\) with \(Q^\gamma Q^{\gamma'}=Q^{\gamma+\gamma'}\).  We call \(Q^\gamma\) the Novikov monomial of \(\gamma\).  If
\[
 \Gamma=\Z_{\geq0}\gamma_1\oplus\cdots\oplus\Z_{\geq0}\gamma_r,
\]
then the \(Q_i=Q^{\gamma_i}\) are the Novikov coordinates and \(R\llbracket\Gamma\rrbracket_{\deg}=R[[Q_1,\ldots,Q_r]]\).  The augmentation is the constant coefficient homomorphism \(\sum_\gamma a_\gamma Q^\gamma\mapsto a_0\), and the augmentation ideal is its kernel.

Suppose \(R\) is a \(\lambda\)-ring with Adams operations \(\psi_m\), which we extend to the monoid algebra by \(\psi_m(Q^\gamma)=Q^{m\gamma}\), so that on a Laurent monomial in equivariant characters, counting variables and the formal square roots chosen below, \(\psi_m\) raises every factor to the \(m\)-th power.  For \(f\) in the augmentation ideal, the plethystic exponential is
\begin{equation}
 \PE[f]
 =\exp\left(\sum_{m\geq1}\frac{\psi_m(f)}m\right).
\label{eq:global-plethystic-exponential}
\end{equation}
The double Pochhammer symbol is
\begin{equation}
 (x;p_1,p_2)_\infty
 =\prod_{a,b\geq0}(1-xp_1^ap_2^b)
 =\exp\left[-\sum_{m\geq1}
 \frac{x^m}{m(1-p_1^m)(1-p_2^m)}\right],
\label{eq:global-double-Pochhammer}
\end{equation}
an identity in \(\Q(p_1,p_2)[[x]]\).  When an analytic value is required, the product is first taken for \(|p_1|,|p_2|<1\).  We write
\[
 \K:=\Q(t^{1/2},q^{1/2}).
\]

\subsection{Torus and variables}
Let \(Y\) be a nonsingular quasi-projective toric Calabi--Yau 3-fold with dense torus \(\mathbf T=(\C^*)^3\), let \(\kappa\in\mathbf T^\vee\) be the character of the \(\mathbf T\)-action on \(K_Y^\vee\), and let \(\widetilde{\mathbf T}\to\mathbf T\) be the finite isogeny of tori with character lattice
\[
 \widetilde{\mathbf T}^{\vee}
 =\mathbf T^\vee+\Z\frac{\kappa}{2}
 \subset \mathbf T^\vee\otimes_{\Z}\Q.
\]
Then \(\kappa^{1/2}\) is a character of \(\widetilde{\mathbf T}\).

The stable pair variable \(y\) and the refined vertex variables \(t,q\) are related by
\begin{equation}
 t=-y\kappa^{1/2},\qquad q=-y\kappa^{-1/2}.
\label{eq:PT-refined-variables}
\end{equation}
We use the compatible square roots of Arbesfeld \cite{Arbesfeld},
\begin{equation}
 (tq)^{1/2}=y,\qquad
 (t/q)^{1/2}=-\kappa^{1/2},\qquad
 (q/t)^{1/2}=-\kappa^{-1/2}.
\label{eq:PT-refined-square-root-branches}
\end{equation}

For a vector bundle or virtual bundle \(E\), let
\[
 \Lambda_{-1}(E)=\sum_{i\geq0}(-1)^i\Lambda^iE.
\]
For a virtual bundle \(E_1-E_2\) this means \(\Lambda_{-1}(E_1)/\Lambda_{-1}(E_2)\) in localized \(K\)-theory.

\subsection{Partitions}
A partition is \(\lambda=(\lambda_1\geq\lambda_2\geq\cdots)\).  We use
\[
 |\lambda|=\sum_i\lambda_i,\qquad
 \|\lambda\|^2=\sum_i\lambda_i^2,\qquad
 n(\lambda)=\sum_i(i-1)\lambda_i,
\]
and the transpose \(\lambda^t\).  For a box \(s=(i,j)\), the arm and leg lengths are
\[
 a_\lambda(s)=\lambda_i-j,\qquad \ell_\lambda(s)=\lambda_j^t-i.
\]
Let \(\rho_i=\frac12-i\), and let
\[
 c_\lambda=\|\lambda\|^2-\|\lambda^t\|^2,
 \qquad
 \widetilde Z_\lambda(t,q)
 =\prod_{s\in\lambda}
  \left(1-q^{a_\lambda(s)}t^{\ell_\lambda(s)+1}\right)^{-1}.
\]
The alphabet \(t^{-\rho}q^{-\nu}\) is \((t^{i-1/2}q^{-\nu_i})_{i\geq1}\).  Let \(s_{\alpha/\beta}\) be the skew Schur function.

\subsection{The refined vertex}
The refined vertex of Iqbal, Koz\c{c}az and Vafa \cite[Section~4]{IqbalKozcazVafa} is
\begin{equation}
\begin{aligned}
 C_{\lambda\mu\nu}(t,q)
 ={}&\left(\frac qt\right)^{\|\mu\|^2/2}
 t^{c_\mu/2}q^{\|\nu\|^2/2}\widetilde Z_\nu(t,q)\\
 &\times\sum_\eta
 \left(\frac qt\right)^{(|\eta|+|\lambda|-|\mu|)/2}
 s_{\lambda^t/\eta}(t^{-\rho}q^{-\nu})
 s_{\mu/\eta}(q^{-\rho}t^{-\nu^t}).
\end{aligned}
\label{eq:refined-vertex-definition}
\end{equation}
The distinguished leg is \(\nu\).  The two edge framing monomials are
\begin{equation}
\begin{aligned}
 f_\lambda(t,q)
 &=(-1)^{|\lambda|}
 t^{\|\lambda^t\|^2/2}q^{-\|\lambda\|^2/2},\\
 \widetilde f_\lambda(t,q)
 &=(-1)^{|\lambda|}
 t^{(\|\lambda^t\|^2-|\lambda|)/2}
 q^{-(\|\lambda\|^2-|\lambda|)/2}.
\end{aligned}
\label{eq:modified-framing}
\end{equation}

\subsection{Symmetrized virtual structure sheaves}
For an effective compact curve class \(\beta\), the space \(P_n(Y,\beta)\) parametrizes complexes
\[
 I^\bullet=\{\cO_Y\xrightarrow{s}F\},\qquad
 [F]=\beta,\qquad \chi(F)=n,
\]
with \(F\) pure of dimension \(1\) with proper support and \(\operatorname{coker}(s)\) of dimension \(0\) \cite{PandharipandeThomasStablePairs}.  The symmetric obstruction theory has virtual tangent class \(T^{\mathrm{vir}}\) and virtual canonical line
\[
 K^{\mathrm{vir}}=\det(T^{\mathrm{vir}})^\vee.
\]

Suppose \(P_n(Y,\beta)\) is proper.  A nonequivariant square root \((K^{\mathrm{vir}})^{1/2}\) exists by \cite[Section~6.2, Proposition~6.4]{NekrasovOkounkov}, and it has a canonical \(\widetilde{\mathbf T}\)-equivariant lift \cite[Section~7.2]{NekrasovOkounkov}.  We fix such a square root.  The symmetrized virtual structure sheaf is
\begin{equation}
 \widehat{\cO}^{\mathrm{vir}}
 :=\cO^{\mathrm{vir}}\otimes(K^{\mathrm{vir}})^{1/2}.
\label{eq:PT-symmetrized-sheaf}
\end{equation}
If \(i_{\mathscr F}:\mathscr F\hookrightarrow P_n(Y,\beta)\) is a connected component of the fixed locus, we write
\begin{equation}
 \widehat{\cO}_{\mathscr F}^{\mathrm{vir}}
 :=\cO_{\mathscr F}^{\mathrm{vir}}
   \otimes i_{\mathscr F}^*(K^{\mathrm{vir}})^{1/2}.
\label{eq:fixed-locus-symmetrized-sheaf}
\end{equation}
For the framed Donaldson--Thomas and stable pair vertex moduli spaces we use the symmetrized virtual structure sheaves of \cite[Lemma~3.3.7]{KuhnLiuThimm}, after passing to \(\widetilde{\mathbf T}\).  These are the sheaves in the vertex factors of the toric factorization \cite[Section~1.2, equation~(3)]{KuhnLiuThimm}.

The explicit Donaldson--Thomas character calculations below use the normalization of Arbesfeld.  If the virtual character is written \(V=V_+-\kappa V_+^\vee\), then
\begin{equation}
 \det(V_+)^\vee\kappa^{\rk(V_+)/2}
\label{eq:polarized-square-root-character}
\end{equation}
is a square root weight for \(\det(V)^\vee\).  Multiplying by the ordinary virtual localization denominator gives exactly the symmetrized fixed point contribution of \cite[equation~(2.6)]{Arbesfeld}.  On a positive dimensional component of a framed vertex moduli space we use the square root of \cite[Lemma~3.3.7]{KuhnLiuThimm}.

\begin{lemma}
\label{lem:fixed-locus-square-root-independence}
Let \(\mathscr F\) be a proper connected torus-fixed component, and let \(L_1,L_2\) be two \(\widetilde{\mathbf T}\)-equivariant square roots of the same equivariant line bundle on \(\mathscr F\).  For every localized equivariant \(K\)-class \(G\),
\[
 \chi_{\widetilde{\mathbf T}}(\mathscr F,G\otimes L_1)
 =\chi_{\widetilde{\mathbf T}}(\mathscr F,G\otimes L_2)
\]
after tensoring the representation ring with \(\Q\).  In particular, the Donaldson--Thomas fixed point contributions computed from \eqref{eq:polarized-square-root-character} agree with those computed from the square root of \cite[Lemma~3.3.7]{KuhnLiuThimm}.
\end{lemma}

\begin{proof}
The ratio \(M=L_1\otimes L_2^{-1}\) is an equivariant line bundle with \(M^{\otimes2}\cong\cO_{\mathscr F}\), so its first Chern class is \(2\)-torsion, and because \(\mathscr F\) is connected and fixed its linearization is a character of \(\widetilde{\mathbf T}\) of order dividing \(2\); the character lattice is torsion free, so the character is trivial.  Thus \(\ch_{\widetilde{\mathbf T}}(M)=1\) in rational equivariant cohomology.  Equivariant Hirzebruch--Riemann--Roch applied to \(G\) after localization gives the equality.  This is nothing more than the fixed locus form of the square root independence discussed in \cite[Section~2.6]{Arbesfeld}.  At every isolated Donaldson--Thomas fixed point of the framed vertex moduli space, both \eqref{eq:polarized-square-root-character} and the restriction of the square root of \cite[Lemma~3.3.7]{KuhnLiuThimm} square to the same virtual canonical character.
\end{proof}

\subsection{The stable pair series}
Let \(\operatorname{NE}_c(Y)\) be the monoid of effective curve classes represented by proper curves, and fix an additive degree
\[
 \deg:\operatorname{NE}_c(Y)\to\Z_{\geq0}
\]
satisfying the finiteness and positivity conditions of \eqref{eq:completed-monoid-algebra-general}.  We write \(\mathbf Q^\beta\) for the Novikov monomial of \(\beta\), and \(\beta>0\) means \(\beta\) is nonzero and effective.  If \(Y^{\mathbf T}\) is proper with inclusion \(j:Y^{\mathbf T}\hookrightarrow Y\), let
\begin{equation}
 \chi_{\mathbf T}^{\mathrm{loc}}(Y,\mathcal G)
 =\chi_{\mathbf T}\left(Y^{\mathbf T},
  \frac{j^*\mathcal G}{\Lambda_{-1}(N_{Y^{\mathbf T}/Y}^{\vee})}\right).
\label{eq:localized-euler-general}
\end{equation}
In a virtual localization formula, the normal bundle in \eqref{eq:localized-euler-general} is replaced by the moving part of the virtual tangent complex.

\begin{definition}
\label{def:KPT-series}
Suppose \(P_n(Y,\beta)\) is proper for every \(n\in\Z\) and every compact effective \(\beta\).  The symmetrized \(K\)-theoretic stable pair series is
\begin{equation}
\begin{aligned}
 Z_Y^{\mathrm{PT},K}(y,\mathbf Q)
 &:=1+\sum_{\substack{\beta>0\\ n\in\Z}}
 y^n\mathbf Q^\beta
 \chi_{\widetilde{\mathbf T}}\!\left(
 P_n(Y,\beta),\widehat{\cO}^{\mathrm{vir}}\right)\\
 &=1+\sum_{\substack{\beta>0\\ n\in\Z}}
 y^n\mathbf Q^\beta
 \sum_{\mathscr F\subset P_n(Y,\beta)^{\mathbf T}}
 \chi_{\widetilde{\mathbf T}}\left(
 \mathscr F,
 \frac{\widehat{\cO}_{\mathscr F}^{\mathrm{vir}}}
      {\Lambda_{-1}((N_{\mathscr F}^{\mathrm{vir}})^\vee)}
 \right).
\end{aligned}
\label{eq:KPT-series}
\end{equation}
\end{definition}

The second equality is virtual localization: \(\mathscr F\) runs over the connected fixed components, and the square root in \eqref{eq:fixed-locus-symmetrized-sheaf} is induced by the global square root \eqref{eq:PT-symmetrized-sheaf}.  For fixed \(\beta\), the integer \(n\) is bounded below.  The series therefore lies in
\[
 K_{\widetilde{\mathbf T}}^{\mathrm{loc}}((y))
 \llbracket\operatorname{NE}_c(Y)\rrbracket_{\deg},
\]
where \(R((y))=R[[y]][y^{-1}]\) and \(K_{\widetilde{\mathbf T}}^{\mathrm{loc}}\) is obtained from the representation ring of \(\widetilde{\mathbf T}\) by inverting the Euler factors \(1-\chi\) of the nontrivial characters \(\chi\) occurring in the moving virtual normal denominators.

\subsection{Slope limits}
To pass from the \(\mathbf T\)-equivariant index to a two-variable vertex formula, let \(\mathbf T_0=\ker(\kappa)\subset\mathbf T\).  A slope is a cocharacter \(\sigma:\C^*\to\mathbf T_0\).  For a localized character \(f\), let
\begin{equation}
 f^\sigma(g)=\lim_{a\to0}f(\sigma(a)g),
 \qquad g\in\widetilde{\mathbf T},
\label{eq:slope-limit}
\end{equation}
whenever the limit exists.  On a chart with weights \((t_1,t_2,t_3)\), write \(\sigma(a)=(a^{r_1},a^{r_2},a^{r_3})\) with \(r_1+r_2+r_3=0\).  For the third coordinate direction we use the four orders
\begin{equation*}
\begin{gathered}
 r_1\gg r_3>0\gg r_2,\qquad
 r_1\gg0>r_3\gg r_2,\\
 r_2\gg r_3>0\gg r_1,\qquad
 r_2\gg0>r_3\gg r_1.
\end{gathered}
\end{equation*}
Here \(\gg\) means coefficientwise stabilization.  For each fixed coefficient in the curve and stable pair variables only finitely many characters occur, and the value is required to stabilize along a sequence of cocharacters, which may depend on the coefficient; the first order is represented by \((r_1,r_3,r_2)=(N^2,N,-N^2-N)\) and the second by \((r_1,r_3,r_2)=(N^2+N,-N,-N^2)\) as \(N\to\infty\), the other two being obtained by interchanging the first two coordinates.  We call the result the coefficientwise limit along the preferred slope, and the limits in the other coordinate directions are defined by permuting the coordinates.

\subsection{Edge terms}
Consider a compact toric edge with tangent character \(h\) at its initial fixed point and ordered normal characters \((u,v)\), and write the normal bundle as \(\cO_{\mathbb P^1}(a)\oplus\cO_{\mathbb P^1}(b)\) with \(a+b=-2\).  For an edge partition \(\lambda\), let
\begin{equation}
\begin{aligned}
 \mathsf T_\lambda(u,v)
 &:=\sum_{s\in\lambda}
 \left(u^{-\ell_\lambda(s)}v^{a_\lambda(s)+1}
       +u^{\ell_\lambda(s)+1}v^{-a_\lambda(s)}\right),\\
 \mathsf E^{\mathrm{vir}}_{\lambda;(a,b)}(h,u,v)
 &:=\frac{\mathsf T_\lambda(u,v)}{1-h^{-1}}
  +\frac{\mathsf T_{\lambda^t}(h^{-b}v,h^{-a}u)}{1-h}.
\end{aligned}
\label{eq:general-edge-character}
\end{equation}
This is the compact edge virtual tangent character of \cite[equation~(4.8)]{Arbesfeld} in our arm and leg conventions.

For a virtual character \(V=\sum_i u_i-\sum_jv_j\), let
\begin{equation}
 \widehat a(V)
 =\frac{\prod_j(v_j^{1/2}-v_j^{-1/2})}
 {\prod_i(u_i^{1/2}-u_i^{-1/2})}.
\label{eq:symmetrized-Euler-character}
\end{equation}
If \(V=-\kappa V^\vee\), a polarization is a virtual character \(V_+\) with \(V=V_+-\kappa V_+^\vee\).  We use the polarization fixed by the cited vertex or edge formula and write
\[
 V_+=\sum_\chi m_\chi\chi,\qquad
 V=\sum_\chi m_\chi(\chi-\kappa\chi^{-1}),
 \qquad m_\chi\in\Z.
\]
A weight is attracting (respectively repelling) for \(\sigma\) if its exponent under \(\sigma(a)\) is positive (respectively negative).  For a generic slope \(\sigma\), let
\begin{equation}
 \operatorname{ind}_\sigma(V)
 =\sum_\chi m_\chi\,
   \operatorname{sgn}\bigl(\text{exponent of }\chi\text{ under }\sigma\bigr).
\label{eq:slope-index-definition}
\end{equation}
Then
\[
 (\widehat a(V))^\sigma=(-\kappa^{1/2})^{\operatorname{ind}_\sigma(V)}
\]
by \cite[Definition~4.1 and equation~(4.18)]{Arbesfeld}.  Let
\begin{equation}
 \chi(\lambda;(a,b))
 :=\sum_{(i,j)\in\lambda}
   \bigl(1-a(i-1)-b(j-1)\bigr)
 =|\lambda|-a n(\lambda)-b n(\lambda^t).
\label{eq:edge-partition-Euler-characteristic}
\end{equation}
The symmetrized edge factor is
\begin{equation}
 \widehat{\mathsf E}_{\lambda;(a,b)}(h,u,v)
 :=y^{\chi(\lambda;(a,b))}
   \widehat a\!\left(
    \mathsf E^{\mathrm{vir}}_{\lambda;(a,b)}(h,u,v)\right).
\label{eq:complete-symmetrized-edge-term}
\end{equation}
By \eqref{eq:PT-refined-variables} and \eqref{eq:slope-index-definition}, its coefficientwise limit along the preferred slope is the monomial
\begin{equation}
 t^{\left(\chi(\lambda;(a,b))+
 \operatorname{ind}_\sigma(
 \mathsf E^{\mathrm{vir}}_{\lambda;(a,b)}(h,u,v))\right)/2}
 q^{\left(\chi(\lambda;(a,b))-
 \operatorname{ind}_\sigma(
 \mathsf E^{\mathrm{vir}}_{\lambda;(a,b)}(h,u,v))\right)/2}.
\label{eq:edge-preferred-limit-monomial}
\end{equation}

\subsection{Vertex--edge factorization}
Let \(\mathsf V^{\mathrm{PT}}_{\lambda\mu\nu}\) be the symmetrized stable pair vertex of \cite[Section~1.2, equation~(3)]{KuhnLiuThimm}; when all three asymptotic partitions are empty its value is \(1\).  For a compact edge \(e\) with data \((h,u,v;a,b)\) as above, let
\[
 \mathsf E_{e,\lambda}
 :=\widehat{\mathsf E}_{\lambda;(a,b)}(h,u,v).
\]
Let \(\boldsymbol\lambda\) assign a partition \(\lambda_e\) to every compact edge, with \(\varnothing\) assigned to every noncompact edge, let \(\mathbf Q_e=Q^{[e]}\) be the Novikov monomial of the class of \(e\), and, at a torus-fixed vertex \(v\), let \(\lambda_{v,1},\lambda_{v,2},\lambda_{v,3}\) be the partitions on the three incident edges.  The vertex--edge factorization is
\begin{equation}
 Z_Y^{\mathrm{PT},K}
 =
 \sum_{\boldsymbol\lambda}
 \prod_v
 \mathsf V^{\mathrm{PT}}_{\lambda_{v,1},
 \lambda_{v,2},\lambda_{v,3}}
 \prod_e\mathsf E_{e,\lambda_e}\,
 \mathbf Q_e^{|\lambda_e|},
\label{eq:KPT-vertex-edge-sum}
\end{equation}
the products being over the torus-fixed vertices and the compact edges; when the stable pair moduli spaces are proper, \cite[Section~1.2, equation~(3)]{KuhnLiuThimm} identifies \eqref{eq:KPT-vertex-edge-sum} with the global invariant \eqref{eq:KPT-series}.  If only the torus-fixed loci are proper, we \emph{define} the localized series coefficientwise by the right side of \eqref{eq:KPT-vertex-edge-sum}, as we do for \(Y_2\); for \(Y_0\) and \(Y_1\) the same expression is the vertex--edge expansion of the global invariant.

Let \(\mathsf V^{\mathrm{DT}}_{\lambda\mu\nu}\) be the symmetrized Donaldson--Thomas vertex, with each compact edge term assigned to the same adjacent fixed point as in \(\mathsf V^{\mathrm{PT}}_{\lambda\mu\nu}\).

\begin{theorem}[Kuhn--Liu--Thimm \cite{KuhnLiuThimm}]
\label{thm:KLT-vertex}
For all asymptotic partitions \(\lambda,\mu,\nu\),
\begin{equation}
 \mathsf V^{\mathrm{DT}}_{\lambda\mu\nu}
 =
 \mathsf V^{\mathrm{DT}}_{\varnothing\varnothing\varnothing}
 \mathsf V^{\mathrm{PT}}_{\lambda\mu\nu}.
\label{eq:KLT-vertex}
\end{equation}
\end{theorem}

The equality is coefficientwise in \(y\) with coefficients in \(K_{\widetilde{\mathbf T}}^{\mathrm{loc}}\), for the symmetrized virtual structure sheaves of \cite{KuhnLiuThimm}, and on the Donaldson--Thomas side Lemma~\ref{lem:fixed-locus-square-root-independence} allows the fixed point contributions to be computed from \eqref{eq:polarized-square-root-character}.  The stable pair side uses the square root of \cite[Lemma~3.3.7]{KuhnLiuThimm}.

The factorization \eqref{eq:KPT-vertex-edge-sum} uses the redistribution of vertex and edge terms arising from the \v Cech complex of \cite[Section~3.3]{KuhnLiuThimm}, and, after converting tangent weights to coordinate function characters, the compact edge character \eqref{eq:general-edge-character} is the redistributed edge term of \cite[equation~(4.8)]{Arbesfeld}, assigned to the initial fixed point.

\subsection{Preferred slope limits of the vertex}
Arbesfeld's preferred slope vertex differs from \eqref{eq:refined-vertex-definition}.  Let
\begin{equation}
\begin{aligned}
 \mathsf C^{\mathrm A}_{\alpha\beta\gamma}(t,q)
 :={}&t^{-\|\alpha^t\|^2/2}q^{-\|\beta\|^2/2}
 \widetilde Z_{\gamma^t}(t,q)\\
 &\times\sum_\eta\left(\frac qt\right)^{|\eta|/2}
 s_{\alpha/\eta}(t^{-\rho}q^{-\gamma^t})
 s_{\beta^t/\eta}(q^{-\rho}t^{-\gamma}).
\end{aligned}
\label{eq:Arbesfeld-vertex}
\end{equation}
This is the function \(C(\alpha,\beta,\gamma)(t,q)\) defined immediately before \cite[Proposition~4.6]{Arbesfeld}.  Comparing with \eqref{eq:refined-vertex-definition},
\begin{equation}
 \mathsf C^{\mathrm A}_{\alpha\beta\gamma}(t,q)
 =a_{\alpha\beta\gamma}(t,q)
  C_{\alpha^t\beta^t\gamma^t}(t,q),
\label{eq:Arbesfeld-IKV-prefactor}
\end{equation}
where
\begin{equation}
 a_{\alpha\beta\gamma}(t,q)
 =t^{(-\|\alpha^t\|^2+\|\beta\|^2+|\alpha|-|\beta|)/2}
  q^{-(\|\beta\|^2+\|\beta^t\|^2+\|\gamma^t\|^2
       +|\alpha|-|\beta|)/2}.
\label{eq:Arbesfeld-IKV-monomial}
\end{equation}
We keep the monomial \eqref{eq:Arbesfeld-IKV-monomial} together with the compact edge terms.

\begin{proposition}
\label{prop:chart-slope-limit}
For a generic cocharacter satisfying one of the four orders above, the coefficientwise slope limit of \(\mathsf V^{\mathrm{DT}}_{\lambda\mu\nu}/ \mathsf V^{\mathrm{DT}}_{\varnothing\varnothing\varnothing}\), equivalently of \(\mathsf V^{\mathrm{PT}}_{\lambda\mu\nu}\), is respectively
\begin{equation*}
\begin{array}{c|c}
 r_1\gg r_3>0\gg r_2&
 \mathsf C^{\mathrm A}_{\lambda\mu\nu}(t,q)\\
 r_1\gg0>r_3\gg r_2&
 \mathsf C^{\mathrm A}_{\mu^t\lambda^t\nu^t}(q,t)\\
 r_2\gg r_3>0\gg r_1&
 \mathsf C^{\mathrm A}_{\mu^t\lambda^t\nu^t}(t,q)\\
 r_2\gg0>r_3\gg r_1&
 \mathsf C^{\mathrm A}_{\lambda\mu\nu}(q,t).
\end{array}
\end{equation*}
\end{proposition}

\begin{proof}
By Theorem~\ref{thm:KLT-vertex}, the quotient of the Donaldson--Thomas vertex of \cite{KuhnLiuThimm} by its value at three empty partitions is the stable pair vertex, Lemma~\ref{lem:fixed-locus-square-root-independence} identifies the localized Donaldson--Thomas fixed point contributions with those computed from Arbesfeld's polarized character \eqref{eq:polarized-square-root-character}, and the four preferred slope formulas stated before \cite[Proposition~4.6]{Arbesfeld}, together with that proposition, compute the normalized Donaldson--Thomas limits.  Written in terms of \(\mathsf C^{\mathrm A}\), they are the four rows of the table.
\end{proof}

\begin{lemma}
\label{lem:global-preferred-slope-limit}
Let \(\sigma_N\) be a sequence of generic cocharacters which, at every toric chart, eventually satisfies one of the four orders of Proposition~\ref{prop:chart-slope-limit}.  Fix \(\beta\) and \(n\), and suppose only finitely many edge partition assignments and tuples of vertex degrees contribute to the coefficient of \(\mathbf Q^\beta y^n\).  Then this coefficient has a preferred slope limit, and the limit is obtained by taking the limit of every normalized vertex function and of every symmetrized edge factor before multiplying and summing over the edge partitions.
\end{lemma}

\begin{proof}
Fix partitions \(\lambda_e\) on the compact edges and write
\[
 \mathsf V_v=\sum_m \mathsf V_{v,m}y^m,
 \qquad
 \mathsf E_{e,\lambda_e}=y^{c_e}\mathsf A_{e,\lambda_e},
 \qquad
 c_e=\chi(\lambda_e;(a_e,b_e)).
\]
The coefficient in question is
\[
 \left[y^n\right]
 \prod_v\mathsf V_v\prod_e\mathsf E_{e,\lambda_e}
 =\left(\prod_e\mathsf A_{e,\lambda_e}\right)
  \sum_{\sum_v m_v=n-\sum_e c_e}
  \prod_v\mathsf V_{v,m_v}.
\]
The sum on the right is finite by hypothesis.  Every \(\mathsf V_{v,m_v}\) in it stabilizes along \(\sigma_N\) by Proposition~\ref{prop:chart-slope-limit}, and every \(\mathsf A_{e,\lambda_e}\) stabilizes by \eqref{eq:edge-preferred-limit-monomial}, and the limit of the finite sum is the product of the limits; the same remark applies to the finite sum over edge partition assignments.  The lower bound on \(n\) for fixed \(\beta\) makes the result a Laurent series in \(y\).
\end{proof}

\section{Framed sheaves on \texorpdfstring{\(\mathbb P^2\)}{P2}}
\label{sec:ADHM}

\subsection{The ADHM description}
On a framed sheaf moduli space, \(q_1,q_2\) denote the characters of the two coordinate functions on \(\mathbb A^2\), the characters \(u_1,\ldots,u_r\) act on the framing space, and \(\mathfrak q\) is the formal variable recording \(c_2\).  In rank \(2\) we let \(u=u_1/u_2\).  Multiplying both framing characters by a common character allows the normalization \(u_2=1\).  

Let \(V=\C^n\) and \(W=\C^r\).  The moduli space of framed sheaves has the ADHM presentation
\begin{equation}
\begin{aligned}
 \M(r,n)=\bigl\{(B_1,B_2,I,J)\in{}&
 \End(V)^{\oplus2}\oplus\Hom(W,V)\oplus\Hom(V,W):\\
 & [B_1,B_2]+IJ=0,\ \text{stable}\bigr\}/GL(V),
\end{aligned}
\label{eq:ADHM-quotient}
\end{equation}
where stability means that a \(B_1,B_2\)-invariant subspace containing \(I(W)\) is all of \(V\).  The quotient \(\M(r,n)\) is nonsingular of dimension \(2rn\).  By the ADHM correspondence it is the moduli space of rank \(r\) torsion free sheaves on \(\mathbb P^2\) framed along the line at infinity, with \(c_1=0\) and \(c_2=n\) \cite{NakajimaYoshiokaI,NakajimaYoshiokaII}.

Let
\[
 \mathbf T_{\mathrm{fr}}=(\C^*)^2_{q_1,q_2}\times
 (\C^*)^r_{u_1,\ldots,u_r}.
\]
Let \(\cV\) be the rank \(n\) bundle associated to the principal \(GL(V)\)-bundle of \eqref{eq:ADHM-quotient} via the defining representation.  This is the tautological bundle.  Let \(\cW=W\otimes\cO\), with character \(u_1+\cdots+u_r\).  The ADHM deformation complex gives, in \(K_{\mathbf T_{\mathrm{fr}}}(\M(r,n))\),
\begin{equation}
 T^*\M(r,n)
 =\cW\cV^\vee+q_1q_2\cV\cW^\vee
 -(1-q_1)(1-q_2)\cV^\vee\cV,
\label{eq:ADHM-tangent}
\end{equation}
see \cite[Section~2]{NakajimaYoshiokaII}.

\begin{definition}\label{def:ADHM-index}
For \(\ell\in\Z\), let
\begin{equation}
 Z_{r,\ell}^{\mathrm{fr}}
 (\mathfrak q;q_1,q_2;u_1,\ldots,u_r)
 =\sum_{n\geq0}\mathfrak q^n
 \chi_{\mathbf T_{\mathrm{fr}}}^{\mathrm{loc}}
 \bigl(\M(r,n),(\det\cV)^\ell\bigr),
\label{eq:ADHM-index}
\end{equation}
where
\begin{equation}
 \chi_{\mathbf T_{\mathrm{fr}}}^{\mathrm{loc}}(X,\cF)
 =\sum_{p\in X^{\mathbf T_{\mathrm{fr}}}}
 \frac{\cF|_p}{\Lambda_{-1}(T_p^*X)}.
\label{eq:localized-Euler-convention}
\end{equation}
\end{definition}

We have
\[
 \operatorname{char}_{\mathbf T_{\mathrm{fr}}}\C[x,y] =((1-q_1)(1-q_2))^{-1}.
\]
In particular,
\[
 Z_{r,\ell}^{\mathrm{fr}}
 \in\Q(q_1,q_2,u_1,\ldots,u_r)[[\mathfrak q]].
\]

\subsection{Fixed points}
The fixed points are indexed by \(r\)-tuples of partitions \(\boldsymbol\lambda=(\lambda_1,\ldots,\lambda_r)\).  If the box \((i,j)\) has weight \(q_1^{i-1}q_2^{j-1}\), then
\begin{align}
 \cV|_{\boldsymbol\lambda}
 &=\sum_{a=1}^r u_a
   \sum_{(i,j)\in\lambda_a}q_1^{i-1}q_2^{j-1},
\label{eq:V-fixed-weight}\\
 \det\cV|_{\boldsymbol\lambda}
 &=\prod_{a=1}^r u_a^{|\lambda_a|}
   q_1^{n(\lambda_a)}q_2^{n(\lambda_a^t)}.
\label{eq:det-fixed-weight}
\end{align}
For an arbitrary box \(s=(i,j)\), whether or not in \(\lambda\), we set \(a_\lambda(s)=\lambda_i-j\) and \(\ell_\lambda(s)=\lambda_j^t-i\), where missing parts of a partition are \(0\).  Let
\begin{equation}
\begin{aligned}
 N_{\lambda\mu}(Q;q_1,q_2)
 :={}&\prod_{s\in\lambda}
 \left(1-Qq_1^{-\ell_\mu(s)}q_2^{a_\lambda(s)+1}\right)\\
 &\times\prod_{s\in\mu}
 \left(1-Qq_1^{\ell_\lambda(s)+1}q_2^{-a_\mu(s)}\right).
\end{aligned}
\label{eq:Nekrasov-factor}
\end{equation}
By \eqref{eq:ADHM-tangent}, the fixed point denominator is
\begin{equation}
 \Lambda_{-1}\!\left(T^*_{\boldsymbol\lambda}\M(r,n)\right)
 =\prod_{a,b=1}^r
 N_{\lambda_a\lambda_b}(u_a/u_b;q_1,q_2).
\label{eq:ADHM-Nekrasov-denominator}
\end{equation}
The closed formula \eqref{eq:Nekrasov-factor} for the tangent character at a fixed point, in terms of the arm and leg lengths of the two partitions, goes back to Flume and Poghossian, who derived it in the course of an algorithm for the instanton expansion of the Seiberg--Witten prepotential \cite[equations~(4.45)--(4.46)]{FlumePoghossian}.  Formula \eqref{eq:ADHM-Nekrasov-denominator} is their result written in the conventions used here.  These are those of the framed sheaf generating series of \cite{NakajimaYoshiokaII}.

\subsection{Linearization}
\begin{lemma}
\label{lem:linearization}
Changing the equivariant linearization from \(\cV\) to \(c\otimes\cV\), where \(c\) is a torus character, replaces
\[
 Z_{r,\ell}^{\mathrm{fr}}(\mathfrak q)
 \quad\text{by}\quad
 Z_{r,\ell}^{\mathrm{fr}}(c^\ell\mathfrak q).
\]
A simultaneous rescaling \(u_a\mapsto au_a\) has the same effect with \(c=a\).
\end{lemma}

\begin{proof}
On \(\M(r,n)\), the line \((\det\cV)^\ell\) acquires the character \(c^{\ell n}\), which is absorbed by \(\mathfrak q^n\).
\end{proof}

\section{Localization on the local Hirzebruch surfaces}
\label{sec:localization-ADHM}

We prove Proposition~\ref{prop:Hirzebruch-PT-ADHM} by evaluating the localization formula on the four affine charts of \(\mathbb F_\ell\), dividing by the fibre class contribution, and comparing the resulting sum over pairs of partitions with the fixed point formula for
\[
 \chi_{\mathbf T_{\mathrm{fr}}}^{\mathrm{loc}}(\M(2,n),(\det\cV)^\ell).
\]
Throughout the section \(0\leq\ell\leq2\), and \(\mathbb F_\ell,Y_\ell,B,F\) are as in the Introduction.

\subsection{Toric geometry}
Let \(B_+\) be the toric section disjoint from \(B\); for \(\ell=0\) it is the second section of the chosen projection \(\mathbb F_0\to\mathbb P^1\).  Let \(F_0,F_\infty\) be the two torus-invariant fibres.  Then
\[
 [B_+]=[B]+\ell[F].
\]
The four fixed points are
\[
 P_{00}=B\cap F_0,\quad P_{01}=B_+\cap F_0,\quad
 P_{10}=B\cap F_\infty,\quad P_{11}=B_+\cap F_\infty.
\]
Keep the anticanonical character \(\kappa\), and choose tangent characters \(\xi\) along \(B\) and \(\eta\) along \(F_0\) at \(P_{00}\).  In this section the slope specialization is the coefficientwise limit
\begin{equation}
 \xi\longmapsto a^R\xi,\qquad
 \eta\longmapsto a^S\eta,\qquad
 \kappa\longmapsto\kappa,\qquad S\gg R>0,
\label{eq:Hirzebruch-slope-choice}
\end{equation}
where \(\gg\) has the stabilization meaning of Section~\ref{sec:PT-refined-vertex}.

The zero section \(\mathbb F_\ell\subset Y_\ell\) is a compact divisor with
\[
 -F\cdot\mathbb F_\ell=2,\qquad
 -B\cdot\mathbb F_\ell=2-\ell.
\]
The divisor shift in \((T_F,T_B)\) is therefore proportional to \((2,2-\ell)\), and the coordinate \(T_M\) of \eqref{eq:intro-fixed-coordinate} is unchanged by it.

\subsection{Variables}
With \(Q_B=e^{-T_B}\) and \(Q_F=e^{-T_F}\) as in the Introduction, let
\begin{equation}
 U=Q_F^{-1},\qquad
 e^{-T_M}=Q_BQ_F^{\ell/2-1},\qquad
 \zeta_\ell=(-1)^\ell e^{-T_M}.
\label{eq:Hirzebruch-framed-variables}
\end{equation}
We fix the logarithm
\begin{equation}
 \log\zeta_\ell=-T_M+\pi\mathrm{i}\ell,
 \qquad
 \zeta_\ell^{1/4}=e^{\pi\mathrm{i}\ell/4}e^{-T_M/4}.
\label{eq:Hirzebruch-signed-branch}
\end{equation}
The sign agrees with \eqref{eq:intro-signed-Novikov-specialization}.

Let \(p_i=e^{\eps_i}\).  The slope specialized stable pair series is written in these variables via
\begin{equation}
 (t,q)=(p_1,p_2^{-1}).
\label{eq:Hirzebruch-refined-variable-change}
\end{equation}
The branch \((tq)^{1/2}=y\) of \eqref{eq:PT-refined-square-root-branches} then gives \(y=p_1^{1/2}p_2^{-1/2}\), and the coordinate function characters of Definition~\ref{def:ADHM-index} are \(p_1^{-1},p_2^{-1}\).

For \(\ell=0,1\), let \(Z_{Y_\ell}^{\PT,\mathrm{ref}}\) be the global series of Definition~\ref{def:KPT-series} after the slope specialization \eqref{eq:Hirzebruch-slope-choice}, and, for \(\ell=2\), let \(Z_{Y_2}^{\PT,\mathrm{ref}}\) be the coefficientwise vertex--edge series \eqref{eq:KPT-vertex-edge-sum} after the same specialization; Lemma~\ref{lem:Hirzebruch-fixed-proper} below shows the moduli spaces in the first case, and the fixed loci in the second, are proper.  For \(\beta=d_BB+d_FF\), we write
\begin{equation}
 Z_{Y_\ell}^{\PT,\mathrm{ref}}(Q_B,Q_F)
 =\sum_{\substack{d_B,d_F\geq0\\ n\in\Z}}Z_{d_B,d_F,n}\,
   \bigl((-1)^\ell Q_B\bigr)^{d_B}Q_F^{d_F}y^n,
\label{eq:Hirzebruch-signed-series}
\end{equation}
where \(Z_{d_B,d_F,n}\) is the coefficient supplied by Definition~\ref{def:KPT-series} for \(\ell=0,1\) and by \eqref{eq:KPT-vertex-edge-sum} for \(\ell=2\).  Equivalently, \eqref{eq:Hirzebruch-signed-series} is the image of the series under the change of Novikov coordinates
\[
 \mathbf Q^{d_BB+d_FF}\longmapsto
 \bigl((-1)^\ell Q_B\bigr)^{d_B}Q_F^{d_F}
\]
on the same monoid of effective classes.

\subsection{The fibre normalization}
We divide by the coefficient of \(Q_B^0\).  Let
\begin{align}
 Z^{\mathrm{fib}}(Q_F;p_1,p_2)
&:={Z}_{Y_\ell}^{\PT,\mathrm{ref}}(0,Q_F;p_1,p_2),
\label{eq:Hirzebruch-fibre-series}\\
 Z_{Y_\ell}^{\PT,\mathrm{norm}}(Q_B,Q_F;p_1,p_2)
&:=\frac{Z_{Y_\ell}^{\PT,\mathrm{ref}}(Q_B,Q_F;p_1,p_2)}
 {Z^{\mathrm{fib}}(Q_F;p_1,p_2)}.
\label{eq:Hirzebruch-base-normalization}
\end{align}
In the specialization \((p_1,p_2)=(t,q^{-1})\), the coefficient of \(Q_B^0\) is independent of \(\ell\):
\begin{equation}
\begin{aligned}
 Z^{\mathrm{fib}}(Q_F;t,q^{-1})
 &=\prod_{i,j\geq1}
   \frac1{(1-Q_Ft^jq^{i-1})(1-Q_Fq^jt^{i-1})}\\
 &=\PE\left[\frac{(t+q)Q_F}{(1-t)(1-q)}\right].
\end{aligned}
\label{eq:Hirzebruch-empty-strip-normalization}
\end{equation}
Indeed, with empty partitions on \(B\) and \(B_+\), the four vertex factors split into the two fibre edge Cauchy contractions.  The Cauchy identity
\[
 \sum_\lambda s_\lambda(\mathbf x)s_\lambda(\mathbf y)
 =\prod_{i,j\geq1}(1-x_iy_j)^{-1}
\]
gives, with the vertex and edge monomials of \eqref{eq:refined-vertex-definition} and \eqref{eq:KPT-vertex-edge-sum},
\[
 \prod_{i,j\geq1}(1-Q_Ft^jq^{i-1})^{-1},
 \qquad
 \prod_{i,j\geq1}(1-Q_Fq^jt^{i-1})^{-1}
\]
respectively.  Taking logarithms of the product gives the second line of \eqref{eq:Hirzebruch-empty-strip-normalization}.

\subsection{The framed sheaf side}
Fix a square root \(U^{1/2}\) and additive lifts \(\log U,\log p_1,\log p_2,\log z\).  At the geometric specialization the lifts are
\[
 \log U=T_F,\qquad \log p_i=\eps_i,\qquad \log z=-T_M+\pi\mathrm{i}\ell.
\]
Let \(U_1=U^{1/2}\) and \(U_2=U^{-1/2}\), and let
\begin{equation}
 \mathcal A_\ell(U;p_1,p_2\mid z)
 :=Z_{2,\ell}^{\mathrm{fr}}
 \left(z(p_1p_2)^{-1-\ell/2};
 p_1^{-1},p_2^{-1};U_1^{-1},U_2^{-1}\right).
\label{eq:Hirzebruch-ADHM-series}
\end{equation}
Let
\begin{align}
 \mathcal A^{(0)}(U;p_1,p_2\mid z)
 &:=(p_1^{-1}p_2^{-1}z)^{
 -\frac{(\log U)^2}{4\log p_1\log p_2}},
\label{eq:Hirzebruch-degree-zero-exponential}\\
 \mathcal E(U;p_1,p_2)
 &:=(U;p_1,p_2)_\infty(U^{-1};p_1,p_2)_\infty,
\label{eq:Hirzebruch-Pochhammer-factor}
\end{align}
and
\begin{equation}
\widehat{\mathcal A}_\ell(U;p_1,p_2\mid z)
 :=\mathcal A^{(0)}(U;p_1,p_2\mid z)
   \mathcal E(U;p_1,p_2)
   \mathcal A_\ell(U;p_1,p_2\mid z).
\label{eq:Hirzebruch-completed-ADHM-series}
\end{equation}
In the gauge theory terminology of Shchechkin, \(\mathcal A^{(0)}\) and \(\mathcal E\) are the classical and one-loop factors \cite[equations~(2.15)--(2.16)]{Shchechkin}.

On any domain to which \(\mathcal E\) has been meromorphically continued, let
\begin{align}
 F_\ell^{(0)}(T_F,T_M;\eps_1,\eps_2)
 &:=\frac{T_F^2(T_M+\eps_1+\eps_2-\pi\mathrm{i}\ell)}
 {4\eps_1\eps_2},
\label{eq:Hirzebruch-degree-zero-polynomial}\\
 \widehat Z_{Y_\ell}(T_F,T_M;\eps_1,\eps_2)
 :={}&e^{F_\ell^{(0)}(T_F,T_M;\eps_1,\eps_2)}
 \mathcal E(e^{T_F};e^{\eps_1},e^{\eps_2})\notag\\
 &\times Z_{Y_\ell}^{\PT,\mathrm{norm}}
 (e^{-T_B},e^{-T_F};e^{\eps_1},e^{\eps_2}),
 \qquad T_B=T_M+\left(1-\frac\ell2\right)T_F.
\label{eq:Hirzebruch-full-PT}
\end{align}
The degree \(0\) part is written in logarithmic coordinates and is not a formal power series in \(Q_B,Q_F\).  The base degree \(0\) coefficient of \(\widehat Z_{Y_\ell}\) is \(e^{F_\ell^{(0)}}\mathcal E(e^{T_F};e^{\eps_1},e^{\eps_2})\); only \(Z_{Y_\ell}^{\PT,\mathrm{norm}}\) has base degree \(0\) coefficient \(1\).

For independent parameters the double product defining \(\mathcal E(U;p_1,p_2)\) need not converge.  The notation for the full partition function records its explicit degree \(0\) factor; it does not assert convergence of its instanton series.  Only the quotient of the two shifted degree \(0\) factors by the unshifted factor enters the normalized blowup identity.  Lemma~\ref{lem:Hirzebruch-shift-quotients} defines this quotient rationally after extracting its monomial.  The analytic interpretation requires the additional convergence assumptions stated in Theorem~\ref{thm:Hirzebruch-main}.

\subsection{Properness}
\begin{lemma}
\label{lem:Hirzebruch-fixed-proper}
For every compact curve class \(\beta\) and \(n\in\Z\), the fixed locus \(P_n(Y_\ell,\beta)^{\mathbf T}\) is proper.  For \(\ell=0,1\), the moduli space \(P_n(Y_\ell,\beta)\) is proper.
\end{lemma}

Consequently the localized symmetrized series
\begin{equation}
 Z_{Y_\ell}^{\PT,\mathrm{ref}}
 (Q_B,Q_F;p_1,p_2)
\label{eq:Hirzebruch-PT-series}
\end{equation}
is defined coefficientwise.

\begin{proof}
The fixed loci have the box configuration description of \cite[Sections~2.5--2.6]{PandharipandeThomasVertex}, and a fixed stable pair is specified by a stable pair vertex at each of the four fixed points of \(\mathbb F_\ell\) and by matching partitions on the compact edges.  Fixing \(\beta\) bounds the edge partitions.  For fixed edge partitions, the formula for \(\chi(F)=n\) of \cite[Section~2.5]{PandharipandeThomasVertex} bounds the boxes not forced by the asymptotic edge partitions, and, for each collection of edge partitions and finite box data, the reduced vertex components are finite products of \(\mathbb P^1\)s, the reduced global fixed component being their product with the prescribed edge matchings \cite[Propositions~2--3 and Section~2.6]{PandharipandeThomasVertex}.  Only finitely many such data occur, so the reduced fixed locus is projective, and properness is insensitive to nilpotent thickening.

For \(\ell=0,1\), the anticanonical divisor of \(\mathbb F_\ell\) is ample.  Compactify \(Y_\ell\) by \(\mathbb P_{\mathbb F_\ell}(\cO\oplus K_{\mathbb F_\ell})\), with divisor at infinity \(D_\infty\); since \(\cO(D_\infty)|_{D_\infty}\cong K_{\mathbb F_\ell}^{-1}\), every effective curve in \(D_\infty\) meets \(D_\infty\) positively.  View \(P_n(Y_\ell,\beta)\) as an open subspace of the projective stable pair space of the compactification, in the pushed forward curve class, noting that that class has intersection \(0\) with \(D_\infty\).  An irreducible component of the support not contained in \(D_\infty\) meets \(D_\infty\) nonnegatively, and positively if at all; a component contained in \(D_\infty\) meets it positively by the ampleness of \(K_{\mathbb F_\ell}^{-1}\).  Since the total intersection is \(0\), neither occurs.  The open subspace is therefore closed.

For \(\ell=2\), the section \(B\) has \(B^2=-2\) and anticanonical degree \(0\), and the argument does not prove properness of the full moduli space.  The first paragraph still gives properness of the fixed locus.
\end{proof}

\subsection{Chart weights}
In Arbesfeld's convention, the ordered tangent weights on the four charts are
\begin{equation}
\begin{array}{c|c|c}
 & (w_1,w_2,w_3)&\text{exponents under }\xi\mapsto a^R\xi,
                                      \ \eta\mapsto a^S\eta\\ \hline
 P_{00}&(\eta,\kappa \xi^{-1}\eta^{-1},\xi)&(S,-R-S,R)\\
 P_{01}&(\kappa \xi^{-1}\eta,\eta^{-1},\xi)&(S-R,-S,R)\\
 P_{10}&(\kappa \xi^{1-\ell}\eta^{-1},\xi^\ell \eta,\xi^{-1})
        &((1-\ell)R-S,\ell R+S,-R)\\
 P_{11}&(\xi^{-\ell}\eta^{-1},\kappa \xi^{\ell+1}\eta,\xi^{-1})
        &(-\ell R-S,(\ell+1)R+S,-R).
\end{array}
\label{eq:Hirzebruch-four-chart-weights}
\end{equation}
The first two charts satisfy \(r_1\gg r_3>0\gg r_2\), and the last two satisfy \(r_2\gg0>r_3\gg r_1\).

The compact edges, oriented from the first listed fixed point to the second, have the following tangent character, ordered normal characters at the first fixed point, and normal bundle degrees:
\begin{equation}
\begin{array}{c|c|c|c}
\text{curve}&\text{fixed points}&h&(u,v);\ (a,b)\\ \hline
F_0&P_{00},P_{01}&\eta&(\kappa \xi^{-1}\eta^{-1},\xi);\ (-2,0)\\
F_\infty&P_{10},P_{11}&\xi^\ell \eta
 &(\xi^{-1},\kappa \xi^{1-\ell}\eta^{-1});\ (0,-2)\\
B&P_{10},P_{00}&\xi^{-1}
 &(\kappa \xi^{1-\ell}\eta^{-1},\xi^\ell \eta);\ (\ell-2,-\ell)\\
B_+&P_{11},P_{01}&\xi^{-1}
 &(\xi^{-\ell}\eta^{-1},\kappa \xi^{\ell+1}\eta);\ (\ell,-\ell-2).
\end{array}
\label{eq:Hirzebruch-four-edge-data}
\end{equation}
In the last column, \((a,b)\) means \(N_{C/Y_\ell}\cong \cO_{\mathbb P^1}(a)\oplus\cO_{\mathbb P^1}(b)\) in the displayed order.  The degrees follow from the self-intersections and adjunction:
\begin{equation}
\begin{gathered}
 N_{F/\mathbb F_\ell}\cong\cO,\qquad
 K_{\mathbb F_\ell}|_F\cong\cO(-2),\\
 N_{B/\mathbb F_\ell}\cong\cO(-\ell),\qquad
 K_{\mathbb F_\ell}|_B\cong\cO(\ell-2),\\
 N_{B_+/\mathbb F_\ell}\cong\cO(\ell),\qquad
 K_{\mathbb F_\ell}|_{B_+}\cong\cO(-\ell-2).
\end{gathered}
\label{eq:Hirzebruch-normal-bundle-degrees}
\end{equation}

\subsection{Strip functions}
For a partition \(\tau\), let
\[
 \mathscr D_\tau(t,q)
 =q^{\|\tau\|^2/2}t^{\|\tau^t\|^2/2}
  \widetilde Z_\tau(t,q)\widetilde Z_{\tau^t}(q,t).
\]
Taki's rank \(2\) strip function is
\begin{equation}
 K_{\alpha\beta}(X;t,q)
 =\sum_\gamma(-X)^{|\gamma|}f_\gamma(t,q)
 C_{\varnothing,\gamma,\alpha}(t,q)
 C_{\gamma^t,\varnothing,\beta}(t,q).
\label{eq:Taki-strip-function}
\end{equation}
For partitions \(\lambda,\mu\), let
\begin{equation}
\begin{aligned}
 \mathscr R_{\lambda\mu}(Q_F;t,q)
 =\prod_{i,j\geq1}
 &\frac{1-Q_Ft^jq^{i-1}}
       {1-Q_Ft^{-\lambda_i^t+j}q^{-\mu_j+i-1}}\\
 &\times
 \frac{1-Q_Ft^{j-1}q^i}
       {1-Q_Ft^{-\lambda_i^t+j-1}q^{-\mu_j+i}}.
\end{aligned}
\label{eq:Hirzebruch-direct-cross-product}
\end{equation}
For fixed \(\lambda,\mu\), the row and column tails cancel, and \eqref{eq:Hirzebruch-direct-cross-product} is a rational function of \(Q_F,t,q\).

\subsection{The four chart localization}
\begin{lemma}
\label{lem:Hirzebruch-four-chart-normalization}
For \(\ell=0,1\) use the global square root \eqref{eq:PT-symmetrized-sheaf}; for \(\ell=2\) use the vertex--edge series \eqref{eq:KPT-vertex-edge-sum}.  With the chart weights \eqref{eq:Hirzebruch-four-chart-weights} and the edge characters \eqref{eq:general-edge-character}, the coefficientwise preferred slope limit of the vertex--edge localization sum, divided by the base degree \(0\) contribution, is
\begin{equation}
\begin{aligned}
 \sum_{\lambda,\mu}
 &(-Q_B)^{|\lambda|+|\mu|}Q_F^{\ell|\lambda|}
 f_\lambda(t,q)^{-\ell-1}f_\mu(t,q)^{1-\ell}
 \mathscr D_\lambda(t,q)\mathscr D_\mu(t,q)
 \mathscr R_{\lambda\mu}(Q_F;t,q).
\end{aligned}
\label{eq:Hirzebruch-four-chart-result}
\end{equation}
\end{lemma}

\begin{proof}
Theorem~\ref{thm:KLT-vertex} identifies each normalized stable pair vertex with the normalized Donaldson--Thomas vertex.  Multiplying over the four charts and cancelling the four empty Donaldson--Thomas vertices against the degree \(0\) Donaldson--Thomas factor gives, coefficient by coefficient, the stable pair/Donaldson--Thomas identity after division by the degree \(0\) Donaldson--Thomas series, the global vertex--edge normalization being \cite[equation~(4.17)]{Arbesfeld}, with compact edge characters as in \cite[equation~(4.8)]{Arbesfeld}.  The chart weights \eqref{eq:Hirzebruch-four-chart-weights} satisfy the hypotheses of Lemma~\ref{lem:global-preferred-slope-limit}, so for each coefficient we may replace every normalized vertex and every symmetrized edge factor by its preferred slope limit before taking products and the finite sum over edge partitions.

\emph{Vertices.}  Put partitions \(\alpha\) and \(\beta\) on \(F_0\) and \(F_\infty\), and \(\mu_2\) and \(\mu_1\) on \(B\) and \(B_+\).  With the three entries ordered as the chart coordinates in \eqref{eq:Hirzebruch-four-chart-weights}, Proposition~\ref{prop:chart-slope-limit} gives the four vertex functions
\begin{equation}
\begin{gathered}
 \mathsf C^{\mathrm A}_{\alpha,\varnothing,\mu_2^t}(t,q),\qquad
 \mathsf C^{\mathrm A}_{\varnothing,\alpha^t,\mu_1^t}(t,q),\\
 \mathsf C^{\mathrm A}_{\varnothing,\beta,\mu_2}(q,t),\qquad
 \mathsf C^{\mathrm A}_{\beta^t,\varnothing,\mu_1}(q,t).
\end{gathered}
\label{eq:Hirzebruch-four-Arbesfeld-vertices}
\end{equation}
By \eqref{eq:Arbesfeld-IKV-prefactor}, the four refined vertices are
\begin{equation}
\begin{gathered}
 C_{\alpha^t,\varnothing,\mu_2}(t,q),\qquad
 C_{\varnothing,\alpha,\mu_1}(t,q),\\
 C_{\varnothing,\beta^t,\mu_2^t}(q,t),\qquad
 C_{\beta,\varnothing,\mu_1^t}(q,t).
\end{gathered}
\label{eq:Hirzebruch-four-IKV-vertices}
\end{equation}
Their product is the four vertex product in \(K_{\mu_1\mu_2}(\mathord\cdot;t,q) K_{\mu_2^t\mu_1^t}(\mathord\cdot;q,t)\), with \(K\) as in \eqref{eq:Taki-strip-function}; compare \cite[equations~(3.2)--(3.4)]{TakiRefined}.  The product of the four monomials \eqref{eq:Arbesfeld-IKV-monomial} is
\begin{equation}
 q^{-\frac12(\|\alpha\|^2+\|\alpha^t\|^2)}
 t^{-\frac12(\|\beta\|^2+\|\beta^t\|^2)}
 \prod_{i=1}^2q^{-\|\mu_i\|^2/2}t^{-\|\mu_i^t\|^2/2}.
\label{eq:Hirzebruch-four-vertex-monomials}
\end{equation}

\emph{Edges.}  For an edge partition \(\tau\), use \(\chi(\tau;(a,b))\) of \eqref{eq:edge-partition-Euler-characteristic}.  Substituting \eqref{eq:Hirzebruch-four-edge-data} in \eqref{eq:edge-preferred-limit-monomial} gives
\begin{equation}
\begin{array}{c|c}
F_0&
t^{(\|\alpha^t\|^2+|\alpha|)/2}
q^{(\|\alpha^t\|^2-|\alpha|)/2}\\[2pt]
F_\infty&
t^{(\|\beta\|^2-|\beta|)/2}
q^{(\|\beta\|^2+|\beta|)/2}\\[2pt]
B&
t^{(2-\ell)\|\mu_2^t\|^2/2}
q^{\ell\|\mu_2\|^2/2}\\[2pt]
B_+&
t^{-\ell\|\mu_1^t\|^2/2}
q^{(\ell+2)\|\mu_1\|^2/2}.
\end{array}
\label{eq:Hirzebruch-four-edge-monomials}
\end{equation}
The four slope indices can be computed from one finite polarization.  For a box of an edge partition, write \(L\) for its leg and \(A\) for its arm, and put
\[
 X=u^{-L}v^{A+1},\qquad Y=u^{L+1}v^{-A},\qquad
 d=aL-b(A+1).
\]
Since \(a+b=-2\) and \(huv=\kappa\), its contribution to \eqref{eq:general-edge-character} is
\[
 X\frac{h^d-h}{1-h}+Y\frac{h^{2-d}-h}{1-h}.
\]
Thus \(X(h^d-h)/(1-h)\) is a polarization.  More explicitly,
\[
 \frac{h^d-h}{1-h}=
 \begin{cases}
 \displaystyle\sum_{k=d}^{0}h^k,&d\leq0,\\[3pt]
 0,&d=1,\\[2pt]
 \displaystyle-\sum_{k=1}^{d-1}h^k,&d\geq2.
 \end{cases}
\]
The minus sign in the last case is a virtual multiplicity in \eqref{eq:slope-index-definition}.

For \(F_0\), \(d=-2L\), and the exponent of \(Xh^k\) is
\[
 (L+A+1)R+(L+k)S,\qquad -2L\leq k\leq0.
\]
There are \(L+1\) attracting weights and \(L\) repelling weights, so the index per box is \(1\).
For \(F_\infty\), \(d=2(A+1)\), and the polarization has negative multiplicity on \(1\leq k\leq2A+1\).  Its weight exponent is
\[
 \bigl(L+(A+1)(1-\ell)+k\ell\bigr)R+(k-A-1)S.
\]
The exponent at \(k=A+1\) is \((L+A+1)R>0\).  Hence \(A+1\) weights are attracting and \(A\) are repelling; their negative multiplicities give index \(-1\).

For \(B\), \(d=(\ell-2)L+\ell(A+1)\), and the exponent of \(Xh^k\) is
\[
 \bigl((\ell-1)L+\ell(A+1)-k\bigr)R+(L+A+1)S.
\]
For \(B_+\), \(d=\ell L+(\ell+2)(A+1)\), and it is
\[
 \bigl(\ell L+(\ell+1)(A+1)-k\bigr)R+(L+A+1)S.
\]
All weights in these two finite polarizations are attracting for the stabilized slope \(S\gg R>0\).  The signed number of weights is \(1-d\), giving respectively
\[
 (2-\ell)L-\ell A+1-\ell,\qquad
 -(\ell+2)A-\ell L-\ell-1.
\]
Summing over boxes, with \(\|\tau\|^2=\sum_{s\in\tau}(2a_\tau(s)+1)\) and \(\|\tau^t\|^2=\sum_{s\in\tau}(2\ell_\tau(s)+1)\), we find
\begin{equation}
\begin{aligned}
 \operatorname{ind}_\sigma
 \mathsf E^{\mathrm{vir}}_{\alpha;(-2,0)}
 \left(\eta,\kappa\xi^{-1}\eta^{-1},\xi\right)
 &=|\alpha|,\\
 \operatorname{ind}_\sigma
 \mathsf E^{\mathrm{vir}}_{\beta;(0,-2)}
 \left(\xi^\ell\eta,\xi^{-1},
       \kappa\xi^{1-\ell}\eta^{-1}\right)
 &=-|\beta|,\\
 \operatorname{ind}_\sigma
 \mathsf E^{\mathrm{vir}}_{\mu_2;(\ell-2,-\ell)}
 \left(\xi^{-1},\kappa\xi^{1-\ell}\eta^{-1},\xi^\ell\eta\right)
 &=\frac{(2-\ell)\|\mu_2^t\|^2
              -\ell\|\mu_2\|^2}{2},\\
 \operatorname{ind}_\sigma
 \mathsf E^{\mathrm{vir}}_{\mu_1;(\ell,-\ell-2)}
 \left(\xi^{-1},\xi^{-\ell}\eta^{-1},
       \kappa\xi^{\ell+1}\eta\right)
 &=-\frac{(\ell+2)\|\mu_1\|^2
               +\ell\|\mu_1^t\|^2}{2}.
\end{aligned}
\label{eq:Hirzebruch-four-edge-indices}
\end{equation}

Let \(\bar f_\tau(t,q)=t^{\|\tau^t\|^2/2}q^{-\|\tau\|^2/2}\), so \(\bar f_\tau=(-1)^{|\tau|}f_\tau\).  Multiplying \eqref{eq:Hirzebruch-four-edge-monomials} by \eqref{eq:Hirzebruch-four-vertex-monomials} gives, edge by edge,
\begin{equation}
 \left(\frac tq\right)^{|\alpha|/2}\bar f_\alpha(t,q),\quad
 \left(\frac qt\right)^{|\beta|/2}\bar f_{\beta^t}(q,t),\quad
 \bar f_{\mu_2}(t,q)^{1-\ell},\quad
 \bar f_{\mu_1}(t,q)^{-\ell-1}.
\label{eq:Hirzebruch-combined-edge-monomials}
\end{equation}

\emph{Fibre sums.}  For either fibre edge, the factor \(\mathbf Q_e^{|\lambda_e|}\) in \eqref{eq:KPT-vertex-edge-sum} is \(Q_F^{|\lambda_e|}\) by \eqref{eq:Hirzebruch-signed-series}.  The first two factors in \eqref{eq:Hirzebruch-combined-edge-monomials} turn the sums over \(\alpha\) and \(\beta\) into
\begin{equation}
 K_{\mu_1\mu_2}\!
 \left(\sqrt{t/q}\,Q_F;t,q\right),\qquad
 K_{\mu_2^t\mu_1^t}\!
 \left(\sqrt{q/t}\,Q_F;q,t\right),
\label{eq:Hirzebruch-two-oriented-strip-arguments}
\end{equation}
compare \cite[equation~(3.2)]{TakiRefined}.  The strip evaluation \cite[equation~(3.3)]{TakiRefined} uses the argument \(\sqrt{q/t}\,X\) for \(K(X;t,q)\) and \(\sqrt{t/q}\,X\) for \(K(X;q,t)\), so both functions in \eqref{eq:Hirzebruch-two-oriented-strip-arguments} have the resulting argument \(Q_F\), whence the two Cauchy identities give \(\mathscr D_{\mu_1}\mathscr D_{\mu_2} \mathscr R_{\mu_1\mu_2}(Q_F;t,q)\), exactly as in \cite[equations~(3.10)--(3.11)]{TakiRefined}.

\emph{Signs.}  Let \(D=|\mu_1|+|\mu_2|\).  The signed specialization \eqref{eq:intro-signed-Novikov-specialization} contributes
\[
 \bigl((-1)^\ell Q_B\bigr)^D Q_F^{\ell|\mu_1|}.
\]
Since \(\bar f_\tau=(-1)^{|\tau|}f_\tau\) and \(-\ell-1\equiv1-\ell\equiv\ell+1\pmod2\),
\begin{equation}
\begin{aligned}
 &\bigl((-1)^\ell Q_B\bigr)^D
  \bar f_{\mu_1}^{-\ell-1}\bar f_{\mu_2}^{1-\ell}\\
 &\qquad=(-Q_B)^D
  f_{\mu_1}^{-\ell-1}f_{\mu_2}^{1-\ell}.
\end{aligned}
\label{eq:Hirzebruch-section-sign-conversion}
\end{equation}
Relabeling \((\mu_1,\mu_2)=(\lambda,\mu)\) and applying Lemma~\ref{lem:global-preferred-slope-limit} gives \eqref{eq:Hirzebruch-four-chart-result}.
\end{proof}

\subsection{The fixed point formula}
For a partition \(\lambda\), let
\begin{equation}
 T_\lambda(U;p_1,p_2)
 =\prod_{(i,j)\in\lambda}U^{-1}p_1^{1-i}p_2^{1-j}.
\label{eq:Hirzebruch-determinant-box-monomial}
\end{equation}

\begin{lemma}
\label{lem:Hirzebruch-ADHM-normalization}
At the fixed point \((\lambda_1,\lambda_2)\), the summand of \(\mathcal A_\ell\) is
\begin{equation}
 \frac{
 \displaystyle
 \prod_{a=1}^2
 \left((p_1p_2)^{-\lvert\lambda_a\rvert/2}
 T_{\lambda_a}(U_a;p_1,p_2)\right)^\ell
 (p_1^{-1}p_2^{-1}z)^{\lvert\lambda_1\rvert+\lvert\lambda_2\rvert}}
 {\displaystyle
 \prod_{a,b=1}^2
 N_{\lambda_a\lambda_b}
 (U_b/U_a;p_1^{-1},p_2^{-1})},
\label{eq:Hirzebruch-Nekrasov-fixed-term}
\end{equation}
where \(N_{\lambda\mu}\) is the finite arm--leg product \eqref{eq:Nekrasov-factor}.
\end{lemma}

The exponent \(\ell\) is exactly the power of the determinant line.  The insertion \((\det\cV)^\ell\) is the whole shape dependent numerator apart from the monomial recording \(c_2\).  

\begin{proof}
Substitute \(q_i=p_i^{-1}\) and \(u_a=U_a^{-1}\) in \eqref{eq:det-fixed-weight}:
\[
 \det\cV
 =\prod_{a=1}^2
 U_a^{-\lvert\lambda_a\rvert}
 p_1^{-n(\lambda_a)}p_2^{-n(\lambda_a^t)}
 =\prod_{a=1}^2T_{\lambda_a}(U_a;p_1,p_2).
\]
The factor \((p_1p_2)^{-\ell(\lvert\lambda_1\rvert+ \lvert\lambda_2\rvert)/2}\) comes from the substitution for the counting variable in \eqref{eq:Hirzebruch-ADHM-series}.  The same substitution in \eqref{eq:ADHM-Nekrasov-denominator} gives the denominator, since \(u_a/u_b=U_b/U_a\).
\end{proof}

Under \(U=e^{T_F}\), \(z=\zeta_\ell\) and \(p_i=e^{\eps_i}\), the branch \eqref{eq:Hirzebruch-signed-branch} identifies the exponential in \eqref{eq:Hirzebruch-full-PT} with \eqref{eq:Hirzebruch-degree-zero-exponential}.

\subsection{Comparison}
\begin{proposition}
\label{prop:Hirzebruch-PT-ADHM}
For \(0\leq\ell\leq2\),
\begin{equation}
 Z_{Y_\ell}^{\PT,\mathrm{norm}}(Q_B,Q_F;p_1,p_2)
 =\mathcal A_\ell
 (Q_F^{-1};p_1,p_2\mid\zeta_\ell).
\label{eq:Hirzebruch-normalized-PT-ADHM}
\end{equation}
Consequently,
\begin{equation}
 \widehat Z_{Y_\ell}(T_F,T_M;\eps_1,\eps_2)
 =\widehat{\mathcal A}_\ell
 \bigl(e^{T_F};e^{\eps_1},e^{\eps_2}\mid\zeta_\ell\bigr).
\label{eq:Hirzebruch-PT-ADHM}
\end{equation}
\end{proposition}

Equation~\eqref{eq:Hirzebruch-normalized-PT-ADHM} holds coefficientwise in \(Q_B\), with coefficients in \(\Q(Q_F^{1/2},p_1^{1/2},p_2^{1/2})\), each coefficient being mapped to its Laurent expansion at \(Q_F^{1/2}=0\); the expressions of Lemma~\ref{lem:Hirzebruch-four-chart-normalization} and the framed sheaf side use the same injective expansion homomorphism.  For \(\ell=2\), the equality is an equality of localized indices in the sense of Lemma~\ref{lem:Hirzebruch-fixed-proper}.  Multiplying both sides by the same meromorphic continuations of the classical and one-loop factors gives \eqref{eq:Hirzebruch-PT-ADHM}.

\begin{proof}
Lemma~\ref{lem:Hirzebruch-four-chart-normalization} gives the localization sum in the preferred slope limit with all prefactors and edge characters included; the passage from chart limits to the global sum is Lemma~\ref{lem:global-preferred-slope-limit}, using the finiteness of Lemma~\ref{lem:Hirzebruch-fixed-proper}.  For a fixed power of \(Q_B\) only finitely many pairs of partitions occur.

The rest is Taki's calculation, in rank \(2\) and in our variables.  In \cite[Section~3]{TakiRefined}, the framing parameters satisfy \(e_1e_2=1\) and the ratio \(e_1/e_2\) is our \(Q_F\), so we take \(e_1=Q_F^{1/2}\), \(e_2=Q_F^{-1/2}\) on the fixed square root branch, and the two arguments before Taki's Cauchy evaluation are the oriented arguments \eqref{eq:Hirzebruch-two-oriented-strip-arguments}; after the two square root substitutions in the definition of \(\widetilde Q_F\) in \cite[equation~(3.3)]{TakiRefined}, both give \(Q_F\).  His base Novikov coordinate becomes \(-Q_B\) in the stable pair sign convention.  The rank \(2\) case of \cite[equations~(3.3)--(3.5) and (3.12)--(3.19) in the arXiv version]{TakiRefined}, with \(m=\ell\), is then exactly \eqref{eq:Hirzebruch-four-chart-result}, and division by \eqref{eq:Hirzebruch-empty-strip-normalization} removes the base degree \(0\) contribution; in Taki's notation, the \(B_+\)-edge carries \(Q_{B,1}=-Q_BQ_F^\ell\) and the \(B\)-edge carries \(Q_{B,2}=-Q_B\).  Equivalently, if \(d=|\lambda|+|\mu|\), then
\begin{equation}
 Q_B^dQ_F^{\ell|\lambda|}
 =\bigl(Q_BQ_F^{\ell/2-1}\bigr)^dQ_F^d
   e_1^{\ell|\lambda|}e_2^{\ell|\mu|},
\label{eq:Hirzebruch-Taki-curve-translation}
\end{equation}
the \(N=2\) case of the curve monomial identities \cite[equations~(3.15)--(3.18) in the arXiv version]{TakiRefined}.

For the partition dependent factors, abbreviate \(N_{\alpha\beta}(x)=N_{\alpha\beta}(x;t^{-1},q)\).  The diagonal and cross term identities \cite[equations~(3.12)--(3.14) in the arXiv version]{TakiRefined}, in our arm--leg convention, are
\begin{equation}
 \frac1{N_{\tau\tau}(1)}
 =(-1)^{|\tau|}\left(\frac tq\right)^{|\tau|/2}
  \mathscr D_\tau(t,q)
\label{eq:Hirzebruch-direct-diagonal-factor}
\end{equation}
and
\begin{equation}
\begin{aligned}
 &\frac1{N_{\lambda\mu}(Q_F)N_{\mu\lambda}(Q_F^{-1})}\\
 &\quad=(-Q_F)^d
 \left(\frac qt\right)^{-d/2+
 (\|\lambda^t\|^2-\|\mu^t\|^2)/2}
 q^{(c_\lambda-c_\mu)/2}
 \mathscr R_{\lambda\mu}(Q_F;t,q).
\end{aligned}
\label{eq:Hirzebruch-direct-cross-factor}
\end{equation}
Using \(\|\tau\|^2=|\tau|+2n(\tau^t)\) and \(\|\tau^t\|^2=|\tau|+2n(\tau)\), the product of \eqref{eq:Hirzebruch-direct-diagonal-factor} and \eqref{eq:Hirzebruch-direct-cross-factor} gives, term by term,
\begin{equation}
\begin{aligned}
 &(-Q_B)^dQ_F^{\ell|\lambda|}
 f_\lambda^{-\ell-1}f_\mu^{1-\ell}
 \mathscr D_\lambda\mathscr D_\mu\mathscr R_{\lambda\mu}\\
 &\quad=
 \frac{\left((-1)^\ell(q/t)^{1+\ell/2}
 Q_BQ_F^{-1}\right)^d
 \left(
 Q_F^{|\lambda|}t^{-n(\lambda)-n(\mu)}
 q^{n(\lambda^t)+n(\mu^t)}
 \right)^\ell}
 {N_{\lambda\lambda}(1)N_{\lambda\mu}(Q_F)
  N_{\mu\lambda}(Q_F^{-1})N_{\mu\mu}(1)}.
\end{aligned}
\label{eq:Hirzebruch-direct-termwise-comparison}
\end{equation}
The sign is also termwise:
\begin{equation}
 (-1)^d
 (-1)^{-(\ell+1)|\lambda|+(1-\ell)|\mu|}
 =(-1)^{\ell d}.
\label{eq:Hirzebruch-parity-conversion}
\end{equation}

Summing \eqref{eq:Hirzebruch-direct-termwise-comparison} over \(\lambda,\mu\) proves
\begin{equation}
 Z_{Y_\ell}^{\PT,\mathrm{norm}}(Q_B,Q_F;t,q^{-1})
 =Z_{2,\ell}^{\mathrm{fr}}\left(
 (-1)^\ell\left(\frac qt\right)^{1+\ell/2}
 Q_BQ_F^{-1};t^{-1},q;Q_F,1\right).
\label{eq:Hirzebruch-uncentered-comparison}
\end{equation}
Rescaling the framing characters from \((Q_F,1)\) to \((Q_F^{1/2},Q_F^{-1/2})\) by Lemma~\ref{lem:linearization} gives \eqref{eq:Hirzebruch-normalized-PT-ADHM} with \((p_1,p_2)=(t,q^{-1})\).  Since \(t\) and \(q\) are independent, the identity holds for independent \(p_1,p_2\).  Finally, multiply both sides by \(\mathcal A^{(0)}(U;p_1,p_2\mid\zeta_\ell) \mathcal E(U;p_1,p_2)\); by \eqref{eq:Hirzebruch-signed-branch}, the first factor is the degree \(0\) exponential of \eqref{eq:Hirzebruch-full-PT}.
\end{proof}

\section{Blowup equations for local Hirzebruch surfaces}
\label{sec:NY}

The blowup identities of Nakajima--Yoshioka, Bershtein--Shchechkin and Shchechkin, written in the variables of Proposition~\ref{prop:Hirzebruch-PT-ADHM}, give Theorem~\ref{thm:Hirzebruch-main} once the translations of \(T_F\) and \(T_M\) are computed.

\subsection{The bilinear sum}
The determinant line insertions belong to the \(K\)-theoretic Donaldson theory of G\"ottsche--Nakajima--Yoshioka \cite{GottscheNakajimaYoshioka}.  We use the rank \(2\) formulas in the variables fixed above.

Let \(p_i=e^{\eps_i}\), and let \(\widehat{\mathcal A}_\ell\) be the full framed sheaf partition function \eqref{eq:Hirzebruch-completed-ADHM-series}.  Recall the index sets \(\mathcal I_\ell\) of \eqref{eq:intro-index-sets}: \(j\in\{0,1\}\) selects the lattice coset \(\nu\in\Z+j/2\), and \(d\) shifts the counting variable in the two chart factors.  For \((j,d)\in\mathcal I_\ell\), let
\begin{equation}
\begin{aligned}
 \mathscr B_{\ell}^{j,d}(U;p_1,p_2\mid z)
 :={}&\sum_{\nu\in\Z+j/2}
 \widehat{\mathcal A}_\ell
 \left(Up_1^{2\nu};p_1,p_2/p_1
       \mid p_1^{d-1+\ell(j-1)/2}z\right)\\
 &\hspace{13mm}\times
 \widehat{\mathcal A}_\ell
 \left(Up_2^{2\nu};p_1/p_2,p_2
       \mid p_2^{d-1+\ell(j-1)/2}z\right).
\end{aligned}
\label{eq:finite-NY-bilinear-operator}
\end{equation}
This notation is interpreted coefficientwise after division by the unshifted degree \(0\) factor, as made explicit below.  To interpret it as an analytic sum also requires normal convergence of the normalized summands and compatible continuations of the degree \(0\) factors.  Let
\[
 \K_{\mathrm{fr}}
 =\Q(U^{1/2},p_1^{1/4},p_2^{1/4}).
\]
Specializations are allowed wherever the fixed point denominators \eqref{eq:Nekrasov-factor} occurring in \eqref{eq:ADHM-Nekrasov-denominator} are regular.

\subsection{Shift quotients}
\begin{lemma}
\label{lem:Hirzebruch-shift-quotients}
Let \(r\in\Z\), \(s\in\tfrac12\Z\), and keep the chosen logarithms.  The quotient of the two shifted classical factors by the unshifted one is
\begin{equation}
\begin{aligned}
 &\frac{\mathcal A^{(0)}(Up_1^r;p_1,p_2/p_1\mid p_1^s z)
             \mathcal A^{(0)}(Up_2^r;p_1/p_2,p_2\mid p_2^s z)}
 {\mathcal A^{(0)}(U;p_1,p_2\mid z)}\\
 &\hspace{18mm}=
 z^{r^2/4}U^{r(s+1)/2}(p_1p_2)^{r^2s/4}.
\end{aligned}
\label{eq:Hirzebruch-classical-shift-quotient}
\end{equation}
For \(r\in\Z\), let
\begin{equation}
 D_r(x,y)=
 \frac{x^r}{(1-x)(1-y/x)}
 +\frac{y^r}{(1-x/y)(1-y)}
 -\frac1{(1-x)(1-y)}.
\label{eq:Hirzebruch-Dr-definition}
\end{equation}
Then
\begin{equation}
 D_r(x,y)=
 \begin{cases}
 -\displaystyle\sum_{\substack{a,b\geq0\\a+b\leq r-1}}x^ay^b,
     &r\geq1,\\[6pt]
 0,&r=0,-1,\\[2pt]
 -\displaystyle\sum_{\substack{a,b\geq0\\a+b\leq-r-2}}
     x^{-a-1}y^{-b-1},&r\leq-2.
 \end{cases}
\label{eq:Hirzebruch-Dr-Laurent-polynomial}
\end{equation}
The one-loop quotient
\begin{equation}
 \mathcal L_r(U;p_1,p_2)
 :=\frac{\mathcal E(Up_1^r;p_1,p_2/p_1)
             \mathcal E(Up_2^r;p_1/p_2,p_2)}
 {\mathcal E(U;p_1,p_2)}
\label{eq:Hirzebruch-one-loop-shift-quotient}
\end{equation}
is a rational function of \(U,p_1,p_2\), with
\begin{equation}
 \log\mathcal L_r
 =-\sum_{m\geq1}\frac{
 U^mD_r(p_1^m,p_2^m)+U^{-m}D_{-r}(p_1^m,p_2^m)}m.
\label{eq:Hirzebruch-one-loop-shift-log}
\end{equation}
\end{lemma}

Equation~\eqref{eq:Hirzebruch-one-loop-shift-log} followed by \eqref{eq:Hirzebruch-Dr-Laurent-polynomial} gives the rational continuation of the quotient even when the three double products in \eqref{eq:Hirzebruch-one-loop-shift-quotient} have no common chamber of convergence.

\begin{proof}
Write \(L=\log U\), \(a_i=\log p_i\) and \(Z=\log z\).  The logarithm of the left side of \eqref{eq:Hirzebruch-classical-shift-quotient}, put over the common denominator \(4a_1a_2(a_2-a_1)\), reduces to
\[
 \frac{r^2}{4}Z+\frac{r(s+1)}2L+\frac{r^2s}{4}(a_1+a_2).
\]
Taking plethystic logarithms of the three double Pochhammer factors in \eqref{eq:Hirzebruch-one-loop-shift-quotient} gives \eqref{eq:Hirzebruch-one-loop-shift-log}.

For \(r\geq1\), multiply \(D_r(x,y)\) by \((1-x)(1-y)(x-y)\) and sum the two finite geometric progressions in \eqref{eq:Hirzebruch-Dr-Laurent-polynomial}; the numerators agree, and replacing \((r,x,y)\) by \((-r-1,x^{-1},y^{-1})\) gives the case \(r\leq-2\), by the identity
\[
 D_{-r-1}(x^{-1},y^{-1})=xyD_r(x,y).
\]
Direct substitution gives \(D_0=D_{-1}=0\).  Exponentiating \eqref{eq:Hirzebruch-one-loop-shift-log} now gives a finite product of factors \((1-Up_1^ap_2^b)^{-1}\) and \((1-U^{-1}p_1^ap_2^b)^{-1}\) with integral \(a,b\).
\end{proof}

\subsection{The normalized bilinear sum}
For \((j,d)\in\mathcal I_\ell\), let
\[
 s_{\ell;j,d}=d-1+\frac\ell2(j-1),
\]
and, with \(\mathcal L_r\) as in \eqref{eq:Hirzebruch-one-loop-shift-quotient},
\begin{equation}
\begin{aligned}
 \overline{\mathscr B}_{\ell}^{j,d}(U;p_1,p_2\mid z)
 :={}&\sum_{\nu\in\Z+j/2}
 z^{\nu^2}U^{\nu(s_{\ell;j,d}+1)}
 (p_1p_2)^{\nu^2s_{\ell;j,d}}
 \mathcal L_{2\nu}(U;p_1,p_2)\\
 &\quad\times
 \mathcal A_\ell\left(
 Up_1^{2\nu};p_1,p_2/p_1
 \mid p_1^{s_{\ell;j,d}}z\right)\\
 &\quad\times
 \mathcal A_\ell\left(
 Up_2^{2\nu};p_1/p_2,p_2
 \mid p_2^{s_{\ell;j,d}}z\right).
\end{aligned}
\label{eq:normalized-finite-NY-bilinear-operator}
\end{equation}
Coefficientwise after cancellation of the degree \(0\) factors, the two shift quotient formulas give
\[
 \overline{\mathscr B}_{\ell}^{j,d}
 =\bigl(\mathcal A^{(0)}(U;p_1,p_2\mid z)
          \mathcal E(U;p_1,p_2)\bigr)^{-1}
   \mathscr B_{\ell}^{j,d}.
\]

In a summand of \eqref{eq:normalized-finite-NY-bilinear-operator} we have \(r=2\nu\) and \(s=d-1+\frac\ell2(j-1)\).  The framed sheaf sums contain only nonnegative integral powers of the shifted \(z\)-variables.  By \eqref{eq:Hirzebruch-classical-shift-quotient} the lowest power in the summand is \(z^{r^2/4}=z^{\nu^2}\), and by \eqref{eq:Hirzebruch-one-loop-shift-quotient} the remaining one-loop factor is rational.  Every coefficient therefore receives contributions from finitely many \(\nu\), and
\[
 \overline{\mathscr B}_{\ell}^{0,d}
 \in\K_{\mathrm{fr}}[[z]],
 \qquad
 \overline{\mathscr B}_{\ell}^{1,d}
 \in z^{1/4}\K_{\mathrm{fr}}[[z]].
\]

\subsection{A dense family for the degree \(0\) factors}
\begin{lemma}
\label{lem:Shchechkin-dense-domain}
Let \(s,V\in\C^*\), let \(m,n\) be positive integers, and let
\begin{equation}
 U^{1/2}=V,\qquad
 p_1^{1/4}=s^{-m},\qquad p_2^{1/4}=s^n,
 \qquad 0<\lvert s\rvert<1.
\label{eq:Shchechkin-dense-specialization}
\end{equation}
Exclude the finitely many ratios \(m/n\) for which a base occurring in \cite[equations~(3.2)--(3.27)]{Shchechkin} becomes \(1\), and exclude the poles of the factors there.  On every remaining specialization, all double Pochhammer factors in Shchechkin's convolution argument admit simultaneous meromorphic continuations, obtained by base inversion from products whose bases have modulus less than \(1\).  The set of specializations \eqref{eq:Shchechkin-dense-specialization}, subject to these exclusions, is Zariski dense in the torus with coordinates \(U^{1/2},p_1^{1/4},p_2^{1/4}\).
\end{lemma}

\begin{proof}
Every base in the cited formulas is a monomial in \(p_1,p_2\), hence a nonzero integral power of \(s\) under \eqref{eq:Shchechkin-dense-specialization}.  For a base \(a\) with \(\lvert a\rvert>1\) we use the inversion rule
\begin{equation}
 (x;a,b)_\infty
 =\left(xa^{-1};a^{-1},b\right)_\infty^{-1},
\label{eq:double-Pochhammer-base-inversion}
\end{equation}
which follows by comparing plethystic logarithms, and similarly for \(b\).  Once the cases with a base equal to \(1\) are excluded, every required factor is a product or a reciprocal of a product with both bases of modulus less than \(1\), and such products converge normally away from their polar divisors; this is the convergence family \(p_1=q_{\mathrm{Sh}}^{-m'},p_2=q_{\mathrm{Sh}}^{n'}\) of \cite[Section~2.1 and footnote~2]{Shchechkin}, with \(q_{\mathrm{Sh}}=s^4\), \(m'=m\), \(n'=n\).  On a simply connected open subset of the \((s,V)\)-domain, choose \(\log s\) and \(\log V\), and use the compatible logarithms
\[
 \log p_1=-4m\log s,\qquad
 \log p_2=4n\log s,\qquad
 \log U=2\log V.
\]

For density, let \(F(V,X,Y)\) be a nonzero Laurent polynomial,
\[
 F(V,X,Y)=\sum_{(a,b)}f_{a,b}(V)X^aY^b,
 \qquad f_{a,b}(V)\in\C[V^{\pm1}],
\]
with distinct pairs \((a,b)\).  Choose positive integers \(m,n\) outside the excluded set such that the integers \(-ma+nb\) are distinct for the finitely many \((a,b)\) with \(f_{a,b}\neq0\); this avoids only finitely many rational slopes, and \(F(V,s^{-m},s^n)\) is then a nonzero Laurent polynomial in \(V,s\), which cannot vanish for all \(V\in\C^*\) and all \(s\) in the punctured disc.  The zeros and poles excluded in the coefficientwise argument are proper analytic subsets of the open patches, with dense complements.  Deleting them does not affect the conclusion.
\end{proof}

\subsection{Shchechkin's identities in coefficientwise form}
\begin{lemma}
\label{lem:Shchechkin}
\textup{(1)} Let \(\mathcal Z_{\mathrm{Sh}}^{[\ell]}(U;p_1,p_2\mid z)\) be Shchechkin's full rank \(2\) five-dimensional Nekrasov function, the product of his classical factor, one-loop factor and framed sheaf sum.  Then
\begin{equation}
 \mathcal Z_{\mathrm{Sh}}^{[\ell]}(U;p_1,p_2\mid z)
 =\widehat{\mathcal A}_\ell(U;p_1,p_2\mid z).
\label{eq:Shchechkin-full-function-identification}
\end{equation}
After setting his compactification radius \(R\) to \(1\), the variables are related by
\begin{equation}
 u_{\mathrm{Sh}}=U,\quad
 (u_{1,\mathrm{Sh}},u_{2,\mathrm{Sh}})=(U^{1/2},U^{-1/2}),\quad
 (q_{1,\mathrm{Sh}},q_{2,\mathrm{Sh}})=(p_1,p_2),\quad
 z_{\mathrm{Sh}}=z,\quad l_{\mathrm{Sh}}=\ell.
\label{eq:Shchechkin-variable-map}
\end{equation}
If \(d\) is the integer in \eqref{eq:finite-NY-bilinear-operator}, then \(d_{\mathrm{Sh}}=d-1\); his lattice indices are \(j_{\mathrm{Sh}}=j\) and \(n_{\mathrm{Sh}}=\nu\).

\textup{(2)} Keep the chosen roots in \(\K_{\mathrm{fr}}\) and let \(w=z^{1/4}\).  After division by the common classical and one-loop factor, every identity used in Theorem~\ref{thm:finite-NY-levels} is an identity of coefficients in \(\K_{\mathrm{fr}}[[w]]\).  The identities \eqref{eq:determinant-power-common-relation} lie in \(\K_{\mathrm{fr}}[[z]]\) for \(j=0\) and in \(z^{1/4}\K_{\mathrm{fr}}[[z]]\) for \(j=1\).
\end{lemma}

\begin{proof}
Let \(N^{\mathrm{Sh}}\) be the finite product of Shchechkin's equation~(2.4).  Directly from the two arm--leg definitions,
\[
 N_{\lambda\mu}(Q;p_1^{-1},p_2^{-1})
 =N^{\mathrm{Sh}}_{\mu\lambda}(Q;p_1,p_2).
\]
Relabeling \((a,b)\) in the full rank \(2\) product,
\begin{equation}
 \prod_{a,b}N_{\lambda_a\lambda_b}
  (U_b/U_a;p_1^{-1},p_2^{-1})
 =\prod_{a,b}N^{\mathrm{Sh}}_{\lambda_a\lambda_b}
  (U_a/U_b;p_1,p_2).
\label{eq:Shchechkin-denominator-translation}
\end{equation}
The numerator of his equation~(2.3) is
\[
 (p_1^{-1}p_2^{-1}z)^{\sum_a|\lambda_a|}
 \prod_a(p_1p_2)^{-\ell|\lambda_a|/2}
 T_{\lambda_a}(U_a;p_1,p_2)^\ell,
\]
which is \eqref{eq:Hirzebruch-Nekrasov-fixed-term}: his equation~(2.3) is the fixed point formula of Lemma~\ref{lem:Hirzebruch-ADHM-normalization} after \eqref{eq:Shchechkin-variable-map}, his classical factor (2.15) is \(\mathcal A^{(0)}\) and his one-loop factor (2.16) is \(\mathcal E\), whence \eqref{eq:Shchechkin-full-function-identification} follows.  In the first chart of his equation~(2.21), the exponent of \(p_1\) multiplying \(z\) is
\[
 d_{\mathrm{Sh}}+\frac\ell2(j-1)
 =d-1+\frac\ell2(j-1),
\]
and the second chart has the analogous exponent of \(p_2\); since the framing shifts are \(Up_i^{2\nu}\), every entry of \eqref{eq:finite-NY-bilinear-operator} agrees with his equation~(2.21).  This proves (1).

For (2), every coefficient of \(\mathcal A_\ell\) is a finite sum over pairs of partitions and lies in \(\K_{\mathrm{fr}}\), and in a summand of \eqref{eq:normalized-finite-NY-bilinear-operator} indexed by \(\nu\in\Z+j/2\) the classical quotient contributes the lowest power \(z^{\nu^2}\), so only finitely many \(\nu\) and finitely many partitions contribute to a fixed power of \(w\).

We use Shchechkin's convolution calculations as formal series calculations; no convergence of an instanton series with nonzero determinant power is required.  The inputs are the coefficientwise Nakajima--Yoshioka formulas cited below and the extreme cases explained in \cite[footnote~6]{BershteinShchechkin}.

To justify reindexing the double sum in \cite[equations~(3.3)--(3.7)]{Shchechkin}, put \(a=\log p_1\), \(b=\log p_2\), and \(L=\log U\).  After dividing by the common unshifted classical factor, the exponent of \(z\) in its three classical factors is
\[
\begin{aligned}
 &\frac{(L+2m(a+b))^2}{8a(a+b)}
 -\frac{(L+2na+2mb)^2}{8a(b-a)}\\
 &\qquad-\frac{(L+2(m+n)b)^2}{8b(a-b)}
 +\frac{L^2}{8b(a+b)}
 =\frac{m^2+n^2}{2}.
\end{aligned}
\]
The counting-variable shifts multiply \(z\) by constants and do not change this valuation.  Thus a coefficient of degree at most \(N\) involves only \(m^2+n^2\leq2N\) and finitely many partitions.  The substitution \(m=m'+n'\), \(n=m'-n'\) changes the quadratic form into \(m'^2+n'^2\) and is legitimate coefficientwise, including its two parity classes.  The parity-dependent counting-variable shifts at determinant power \(1\) leave this argument unchanged.

The degree \(0\) coefficients in the intermediate convolutions can be read in a field of meromorphic germs on any generic patch supplied by Lemma~\ref{lem:Shchechkin-dense-domain}; adjoin the finitely many roots needed for \(z^{1/8}\) and the monomial coefficients.  Only the degree \(0\) products are evaluated on these patches.  All instanton sums remain formal.  The even convolution has a nonzero leading coefficient \(b_0\), so coefficient induction gives injectivity.  The two odd convolutions have a determinant equal to a nonzero product times
\[
 U^{(d_{12}-d'_{12})/4}
 -U^{(d'_{12}-d_{12})/4},
 \qquad d_{12}\ne d'_{12},
\]
where \(d_{12},d'_{12}\) are their total counting-variable shifts.  It is invertible for generic \(U\).  These coefficient inductions, and the parity-dependent one at determinant power \(1\), prove the relations used in Shchechkin's equations~(4.1)--(4.4) over the germ field.

The proof of \cite[Proposition~4.1]{Shchechkin} is coefficientwise as well.  At order \(z^k\), the two unknown coefficients satisfy the linear system of his equation~(4.7), with determinant
\[
 (p_1^{-k}-1)(p_2^{-k}-1)(p_2^{-k}-p_1^{-k}),
\]
a nonzero element of the rational function field.  The double Pochhammer expressions of his equation~(4.9) are ordinary formal power series in \(z\); their quotients follow from plethystic logarithms and \eqref{eq:Hirzebruch-Dr-Laurent-polynomial}.

By \eqref{eq:Hirzebruch-one-loop-shift-quotient}, \eqref{eq:Hirzebruch-classical-shift-quotient} and the finite fixed point formula, every coefficient of each normalized identity lies in \(\K_{\mathrm{fr}}\).  The difference of the two sides vanishes on the Zariski dense family of Lemma~\ref{lem:Shchechkin-dense-domain}, hence is the zero rational function.
\end{proof}

\subsection{The framed sheaf blowup identities}
For \(0\leq\ell\leq2\) and \((0,d)\in\mathcal I_\ell\), let
\begin{equation}
 c_{\ell;0,d}(p_1,p_2\mid z)=
 \begin{cases}
 1-(p_1p_2)^{-1}z,&(\ell,d)=(0,-1),\\[2pt]
 1-p_1p_2z,&(\ell,d)=(0,3),\\[2pt]
 1,&\ell=0\text{ and }0\leq d\leq2,\\[2pt]
 1,&\ell=1,\\[2pt]
 1,&\ell=2\text{ and }0\leq d\leq2,\\[2pt]
 (1-z)^{-1},&(\ell,d)=(2,3),\\[2pt]
 \displaystyle\frac{1-p_1p_2z}
 {(1-z)(1-p_1z)(1-p_2z)},&(\ell,d)=(2,4).
 \end{cases}
\label{eq:intro-Hirzebruch-even-scalars}
\end{equation}
For \(0\leq\ell\leq2\) and \((1,d)\in\mathcal I_\ell\), let
\begin{equation}
 c_{\ell;1,d}(p_1,p_2\mid z)=
 \begin{cases}
 (p_1p_2)^{-1/4}z^{1/4},&d=0,\\
 0,&d=1,\\
 -(p_1p_2z)^{1/4},&d=2\text{ and }\ell=0,1,\\[2pt]
 \displaystyle-\frac{(p_1p_2z)^{1/4}}{1-z},
 &d=2\text{ and }\ell=2.
 \end{cases}
\label{eq:intro-Hirzebruch-odd-scalars}
\end{equation}

\begin{theorem}
\label{thm:finite-NY-levels}
For \(0\leq\ell\leq2\) and every \((j,d)\in\mathcal I_\ell\), the following equality holds in \(\K_{\mathrm{fr}}[[z^{1/4}]]\), with the rational scalars expanded at \(z=0\):
\begin{equation}
 \overline{\mathscr B}_{\ell}^{j,d}(U;p_1,p_2\mid z)
 =c_{\ell;j,d}(p_1,p_2\mid z)
  \mathcal A_\ell(U;p_1,p_2\mid z).
\label{eq:determinant-power-common-relation}
\end{equation}
\end{theorem}

The equality is in \(\K_{\mathrm{fr}}[[z]]\) for \(j=0\) and in \(z^{1/4}\K_{\mathrm{fr}}[[z]]\) for \(j=1\).  The scalars belong to \(\K_{\mathrm{fr}}(z^{1/4})\), not to \(\K_{\mathrm{fr}}\) alone; their Taylor expansions at \(z=0\) belong to the indicated series spaces.  Restoring the common degree \(0\) factor gives \(\mathscr B_{\ell}^{j,d}=c_{\ell;j,d}\widehat{\mathcal A}_\ell\) in this coefficientwise sense.

\begin{proof}
By Lemma~\ref{lem:Shchechkin}(2), every identity invoked below holds coefficientwise over \(\K_{\mathrm{fr}}\).

For \(d_{\mathrm{Sh}}=-1,0,1\), Lemma~\ref{lem:Shchechkin}(1) identifies Shchechkin's equation~(2.21), after division by the factor \(\mathcal A^{(0)}\mathcal E\) in \eqref{eq:normalized-finite-NY-bilinear-operator}, with \eqref{eq:determinant-power-common-relation}.  His coefficient list gives
\[
 c_{\ell;0,d}=1\quad(d=0,1,2),
\]
and at determinant powers \(0\) and \(1\),
\[
 c_{\ell;1,0}=(p_1p_2)^{-1/4}z^{1/4},\qquad
 c_{\ell;1,1}=0,\qquad
 c_{\ell;1,2}=-(p_1p_2z)^{1/4}.
\]
At determinant power \(1\), the extension to the two extreme values is explained in \cite[footnote~6]{BershteinShchechkin}: the Grassmannian pushforward in the last wall-crossing step is computed by Borel--Bott--Weil.  The scalar normalization used here is that of \cite[equation~(2.21)]{Shchechkin}.  It is checked directly at lowest order: \(\mathcal L_1=(1-U)^{-1}\), \(\mathcal L_{-1}=(1-U^{-1})^{-1}\), and the coefficient of \(z^{1/4}\) in the odd sum is
\[
 (p_1p_2)^{(d-1)/4}
 \frac{U^{d/2}-U^{1-d/2}}{1-U}.
\]
For \(d=0,1,2\) this is respectively
\((p_1p_2)^{-1/4},0,-(p_1p_2)^{1/4}\).  Shchechkin's equations~(4.2)--(4.3) give
\[
 c_{0;0,-1}=1-(p_1p_2)^{-1}z,
 \qquad c_{0;0,3}=1-p_1p_2z,
\]
and the calculation following his equation~(4.3) gives \(c_{1;0,3}=1\).  Lemma~\ref{lem:Hirzebruch-ADHM-normalization} identifies the determinant power \(\ell\) in the fixed point formula with the insertion \((\det\cV)^\ell\).

The determinant power \(2\) cases all follow from
\begin{equation}
 \widehat{\mathcal A}_2(U;p_1,p_2\mid z)
 =(z;p_1,p_2)_\infty
  \widehat{\mathcal A}_0(U;p_1,p_2\mid z),
\label{eq:det-two-det-zero}
\end{equation}
which is Proposition~4.1, equation~(4.5), of \cite{Shchechkin}.  For \(r\in\Z\), let
\begin{equation}
 \mathscr P_r(p_1,p_2\mid z)
 :=\frac{(p_1^rz;p_1,p_2/p_1)_\infty
          (p_2^rz;p_1/p_2,p_2)_\infty}
         {(z;p_1,p_2)_\infty}.
\label{eq:det-two-Pochhammer-ratio}
\end{equation}
With \(D_r\) as in \eqref{eq:Hirzebruch-Dr-definition}, the plethystic logarithm of \(\mathscr P_r\) is \(-\sum_{m\geq1}z^mD_r(p_1^m,p_2^m)/m\).  Reduction to a common denominator gives
\[
 D_{-2}=-(xy)^{-1},\qquad D_{-1}=D_0=0,
 \qquad D_1=-1,\qquad D_2=-(1+x+y),
\]
and exponentiating,
\begin{equation}
\begin{gathered}
 \mathscr P_{-2}=\frac1{1-(p_1p_2)^{-1}z},\qquad
 \mathscr P_{-1}=\mathscr P_0=1,\qquad
 \mathscr P_1=\frac1{1-z},\\
 \mathscr P_2=\frac1{(1-z)(1-p_1z)(1-p_2z)}.
\end{gathered}
\label{eq:det-two-Pochhammer-values}
\end{equation}
For \(j=0\), substituting \eqref{eq:det-two-det-zero} in the two factors of \eqref{eq:finite-NY-bilinear-operator} and using the determinant power \(0\) relation with index \(d-1\) gives
\begin{equation}
 c_{2;0,d}=c_{0;0,d-1}\mathscr P_{d-2},
 \qquad 0\leq d\leq4.
\label{eq:det-two-factor-from-det-zero}
\end{equation}
For \(j=1\), the same substitution gives
\begin{equation}
 c_{2;1,d}=c_{0;1,d}\mathscr P_{d-1},
 \qquad 0\leq d\leq2.
\label{eq:det-two-odd-from-det-zero}
\end{equation}
Equations \eqref{eq:det-two-Pochhammer-values}--\eqref{eq:det-two-odd-from-det-zero} give the determinant power \(2\) entries of \eqref{eq:intro-Hirzebruch-even-scalars} and \eqref{eq:intro-Hirzebruch-odd-scalars}.

The normalized identities lie in the two formal series spaces of Lemma~\ref{lem:Shchechkin}(2), and each framed sheaf coefficient is a finite partition sum with coefficients in \(\K_{\mathrm{fr}}\).  After the stable pair substitution, all fourth roots use the branch \eqref{eq:Hirzebruch-signed-branch}.
\end{proof}

\subsection{The stable pair blowup identities}

\begin{theorem}\label{thm:Hirzebruch-main}
Let \(0\leq\ell\leq2\) and \((j,d)\in\mathcal I_\ell\).  Use \(\widehat Z_{Y_\ell}\), \(\mathbf t,\mathbf C,\mathbf r_{\ell;j,d}\), and \(\zeta_\ell\) as defined above, with the lift \eqref{eq:Hirzebruch-signed-branch}.  Then
\begin{equation}
\begin{aligned}
 &\sum_{n\in\Z}
 \widehat Z_{Y_\ell}\left(
 \mathbf t+\eps_1\left(\mathbf Cn+
 \frac{\mathbf r_{\ell;j,d}}2\right);
 \eps_1,\eps_2-\eps_1\right)\\
 &\hspace{18mm}\times
 \widehat Z_{Y_\ell}\left(
 \mathbf t+\eps_2\left(\mathbf Cn+
 \frac{\mathbf r_{\ell;j,d}}2\right);
 \eps_1-\eps_2,\eps_2\right)\\
 &\qquad=
 c_{\ell;j,d}(e^{\eps_1},e^{\eps_2}\mid\zeta_\ell)
 \widehat Z_{Y_\ell}(\mathbf t;\eps_1,\eps_2),
\end{aligned}
\label{eq:Hirzebruch-main-blowup}
\end{equation}
where the scalar coefficients are \eqref{eq:intro-Hirzebruch-even-scalars} and \eqref{eq:intro-Hirzebruch-odd-scalars}.

Precisely, divide both sides by
\[
 e^{F_\ell^{(0)}(T_F,T_M;\eps_1,\eps_2)}
 \mathcal E(e^{T_F};e^{\eps_1},e^{\eps_2})
\]
and use the rational shift quotients of Lemma~\ref{lem:Hirzebruch-shift-quotients}.  The resulting equality lies in \(\K_{\mathrm{fr}}[[\zeta_\ell]]\) for \(j=0\), and in \(\zeta_\ell^{1/4}\K_{\mathrm{fr}}[[\zeta_\ell]]\) for \(j=1\); the rational scalars are expanded at \(\zeta_\ell=0\).  Every coefficient is a finite sum.  For \(\ell=2\), the stable pair coefficients are localized indices.  The scalar vanishes identically exactly when \((j,d)=(1,1)\).

If the three degree \(0\) factors have compatible meromorphic continuations realizing these shift quotients, and the instanton series and normalized bilateral sum converge normally on a common domain adjoining the expansion at \(\zeta_\ell=0\), then \eqref{eq:Hirzebruch-main-blowup} also holds there analytically, and on its connected meromorphic continuations.
\end{theorem}

\begin{proposition}
\label{prop:Hirzebruch-main-rational}
Let \(0\leq\ell\leq2\) and \((j,d)\in\mathcal I_\ell\).  With \(U=e^{T_F}\), \(p_i=e^{\eps_i}\) and \(\zeta_\ell=e^{\pi\mathrm{i}\ell-T_M}\), the coefficientwise identity
\begin{equation}
 \overline{\mathscr B}_{\ell}^{j,d}
 \bigl(e^{T_F};e^{\eps_1},e^{\eps_2}\mid\zeta_\ell\bigr)
 =c_{\ell;j,d}(e^{\eps_1},e^{\eps_2}\mid\zeta_\ell)
  \mathcal A_\ell
  \bigl(e^{T_F};e^{\eps_1},e^{\eps_2}\mid\zeta_\ell\bigr)
\label{eq:Hirzebruch-main-rational}
\end{equation}
holds over \(\K_{\mathrm{fr}}\), and every coefficient is a finite sum in \(\K_{\mathrm{fr}}\).  It is precisely the normalized interpretation of \eqref{eq:Hirzebruch-main-blowup}.  Under the additional convergence and continuation assumptions of Theorem~\ref{thm:Hirzebruch-main}, it also gives the analytic identity.  For \(\ell=2\), these are identities of localized indices.
\end{proposition}

\begin{proof}
Apply Theorem~\ref{thm:finite-NY-levels} and write \(\nu=n+j/2\).  Under the specialization of Proposition~\ref{prop:Hirzebruch-PT-ADHM}, the first framed sheaf factor has
\[
 U\longmapsto Up_1^{2n+j},\qquad
 z\longmapsto p_1^{d-1+\ell(j-1)/2}z,
\]
which is the translation by \(\eps_1(\mathbf Cn+\mathbf r_{\ell;j,d}/2)\) in \((T_F,T_M)\), while the second factor gives the translation by \(\eps_2(\mathbf Cn+\mathbf r_{\ell;j,d}/2)\), and the changes \((p_1,p_2)\mapsto(p_1,p_2/p_1)\) and \((p_1,p_2)\mapsto(p_1/p_2,p_2)\) become \((\eps_1,\eps_2-\eps_1)\) and \((\eps_1-\eps_2,\eps_2)\).

The explicit factor multiplying the two \(\mathcal A_\ell\)-series in \eqref{eq:normalized-finite-NY-bilinear-operator} is the product of the classical quotient \eqref{eq:Hirzebruch-classical-shift-quotient} and the one-loop quotient \eqref{eq:Hirzebruch-one-loop-shift-quotient}.  It is the ratio of the two shifted degree \(0\) factors of \(\widehat Z_{Y_\ell}\) to the unshifted one.  Multiplying \eqref{eq:determinant-power-common-relation} by the unshifted factor converts its left side into the bilinear sum of the two shifted copies of \(\widehat Z_{Y_\ell}\), in the coefficientwise rational sense.

For each coefficient of the lattice sum, substitute \(U=e^{T_F}\), \(p_i=e^{\eps_i}\) and \(z=e^{\pi\mathrm{i}\ell-T_M}\); Proposition~\ref{prop:Hirzebruch-PT-ADHM} identifies the two shifted factors and the unshifted factor with the corresponding stable pair series, and Theorem~\ref{thm:finite-NY-levels} carries the coefficient on the right to \(c_{\ell;j,d}(e^{\eps_1},e^{\eps_2}\mid\zeta_\ell)\), which is \eqref{eq:Hirzebruch-main-rational}.  Under the analytic assumptions of Theorem~\ref{thm:Hirzebruch-main}, normal convergence identifies the sums with their expansions, and the identity principle gives the stated analytic continuation.
\end{proof}

The equivalence \eqref{eq:Hirzebruch-r-equivalence} only reindexes \(n\) in \eqref{eq:Hirzebruch-main-blowup}, and the lists of representatives in the Introduction are read off from \eqref{eq:Hirzebruch-main-shift}.  The coefficient vanishes for \((j,d)=(1,1)\) by \eqref{eq:intro-Hirzebruch-odd-scalars}; this is the vanishing equation for the class of \((2,0)\).

\subsection{Local \texorpdfstring{\(\mathbb F_0\)}{F0} and local \texorpdfstring{\(\mathbb F_2\)}{F2}}
Let \(z=e^{-T_M}\), and choose the base Novikov coordinates so that \(Q_B^{(0)}=Q_B^{(2)}Q_F\).  The base degree normalized series satisfy
\begin{equation}
 Z_{Y_2}^{\PT,\mathrm{norm}}(Q_B^{(2)},Q_F;p_1,p_2)
 =(z;p_1,p_2)_\infty
  Z_{Y_0}^{\PT,\mathrm{norm}}(Q_B^{(0)},Q_F;p_1,p_2),
\label{eq:F0-F2-base-normalized-relation}
\end{equation}
and, after adjoining the classical and one-loop factors,
\begin{equation}
\begin{aligned}
 \widehat Z_{Y_2}(T_F,T_M;\eps_1,\eps_2)
 ={}&\exp\left(-\frac{\pi\mathrm{i}T_F^2}{2\eps_1\eps_2}\right)
 \bigl(e^{-T_M};e^{\eps_1},e^{\eps_2}\bigr)_\infty\\
 &\times\widehat Z_{Y_0}(T_F,T_M;\eps_1,\eps_2).
\end{aligned}
\label{eq:F0-F2-stable-pair-relation}
\end{equation}
Indeed, the choice \(Q_B^{(0)}=Q_B^{(2)}Q_F\) gives \(\zeta_0=\zeta_2=z\) as multiplicative coordinates, and Proposition~\ref{prop:Hirzebruch-PT-ADHM} applied to \eqref{eq:det-two-det-zero} gives \eqref{eq:F0-F2-base-normalized-relation}.  The logarithms \eqref{eq:Hirzebruch-signed-branch} differ by \(2\pi\mathrm{i}\), and \eqref{eq:Hirzebruch-degree-zero-polynomial} gives
\[
 F_2^{(0)}-F_0^{(0)}
 =-\frac{\pi\mathrm{i}T_F^2}{2\eps_1\eps_2}.
\]
Adjoining the degree \(0\) exponentials and \(\mathcal E(U;p_1,p_2)\) gives \eqref{eq:F0-F2-stable-pair-relation}.

\subsection{The Huang--Sun--Wang normalization}
In the integral coordinates \((T_F,T_B)\), Theorem~\ref{thm:Hirzebruch-main} takes the form of \cite{HuangSunWang} once the parity congruence and the degree \(0\) normalization are compared.

Let \(D=\eps_1\eps_2\), \(s=\eps_1+\eps_2\), and
\begin{equation}
\begin{aligned}
 F_{\mathrm{HSW}}(T_F,T_M;\eps_1,\eps_2)
 :={}&\frac{2\zeta(3)}D-
 \frac{\pi\mathrm{i}\,sT_F}{2D}
 +\frac{\eps_1^2+\eps_2^2+3D}{48D}T_M\\
 &+\frac{\eps_1^3+\eps_1^2\eps_2+
          \eps_1\eps_2^2+\eps_2^3}{48D}.
\end{aligned}
\label{eq:Hirzebruch-HSW-exponent}
\end{equation}
This is the rank \(2\) specialization of \cite[equation~(3.120)]{HuangSunWang}, whose gauge root coordinate \(a_{12}\) is our \(T_F\).  Let \(Q_{\mathrm{inst}}=e^{-T_M}\) with logarithm \(\log Q_{\mathrm{inst}}=2\pi\mathrm{i}-T_M\).  The framed sheaf counting coordinate in Theorem~\ref{thm:Hirzebruch-main} is \(\zeta_\ell=(-1)^\ell Q_{\mathrm{inst}}\), with fourth root fixed by \eqref{eq:Hirzebruch-signed-branch}.  We keep the two additive lifts distinct: \(2\pi\mathrm{i}-T_M\) is the unsigned gauge theory lift used in the perturbative normalization of Huang--Sun--Wang, and \(-T_M+\pi\mathrm{i}\ell\) is the lift of the signed stable pair coordinate.  Let
\[
 \widehat Z_{Y_\ell}^{\mathrm{HSW}}
 =e^{-F_{\mathrm{HSW}}}\widehat Z_{Y_\ell}.
\]
This removes the normalization term of \cite[equations~(3.118)--(3.120)]{HuangSunWang}.  Their further term linear in the shift invariant coordinate \(T_M\) is convention dependent \cite[equation~(3.124)]{HuangSunWang}, and we set it to \(0\); restoring it only multiplies the nonzero coefficient on the right by a function independent of \(T_F\).

To state the lattice and parity data in an integral basis, let
\begin{equation}
\begin{gathered}
 \widetilde{\mathbf t}=\begin{pmatrix}T_F\\T_B\end{pmatrix},\qquad
 A_\ell=\begin{pmatrix}1&0\\ \ell/2-1&1\end{pmatrix},\qquad
 \mathbf t=A_\ell\widetilde{\mathbf t},\\
 \widetilde{\mathbf C}_\ell=A_\ell^{-1}\mathbf C
 =\begin{pmatrix}2\\2-\ell\end{pmatrix},\qquad
 \widetilde{\mathbf r}_{\ell;j,d}=A_\ell^{-1}\mathbf r_{\ell;j,d}
 =\begin{pmatrix}
 2j\\ \ell+2-2d+2(1-\ell)j
 \end{pmatrix}.
\end{gathered}
\label{eq:Hirzebruch-integral-HSW-data}
\end{equation}
Both vectors on the last line are integral, also for \(\ell=1\), and
\begin{equation}
 \widetilde{\mathbf r}_{\ell;j,d}
 \equiv \widetilde{\mathbf B}_\ell
 :=\begin{pmatrix}0\\\ell\end{pmatrix}\pmod{2}.
\label{eq:Hirzebruch-common-checkerboard}
\end{equation}
The class \(\widetilde{\mathbf B}_\ell\) is the parity class \((0,2+\ell)^t\bmod 2\) of the corresponding rank \(2\) geometry in \cite[Section~3.4]{HuangSunWang}.  All equations below, including the \(j=1\) equations for \(\mathbb F_1\), satisfy the same parity congruence, and the equivalence relation becomes the integral relation \(\widetilde{\mathbf r}'-\widetilde{\mathbf r} \in2\widetilde{\mathbf C}_\ell\Z\).  Let
\[
 \widehat Z_{Y_\ell,\mathrm{int}}^{\mathrm{HSW}}
 (\widetilde{\mathbf t};\eps_1,\eps_2)
 :=\widehat Z_{Y_\ell}^{\mathrm{HSW}}
 \left(T_F,T_B+\left(\frac\ell2-1\right)T_F;
 \eps_1,\eps_2\right).
\]

\begin{corollary}
\label{cor:Hirzebruch-exact-HSW}
For \((j,d)\in\mathcal I_\ell\),
\begin{equation}
\begin{aligned}
 &\sum_{n\in\Z}(-1)^n
 \widehat Z_{Y_\ell,\mathrm{int}}^{\mathrm{HSW}}\left(
 \widetilde{\mathbf t}+\eps_1\left(\widetilde{\mathbf C}_\ell n+
 \frac{\widetilde{\mathbf r}_{\ell;j,d}}2\right);
 \eps_1,\eps_2-\eps_1\right)\\
 &\hspace{18mm}\times
 \widehat Z_{Y_\ell,\mathrm{int}}^{\mathrm{HSW}}\left(
 \widetilde{\mathbf t}+\eps_2\left(\widetilde{\mathbf C}_\ell n+
 \frac{\widetilde{\mathbf r}_{\ell;j,d}}2\right);
 \eps_1-\eps_2,\eps_2\right)\\
 &\qquad=
 \mathrm{i}^j\exp\left(-\frac{\bigl(\ell(1-j)-2d\bigr)s}{48}\right)
 c_{\ell;j,d}(e^{\eps_1},e^{\eps_2}\mid\zeta_\ell)
 \widehat Z_{Y_\ell,\mathrm{int}}^{\mathrm{HSW}}
 (\widetilde{\mathbf t};\eps_1,\eps_2),
 \qquad T_M=T_B+\left(\frac\ell2-1\right)T_F.
\end{aligned}
\label{eq:Hirzebruch-exact-HSW}
\end{equation}
\end{corollary}

\begin{proof}
Let
\[
 A=2n+j,\qquad r=\ell(1-j)+2-2d,
\]
let \(F=F_{\mathrm{HSW}}\), and let
\begin{align*}
 \Delta F={}&F(T_F+\eps_1A,T_M+\eps_1r/2;
                  \eps_1,\eps_2-\eps_1)\\
 &+F(T_F+\eps_2A,T_M+\eps_2r/2;
                  \eps_1-\eps_2,\eps_2)
   -F(T_F,T_M;\eps_1,\eps_2).
\end{align*}
The \(\zeta(3)\), \(T_F\) and \(T_M\) coefficients cancel, and the shifted term linear in \(T_F\) contributes \(-\pi\mathrm{i}A/2\).  Write
\[
 g(a,b)=\frac{a^2+b^2+3ab}{48ab},\qquad
 h(a,b)=\frac{a^3+a^2b+ab^2+b^3}{48ab}.
\]
Putting everything over a common denominator, one finds
\begin{align*}
 &g(a,b-a)+g(a-b,b)=g(a,b),\\
 &a\,g(a,b-a)+b\,g(a-b,b)=\frac{a+b}{24},\\
 &h(a,b-a)+h(a-b,b)-h(a,b)=-\frac{a+b}{24}.
\end{align*}
Hence
\[
 \Delta F=-\frac{\pi\mathrm{i}A}{2}+\frac{(r-2)s}{48}
 =-\pi\mathrm{i}\left(n+\frac j2\right)
  +\frac{\bigl(\ell(1-j)-2d\bigr)s}{48},
\]
and
\begin{equation}
 e^{\Delta F}=(-1)^n e^{-\pi\mathrm{i}j/2}
 \exp\left(\frac{\bigl(\ell(1-j)-2d\bigr)s}{48}\right).
\label{eq:Hirzebruch-HSW-sign}
\end{equation}
Replacing \(\widehat Z_{Y_\ell}\) by \(e^{F_{\mathrm{HSW}}}\widehat Z_{Y_\ell}^{\mathrm{HSW}}\) in \eqref{eq:Hirzebruch-main-blowup} and dividing by \eqref{eq:Hirzebruch-HSW-sign} gives the equation in \((T_F,T_M)\), and conjugating every translation by \(A_\ell^{-1}\) gives \eqref{eq:Hirzebruch-exact-HSW} in the coordinates \((T_F,T_B)\).  The formula for \(\widetilde{\mathbf r}_{\ell;j,d}\) and the congruence \eqref{eq:Hirzebruch-common-checkerboard} follow by multiplying out \eqref{eq:Hirzebruch-integral-HSW-data}.
\end{proof}

For \(j=0\), the exponential in \eqref{eq:Hirzebruch-HSW-sign} is the rank \(2\) factor of \cite[equation~(3.128)]{HuangSunWang}.  For \(j=1\), the fixed logarithm and fourth root branches contribute the additional phase \(e^{-\pi\mathrm{i}/2}\), and the \(\ell\)-dependent term drops out of the real exponent since \(\ell(1-j)=0\).  These factors are independent of \(T_F\); with the definitions \eqref{eq:intro-Hirzebruch-odd-scalars} of \(c_{\ell;j,d}\) and \eqref{eq:Hirzebruch-signed-branch} of \(\zeta_\ell^{1/4}\) they remain in the displayed coefficient.

\section{The transition to local \texorpdfstring{\(\mathbb P^2\)}{P2}}
\label{sec:stable-pair-blowup}

\subsection{The flop and large-volume limit}
\label{sec:diagonal-P2}
Let \(X=\Tot_{\mathbb P^2}K_{\mathbb P^2}\), and let \(H\in H_2(X,\Z)\) be the class of a line in the zero section.  We fix a global square root \((K^{\mathrm{vir}})^{1/2}\); its restrictions supply the square roots on the fixed components.

The stable pair/framed sheaf comparison of this section assumes Conjecture~\ref{conj:P2-vertex-comparison}.  The blowup equations and the coefficient recursion assume in addition Conjecture~\ref{conj:P2-asymptotic-resummation}.

For a rational function of the framing ratio \(u\), let
\begin{equation}
 \iota_0:
 \Q(q_1^{1/2},q_2^{1/2})(u)
 \longrightarrow
 \Q(q_1^{1/2},q_2^{1/2})((u))
\label{eq:global-u-expansion}
\end{equation}
be the Laurent expansion at \(u=0\), applied separately to each coefficient of the variable recording \(c_2\).  In this section the framed sheaf characters are
\begin{equation}
 q_1=t^{-1},\qquad q_2=q,\qquad (u_1,u_2)=(u,1).
\label{eq:introduction-variable-change}
\end{equation}
We use rank \(2\) and determinant power \(1\).  If \(z\) is the variable recording \(c_2\) after comparison with stable pairs, then
\begin{equation}
 \mathfrak q=-\left(\frac qt\right)^{3/2}z.
\label{eq:rank-two-specialization}
\end{equation}
By \eqref{eq:det-fixed-weight} and \eqref{eq:ADHM-Nekrasov-denominator},
\begin{equation}
\begin{aligned}
 \iota_0Z_{2,1}^{\mathrm{fr}}
 \left(-\left(\frac qt\right)^{3/2}z;t^{-1},q;u,1\right)
 ={}&\iota_0\sum_{\lambda_1,\lambda_2}
 \left(-\left(\frac qt\right)^{3/2}z\right)^{|\lambda_1|+|\lambda_2|}\\
 &\quad\times
 \frac{\displaystyle\prod_{a=1}^2
 u_a^{|\lambda_a|}t^{-n(\lambda_a)}
 q^{n(\lambda_a^t)}}
 {\displaystyle\prod_{a,b=1}^2
 N_{\lambda_a\lambda_b}
 (u_a/u_b;t^{-1},q)},
 \\
 &\text{where }(u_1,u_2)=(u,1).
\end{aligned}
\label{eq:geometric-fixed-point-sum}
\end{equation}
The factor multiplying \(z\) comes from the cotangent character convention \eqref{eq:localized-Euler-convention}.

For \(\ell=1\), restoring the fibre series in Proposition~\ref{prop:Hirzebruch-PT-ADHM} gives
\begin{equation}
\begin{aligned}
 Z_{Y_1}^{\PT,\mathrm{ref}}(Q_B,Q_F;t,q^{-1})
 ={}&\PE\left[\frac{(t+q)Q_F}{(1-t)(1-q)}\right]\\
 &\times\iota_0 Z_{2,1}^{\mathrm{fr}}\left(
 -\left(\frac qt\right)^{3/2}Q_BQ_F^{-1};
 t^{-1},q;Q_F,1\right).
\end{aligned}
\label{eq:F1-verified-finite-parameter-form}
\end{equation}
One sets \(u=Q_F\) and keeps \(Q=Q_BQ_F\) fixed, so \(Q_B=Q/u\) and the framed sheaf argument is proportional to \(Qu^{-2}\), and the large-volume specialization after the flop is represented by extraction of the coefficient of \(u^0\), applied separately to every coefficient of \(Q\), as in Theorem~\ref{thm:P2-PT-ADHM}.

\subsection{Properness and the square root}
\begin{lemma}
\label{lem:PT-fixed-proper}
The moduli space \(P_n(X,dH)\) is proper for every \(n\in\Z\) and \(d\geq0\).
\end{lemma}

\begin{proof}
Compactify
\[
 X\subset\overline X
 =\mathbb P_{\mathbb P^2}(\cO\oplus\cO(-3))
\]
with divisor at infinity \(D_\infty\).  The closure of a \(1\)-dimensional support in class \(dH\) has intersection number \(0\) with \(D_\infty\).  If an irreducible curve \(C\) is not contained in \(D_\infty\), then \(D_\infty\cdot C\) is the degree of the effective Cartier divisor \(C\cap D_\infty\), which is positive if the intersection is nonempty.  If \(C\subset D_\infty\), then
\[
 D_\infty\cdot C
 =\deg\!\left(N_{D_\infty/\overline X}|_C\right)>0,
 \qquad
 N_{D_\infty/\overline X}\cong\cO_{\mathbb P^2}(3).
\]
Every effective curve in class \(dH\) is therefore disjoint from \(D_\infty\), and the projective stable pair space of \(\overline X\) in this class is \(P_n(X,dH)\) itself.
\end{proof}

\begin{lemma}
\label{lem:global-orientation-localization}
Let \(d>0\) and write \(Q^d=\mathbf Q^{dH}\).  The coefficient of \(y^nQ^d\) in Definition~\ref{def:KPT-series} is
\begin{equation}
 \chi_{\widetilde{\mathbf T}}
 \bigl(P_n(X,dH),\widehat\cO^{\mathrm{vir}}\bigr).
\label{eq:P2-global-symmetrized-index}
\end{equation}
Let \(\mathscr F\subset P_n(X,dH)^{\mathbf T}\) be connected, choose a generic real cocharacter \(\sigma\in\Hom(\C^*,\ker\kappa) \otimes_{\Z}\mathbb R\), and write the moving virtual tangent class according to its attracting and repelling weights as
\[
 N^{\mathrm{vir}}=N_+-\kappa N_+^\vee.
\]
The square root induced on the fixed obstruction theory is
\begin{equation}
 (K_{\mathscr F}^{\mathrm{vir}})^{1/2}
 =\kappa^{-\rk(N_+)/2}\det(N_+)\otimes
   (K^{\mathrm{vir}})^{1/2}|_{\mathscr F},
\label{eq:P2-induced-fixed-orientation}
\end{equation}
and with these square roots the fixed locus formula \eqref{eq:KPT-series} is the virtual localization of \eqref{eq:P2-global-symmetrized-index}.
\end{lemma}

\begin{proof}
Properness is Lemma~\ref{lem:PT-fixed-proper}.  The determinant line identity and virtual localization for the symmetrized virtual structure sheaf are proved in \cite[Section~7]{NekrasovOkounkov}.  If \(T^{\mathrm{vir}}|_{\mathscr F}=V_+-\kappa V_+^\vee\), then
\[
 K^{\mathrm{vir}}|_{\mathscr F}
 =\kappa^{\rk(V_+)}(\det V_+)^{-2},
\]
and for the moving part
\[
 \det N^{\mathrm{vir}}
 =\kappa^{-\rk(N_+)}(\det N_+)^2.
\]
The identity \(K^{\mathrm{vir}}|_{\mathscr F}=K_{\mathscr F}^{\mathrm{vir}} \otimes(\det N^{\mathrm{vir}})^{-1}\) gives \eqref{eq:P2-induced-fixed-orientation} and the localization formula.
\end{proof}

\begin{definition}
\label{def:PT-refined-series}
Let \(Z_X^{\mathrm{PT},\mathrm{ref}}(Q;t,q)\) be the symmetrized stable pair series of Definition~\ref{def:KPT-series}, with the chosen global square root, after the substitution \eqref{eq:PT-refined-variables}.
\end{definition}

\begin{lemma}
\label{lem:P2-kappa-rigidity}
For every \(n\in\Z\) and \(d\geq0\), the coefficient of \(y^nQ^d\) in \(Z_X^{\mathrm{PT},K}\) lies in \(\Z[\kappa^{\pm1/2}]\).  It may therefore be computed with any generic cocharacter of \(\ker(\kappa)\).  The same holds after the substitution defining \(Z_X^{\mathrm{PT},\mathrm{ref}}\).
\end{lemma}

\begin{proof}
For \(d=0\) the coefficient is the constant term, \(1\) for \(n=0\) and \(0\) otherwise.  For \(d>0\), by Lemmas~\ref{lem:PT-fixed-proper} and \ref{lem:global-orientation-localization} the coefficient is the equivariant Euler characteristic of the global symmetrized virtual structure sheaf on a proper moduli space, and Proposition~7.4 of \cite{NekrasovOkounkov} places it in \(\Z[\kappa^{\pm1/2}]\) and proves invariance under subtori of \(\ker(\kappa)\).
\end{proof}

Under \eqref{eq:introduction-variable-change}, \((q_1q_2)^{1/2}=\sqrt{q/t}=-\kappa^{-1/2}\) on the branch \eqref{eq:PT-refined-square-root-branches}.  We write \(N\) for the integer \(c_2\).

\subsection{Macdonald functions and the two-leg sum}
Let \(\operatorname{Sym}_{\mathsf p}\) be the ring of symmetric functions over \(\K\), with power sums \(\mathsf p_k\).  The variables \(\mathsf p_k\) are independent of the equivariant parameters \(p_1,p_2\).  For a partition \(\alpha=(1^{m_1}2^{m_2}\cdots)\), let
\[
 \mathsf p_\alpha=\prod_i \mathsf p_{\alpha_i},\qquad
 z_\alpha=\prod_{j\geq1}j^{m_j}m_j!,
\]
and let the Macdonald scalar product \cite[Chapter~VI]{Macdonald} be
\begin{equation}
 \langle \mathsf p_\alpha,\mathsf p_\beta\rangle_{q,t}
 =\delta_{\alpha\beta}z_\alpha
 \prod_i\frac{1-q^{\alpha_i}}{1-t^{\alpha_i}}.
\label{eq:Macdonald-scalar-product}
\end{equation}
Let \(P_\lambda(\mathbf X;q,t)\) be the monic Macdonald function, so that
\begin{equation}
 \langle P_\lambda,P_\mu\rangle_{q,t}
 =\frac{\delta_{\lambda\mu}}{b_\lambda(q,t)},
 \qquad
 b_\lambda(q,t)=
 \prod_{s\in\lambda}
 \frac{1-q^{a_\lambda(s)}t^{\ell_\lambda(s)+1}}
      {1-q^{a_\lambda(s)+1}t^{\ell_\lambda(s)}}.
\label{eq:Macdonald-norm}
\end{equation}
We write
\begin{equation}
 s_\eta=\sum_\sigma \mathsf K_{\eta\sigma}(q,t)P_\sigma
\label{eq:Schur-Macdonald-change}
\end{equation}
for the change of homogeneous bases.  Principal specializations use
\[
 t^{-\rho}=(t^{1/2},t^{3/2},\ldots),\qquad
 q^{-\rho}=(q^{1/2},q^{3/2},\ldots).
\]
With \(C_{\lambda\mu\nu}(t,q)\) the refined vertex \eqref{eq:refined-vertex-definition}, let
\begin{align}
 \mathsf Z_{\lambda\mu}(Q)
 &:=
 \sum_\tau(-Q)^{|\tau|}
 C_{\lambda^t\varnothing\tau}(t,q)
 C_{\varnothing\mu\tau^t}(q,t),
 \label{eq:two-leg-Z}\\
 \mathsf G_{\lambda\mu}(Q)
 &:=
 \frac{\mathsf Z_{\lambda\mu}(Q)}
 {\mathsf Z_{\varnothing\varnothing}(Q)}.
 \label{eq:two-leg-G}
\end{align}

\begin{lemma}\label{lem:two-leg-polynomiality}
Let \(m=|\lambda|+|\mu|\).  Then \(\mathsf G_{\lambda\mu}(Q)\) is a polynomial of degree at most \(m\).  In fact
\begin{equation}
\begin{aligned}
 \mathsf G_{\lambda\mu}(Q)
 ={}&g_{\lambda\mu}
 \sum_{\eta,\sigma}
 c_{\lambda\mu}^{\eta}
 \mathsf K_{\eta\sigma}(q,t)P_\sigma(t^{-\rho};q,t)\\
 &\qquad\times
 \prod_{(i,j)\in\sigma}
 \left(1-Qt^{i-1/2}q^{-j+1/2}\right),
\end{aligned}
\label{eq:two-leg-Macdonald}
\end{equation}
where \(c_{\lambda\mu}^{\eta}\) are the Littlewood--Richardson coefficients and
\begin{equation}
 g_{\lambda\mu}
 =(-1)^{|\mu|}
 q^{|\lambda|/2}t^{-|\lambda|/2}
 \widetilde f_{\mu^t}(t,q).
\label{eq:g-lambda-mu}
\end{equation}
Consequently,
\begin{equation}
 x^m\mathsf G_{\lambda\mu}(x^{-1})\in\K[x],
 \qquad
 \left.x^m\mathsf G_{\lambda\mu}(x^{-1})\right|_{x=0}
 =[Q^m]\mathsf G_{\lambda\mu}(Q).
\label{eq:two-leg-leading-coefficient}
\end{equation}
\end{lemma}

\begin{proof}
One of the two nondistinguished legs is empty in each vertex, so \eqref{eq:refined-vertex-definition} gives
\begin{equation}
\begin{aligned}
 &C_{\lambda^t,\varnothing,\tau}(t,q)
  C_{\varnothing,\mu,\tau^t}(q,t)\\
 &\quad=g_{\lambda\mu}
 P_{\tau^t}(q^{-\rho};t,q)P_\tau(t^{-\rho};q,t)
 s_\lambda(t^{-\rho}q^{-\tau})
 s_\mu(t^{-\rho}q^{-\tau}).
\end{aligned}
\label{eq:two-leg-vertex-contraction}
\end{equation}
The \(Q\)-independent factor is \(g_{\lambda\mu}\) of \eqref{eq:g-lambda-mu}; this is also the prefactor of \cite[Section~4]{IqbalKozcaz}.  Since
\[
 s_\lambda s_\mu
 =\sum_\eta c_{\lambda\mu}^{\eta}s_\eta,
\]
expanding \(s_\eta\) by \eqref{eq:Schur-Macdonald-change} gives
\[
 \mathsf Z_{\lambda\mu}(Q)
 =g_{\lambda\mu}\sum_{\eta,\sigma}
 c_{\lambda\mu}^{\eta}\mathsf K_{\eta\sigma}(q,t)
 J_\sigma(Q),
\]
where
\begin{equation}
 J_\sigma(Q)
 =
 \sum_\nu(-Q)^{|\nu|}
 P_{\nu^t}(q^{-\rho};t,q)
 P_\nu(t^{-\rho};q,t)
 P_\sigma(t^{-\rho}q^{-\nu};q,t).
\label{eq:J-sigma}
\end{equation}
The evaluation symmetry \cite[Chapter~VI, Section~6, Example~6]{Macdonald} gives
\[
\begin{aligned}
 &P_\nu(t^{-\rho};q,t)
 P_\sigma(t^{-\rho}q^{-\nu};q,t)\\
 &\hspace{22mm}=
 P_\sigma(t^{-\rho};q,t)
 P_\nu(t^{-\rho}q^{-\sigma};q,t),
\end{aligned}
\]
and the dual Cauchy identity applied to \eqref{eq:J-sigma} gives
\[
 J_\sigma(Q)
 =
 P_\sigma(t^{-\rho};q,t)
 \prod_{i,j\geq1}
 \left(1-Qt^{i-1/2}q^{j-1/2-\sigma_i}\right).
\]
After division by \(J_\varnothing(Q)\) the infinite tails cancel:
\begin{equation}
 \frac{J_\sigma(Q)}{J_\varnothing(Q)}
 =
 P_\sigma(t^{-\rho};q,t)
 \prod_{(i,j)\in\sigma}
 \left(1-Qt^{i-1/2}q^{-j+1/2}\right).
\label{eq:J-ratio}
\end{equation}
This is \eqref{eq:two-leg-Macdonald}.

If \(c_{\lambda\mu}^{\eta}\neq0\) then \(|\eta|=m\), and \eqref{eq:Schur-Macdonald-change} preserves degree, so \(\mathsf K_{\eta\sigma}\neq0\) forces \(|\sigma|=m\) and every product in \eqref{eq:two-leg-Macdonald} has exactly \(m\) linear factors.  Equation~\eqref{eq:two-leg-leading-coefficient} is the leading coefficient identity for a polynomial of degree at most \(m\).
\end{proof}

\subsection{The vertex sum}
For \(m=|\lambda|+|\mu|\), let
\begin{equation}
 \mathsf L_{\lambda\mu}(t,q)
 =
 [x^m]\mathsf G_{\lambda\mu}(x).
\label{eq:L-lambda-mu}
\end{equation}
The first values are
\begin{equation}
 \mathsf L_{\varnothing\varnothing}=1,
 \qquad
 \mathsf L_{(1),\varnothing}=-\frac{\sqrt t}{1-t},
 \qquad
 \mathsf L_{\varnothing,(1)}
 =-\frac{t}{\sqrt q(1-t)}.
\label{eq:two-leg-first-coefficients}
\end{equation}
Let
\begin{equation}
\begin{aligned}
 \mathsf Z_{\mathbb P^2}^{\mathrm{vert}}(Q;t,q)
 ={}&\sum_{\lambda,\mu,\nu}
 (-Q)^{|\lambda|+|\mu|+|\nu|}
 f_\nu(q,t)^2
 \widetilde f_\lambda(t,q)\widetilde f_\mu(q,t)\\
 &\quad\times
 C_{\varnothing\lambda\nu^t}(t,q)
 C_{\mu^t\varnothing\nu}(q,t)
 \mathsf L_{\lambda\mu}(t,q).
\end{aligned}
\label{eq:P2-IK-sum}
\end{equation}
Only finitely many triples of partitions occur in each coefficient of \(Q\).  The expression \eqref{eq:two-leg-Macdonald} agrees with the normalized coefficient of Iqbal--Koz\c{c}az \cite[Section~4]{IqbalKozcaz}, which motivates \eqref{eq:L-lambda-mu} and \eqref{eq:P2-IK-sum}.

\subsection{The three charts}
Let \(P_i\) be the fixed point of \(\mathbb P^2\) at which the \(i\)-th homogeneous coordinate is nonzero, and let \(\xi_i\) be the character of that coordinate, lifted so that \(\xi_1\xi_2\xi_3=\kappa\).  We use the two-parameter subgroup of \(\ker(\kappa)\),
\[
 (a,b)\longmapsto(ab,a/b,a^{-2}),
\]
which acts on the characters \(\xi_i\) by
\begin{equation}
 (w_1,w_2,w_3)
 =(\xi_1ab,\xi_2a/b,\xi_3a^{-2}).
\label{eq:P2-iterated-cocharacter}
\end{equation}
We take the iterated limit \(a\to0\) and then \(b\to\infty\).  Lemma~\ref{lem:P2-three-chart-weights} shows that the residual \(b\)-action makes this limit generic at all three fixed points.

\begin{lemma}
\label{lem:P2-three-chart-weights}
For a character of the subgroup \eqref{eq:P2-iterated-cocharacter}, let \(\operatorname{val}_a\) and \(\operatorname{val}_b\) be the exponents of \(a\) and \(b\).  At \(P_i\), order the tangent characters as
\[
 \left(\frac{w_j}{w_i},\frac{w_k}{w_i},w_i^3\right),
\]
with \((i,j,k)\) equal to \((1,2,3)\), \((2,1,3)\), \((3,1,2)\) respectively.  Then
\begin{equation}
 \operatorname{val}_a(T_{P_1}X)
 =\operatorname{val}_a(T_{P_2}X)=(0,-3,3),
 \qquad
 \operatorname{val}_a(T_{P_3}X)=(3,3,-6),
\label{eq:P2-three-chart-valuations}
\end{equation}
and
\begin{equation}
 \operatorname{val}_b(T_{P_1}X)=(-2,-1,3),\qquad
 \operatorname{val}_b(T_{P_2}X)=(2,1,-3),\qquad
 \operatorname{val}_b(T_{P_3}X)=(1,-1,0).
\label{eq:P2-three-chart-residual-valuations}
\end{equation}
Put \(\lambda,\mu,\nu\) on the edges \(P_1P_3,P_2P_3,P_1P_2\).  After the cyclic reordering
\[
 (w_1^3,w_3/w_1,w_2/w_1),\qquad
 (w_2^3,w_3/w_2,w_1/w_2)
\]
of the coordinates at \(P_1\) and \(P_2\), with partition triples \((\varnothing,\lambda^t,\nu)\) and \((\varnothing,\mu^t,\nu)\), the residual limit \(b\to\infty\) gives the orders \(r_1\gg r_3>0\gg r_2\) and \(r_1\gg0>r_3\gg r_2\), and Proposition~\ref{prop:chart-slope-limit} gives
\begin{equation}
 \mathsf C^{\mathrm A}_{\varnothing,\lambda^t,\nu}(t,q)
 \quad\text{at }P_1,
 \qquad
 \mathsf C^{\mathrm A}_{\mu,\varnothing,\nu^t}(q,t)
 \quad\text{at }P_2.
\label{eq:P2-first-two-Arbesfeld-vertices}
\end{equation}
\end{lemma}

At \(P_3\), the two compact directions have leading \(a\)-weights of the same sign.

\begin{proof}
At \(P_i\), the two tangent characters of \(\mathbb P^2\) are \(w_j/w_i\) and \(w_k/w_i\), and the fibre of \(K_{\mathbb P^2}\) has character
\[
 \frac{\kappa w_i^2}{w_jw_k}=w_i^3,
\]
since \(w_1w_2w_3=\kappa\).  The ordered tangent characters of \(X\) are thus \(w_j/w_i\), \(w_k/w_i\), \(w_i^3\), and substituting \eqref{eq:P2-iterated-cocharacter} gives \eqref{eq:P2-three-chart-valuations} and \eqref{eq:P2-three-chart-residual-valuations}.  The residual weights separate the zero weight at the first two points and the two equal weights at the third.
\end{proof}

\begin{conjecture}
\label{conj:P2-vertex-comparison}
For the iterated limit of the subgroup \eqref{eq:P2-iterated-cocharacter}, the coefficientwise slope limit of the symmetrized stable pair localization formula is
\begin{equation}
 Z_{X}^{\mathrm{PT},\mathrm{ref}}(Q;t,q)
 =\mathsf Z_{\mathbb P^2}^{\mathrm{vert}}(Q;t,q).
\label{eq:P2-PT-equals-IK-vertex}
\end{equation}
\end{conjecture}

The natural approach to the conjecture is to cut the two compact toric curves through \(P_3\) and apply a relative \(K\)-theoretic degeneration formula for stable pairs, which requires the degeneration formula itself, and the compatibility of its symmetrized square roots with the two edge factors.  Neither is available at present in the generality needed.

\subsection{Strip identities}
For partitions \(\alpha,\beta\), let
\begin{equation}
 \mathcal K_{\alpha\beta}(u;t,q)
 =\sum_\lambda(-u)^{|\lambda|}\widetilde f_\lambda(t,q)
 C_{\varnothing,\lambda,\alpha}(t,q)
 C_{\lambda^t,\varnothing,\beta}(t,q).
\label{eq:strip-contraction}
\end{equation}
Then
\begin{equation}
\begin{aligned}
 \mathcal K_{\alpha\beta}(u;t,q)
 ={}&q^{(\|\alpha\|^2+\|\beta\|^2)/2}
 \widetilde Z_\alpha(t,q)\widetilde Z_\beta(t,q)\\
 &\times\prod_{i,j\geq1}
 \left(1-u t^{-\alpha_i^t+j-1}q^{-\beta_j+i}\right)^{-1}.
\end{aligned}
\label{eq:strip-product}
\end{equation}
With \(N_{\lambda\mu}(x)=N_{\lambda\mu}(x;t^{-1},q)\),
\begin{equation}
 \frac1{N_{\lambda\lambda}(1)}
 =(-1)^{|\lambda|}\left(\frac tq\right)^{|\lambda|/2}
  \mathscr D_\lambda(t,q),
\label{eq:self-Nekrasov-product}
\end{equation}
and
\begin{equation}
\begin{aligned}
 &\frac1{N_{\alpha\beta}(u)N_{\beta\alpha}(u^{-1})}\\
 &\quad={}
 (-u)^m
 \left(\frac qt\right)^{-m/2+
 (\|\alpha^t\|^2-\|\beta^t\|^2)/2}
 q^{(c_\alpha-c_\beta)/2}
 \mathcal S_{\alpha\beta}(u;t,q),
\end{aligned}
\label{eq:cross-Nekrasov-product}
\end{equation}
where \(m=|\alpha|+|\beta|\) and
\begin{equation}
\begin{aligned}
 \mathcal S_{\alpha\beta}(u;t,q)
 =\prod_{i,j\geq1}
 \frac{(1-u t^{j-1}q^i)(1-u t^jq^{i-1})}
 {(1-u t^{-\alpha_i^t+j-1}q^{-\beta_j+i})
  (1-u t^{-\alpha_i^t+j}q^{-\beta_j+i-1})}.
\end{aligned}
\label{eq:R-alpha-beta}
\end{equation}
The product in \eqref{eq:strip-product} is a formal power series in \(u\).  Solving \eqref{eq:cross-Nekrasov-product} for \(\mathcal S_{\alpha\beta}\) gives the finite expression
\begin{equation}
\begin{aligned}
 \mathcal S_{\alpha\beta}(u;t,q)
 ={}&\frac{(-u)^{-m}
 (q/t)^{m/2-(\|\alpha^t\|^2-\|\beta^t\|^2)/2}
 q^{-(c_\alpha-c_\beta)/2}}
 {N_{\alpha\beta}(u;t^{-1},q)
  N_{\beta\alpha}(u^{-1};t^{-1},q)},
\end{aligned}
\label{eq:R-alpha-beta-finite}
\end{equation}

To verify these identities, substitute the vertex formula of \cite[Section~2]{IqbalKozcaz}: one Schur function remains from each vertex, and the framing monomial cancels their shape dependent prefactors; the Cauchy identity then gives \eqref{eq:strip-product}.  The arm--leg factor satisfies
\begin{equation}
 \frac1{N_{\alpha\beta}(x)}
 =\prod_{i,j\geq1}
 \frac{1-x t^{j-1}q^i}
 {1-x t^{-\alpha_j^t+i-1}q^{-\beta_i+j}},
\label{eq:Nekrasov-infinite-product}
\end{equation}
whose tails cancel against the empty partition product.  Setting \(\alpha=\beta\) gives \eqref{eq:self-Nekrasov-product}.  For the cross factors, multiply \eqref{eq:Nekrasov-infinite-product} for \((\alpha,\beta,u)\) and \((\beta,\alpha,u^{-1})\), and replace every factor of the second product by \(1-X=-X(1-X^{-1})\).  The exponents of \(t\) and \(q\) in the resulting monomial are given by
\[
 \sum_{s\in\alpha}a_\alpha(s)
 -\sum_{s\in\beta}\bigl(a_\beta(s)+1\bigr)
 =n(\alpha^t)-n(\beta^t)-|\beta|,
\]
\[
 \sum_{s\in\alpha}\bigl(\ell_\beta(s)+1\bigr)
 -\sum_{s\in\beta}\ell_\alpha(s)
 =|\beta|-n(\alpha)+n(\beta),
\]
which follow by summing along rows and along columns respectively.  Together with
\[
 \|\lambda\|^2=|\lambda|+2n(\lambda^t),\qquad
 \|\lambda^t\|^2=|\lambda|+2n(\lambda),
\]
these give the monomial of \eqref{eq:cross-Nekrasov-product}; the factors not inverted by \(1-X=-X(1-X^{-1})\) are exactly \eqref{eq:R-alpha-beta}.  These are the identities used in the fixed point comparisons of \cite{TakiRefined,AwataKanno}.

The empty partition factors are
\begin{align}
\mathcal K_{\varnothing\varnothing}(u;t,q)
\mathcal K_{\varnothing\varnothing}(u;q,t)
&=
 \PE\left[\frac{(t+q)u}{(1-t)(1-q)}\right],
\label{eq:empty-one-leg-factors}\\
 \mathsf Z_{\varnothing\varnothing}(x;t,q)
 &=\PE\left[-\frac{\sqrt{tq}\,x}{(1-t)(1-q)}\right].
\label{eq:empty-two-leg-explicit}
\end{align}
The first follows by applying the Cauchy identity to the two empty one-leg sums, the second by setting both external partitions to \(\varnothing\) in \eqref{eq:two-leg-Z} and applying the dual Cauchy identity.

\subsection{The fixed point identity}
\begin{lemma}
\label{lem:direct-ADHM-regrouping}
For partitions \(\lambda,\mu,\nu\), let
\begin{align*}
 \mathsf A_{\lambda\mu\nu}
 &:={}
 f_\nu(q,t)^2\widetilde f_\lambda(t,q)
 \widetilde f_\mu(q,t)\notag\\[-2pt]
 &\qquad\times
 C_{\varnothing\lambda\nu^t}(t,q)
 C_{\mu^t\varnothing\nu}(q,t).
\end{align*}
Then
\begin{equation}
\begin{aligned}
 &\PE\left[
 \frac{(t+q)u+\sqrt{tq}\,Q/u}{(1-t)(1-q)}
 \right]
 \iota_0 Z_{2,1}^{\mathrm{fr}}
 \left(-\left(\frac qt\right)^{3/2}Qu^{-2};t^{-1},q;u,1\right)\\
 &\quad=
 \sum_{\lambda,\mu,\nu}
 (-Q)^{|\lambda|+|\mu|+|\nu|}
 \mathsf A_{\lambda\mu\nu}
 \left(\frac uQ\right)^{|\lambda|+|\mu|}
 \mathsf G_{\lambda\mu}\left(\frac Qu\right).
\end{aligned}
\label{eq:direct-ADHM-regrouping}
\end{equation}
\end{lemma}

For each power of \(Q\), the left side is expanded as a Laurent series at \(u=0\), and the right side is grouped by \(Q\)-degree; by Lemma~\ref{lem:two-leg-polynomiality} the right side lies in \(\K[[Q,u]]\), so \eqref{eq:direct-ADHM-regrouping} says in particular that all negative powers of \(u\) on the left cancel.

\begin{proof}
Insert the fixed point expression \eqref{eq:geometric-fixed-point-sum}, apply \eqref{eq:self-Nekrasov-product} to the two diagonal factors and \eqref{eq:cross-Nekrasov-product} to the two cross factors, evaluate the two strip products by \eqref{eq:strip-product}, and expand the two Cauchy products; this gives
\begin{equation}
\begin{aligned}
 \sum_{\lambda,\mu,\nu,\sigma}
 &(-Q)^{|\nu|}(-Q/u)^{|\sigma|}
 (-u)^{|\lambda|+|\mu|}
 f_\nu(q,t)^2
 \widetilde f_\lambda(t,q)\widetilde f_\mu(q,t)\\
 &\quad\times
 C_{\varnothing\lambda\nu^t}(t,q)
 C_{\mu^t\varnothing\nu}(q,t)
 C_{\lambda^t\varnothing\sigma}(t,q)
 C_{\varnothing\mu\sigma^t}(q,t).
\end{aligned}
\label{eq:ADHM-four-partition-sum}
\end{equation}
The factors containing partitions from both framing components give the two skew functions joined along \(\sigma\); the other two Cauchy factors give the vertices joined along \(\nu\).  The monomial part of \eqref{eq:cross-Nekrasov-product}, together with the determinant weight \eqref{eq:det-fixed-weight}, is
\[
 (-Q)^{|\nu|}(-Q/u)^{|\sigma|}
 (-u)^{|\lambda|+|\mu|}
 f_\nu(q,t)^2\widetilde f_\lambda(t,q)
 \widetilde f_\mu(q,t),
\]
the monomial in \eqref{eq:ADHM-four-partition-sum}.

For fixed \(\lambda,\mu,\nu\), the internal sum in \eqref{eq:ADHM-four-partition-sum} is
\[
 \sum_\sigma(-x)^{|\sigma|}
 C_{\lambda^t\varnothing\sigma}(t,q)
 C_{\varnothing\mu\sigma^t}(q,t)
 =\mathsf Z_{\lambda\mu}(x),
 \qquad x=Q/u.
\]
Its empty partition factor is \eqref{eq:empty-two-leg-explicit}, and the two empty one-leg factors give \eqref{eq:empty-one-leg-factors}; moving these three factors to the left produces the plethystic exponential of \eqref{eq:direct-ADHM-regrouping}, and division by the empty two-leg factor turns \(\mathsf Z_{\lambda\mu}\) into \(\mathsf G_{\lambda\mu}\).  The remaining monomials are
\[
 (-Q)^{|\lambda|+|\mu|+|\nu|}
 (u/Q)^{|\lambda|+|\mu|}
 \mathsf A_{\lambda\mu\nu}.
\]
\end{proof}

\begin{theorem}
\label{thm:P2-PT-ADHM}
Assume Conjecture~\ref{conj:P2-vertex-comparison}.  Under \eqref{eq:introduction-variable-change},
\begin{equation}
\begin{aligned}
 Z_{X}^{\mathrm{PT},\mathrm{ref}}(Q;t,q)
 &=\lim_{u\to0}
 \PE\left[
 \frac{\sqrt{tq}\,Q}{(1-t)(1-q)u}
 \right]\\[-1pt]
 &\quad\times
 \iota_0 Z_{2,1}^{\mathrm{fr}}
 \left(-\left(\frac qt\right)^{3/2}Qu^{-2};t^{-1},q;u,1\right).
\end{aligned}
\label{eq:P2-PT-ADHM}
\end{equation}
\end{theorem}

For every coefficient of \(Q\), the product on the right is regular at \(u=0\), and the limit means evaluation there.

\begin{proof}
Write
\[
 \mathsf G_{\lambda\mu}(x)=\sum_{j=0}^m g_jx^j,
 \qquad m=|\lambda|+|\mu|.
\]
The corresponding term on the right of \eqref{eq:direct-ADHM-regrouping} is a sum of monomials proportional to
\[
 Q^{|\nu|+j}u^{m-j}.
\]
For fixed \(Q^D u^e\) we need \(j=D-|\nu|\) and \(m=e+j\), so \(m\leq D+e\) and only finitely many partitions occur, while the constant term in \(u\) has \(j=m\) and \(|\lambda|+|\mu|+|\nu|=D\); taking it commutes with the partition sums, and Lemma~\ref{lem:two-leg-polynomiality} gives \eqref{eq:P2-IK-sum}.  The factor
\[
 \PE\left[\frac{(t+q)u}{(1-t)(1-q)}\right]
\]
tends to \(1\); removing it from the left of \eqref{eq:direct-ADHM-regrouping} gives the left of \eqref{eq:P2-PT-ADHM}.  Conjecture~\ref{conj:P2-vertex-comparison} identifies \eqref{eq:P2-IK-sum} with the stable pair series of \(X\).
\end{proof}

The combinatorial formula in Theorem~\ref{thm:P2-PT-ADHM} is assembled from the rank \(2\) refined vertex/Nekrasov fixed point identity of Taki and Awata--Kanno \cite{TakiRefined,AwataKanno} and the local \(\mathbb F_1\) to local \(\mathbb P^2\) transition of Iqbal--Koz\c{c}az \cite{IqbalKozcaz}.  Kononov--Pi--Shen record the relation in primitive degree, up to a monomial prefactor, in \cite[equation~(21)]{KononovPiShen}.  In the argument above the empty partition factors are carried through the \(u\)-adic specialization rather than discarded.

\section{Blowup equations for local \texorpdfstring{\(\mathbb P^2\)}{P2}}
\label{sec:P2-blowup}

\subsection{The Nakajima--Yoshioka formula with insertion \texorpdfstring{\(\det\mathcal V\)}{det V}}
\label{subsec:NY-P2-level-one}
For the full generating series we use the normalization of \cite{NakajimaYoshiokaPerverse}.  Let \(s_\eps=\eps_1+\eps_2\).  Initially take \(\operatorname{Re}\eps_i>0\) and \(\operatorname{Re}x\) sufficiently large.  The logarithmic series below then converge.  It is their exponentials, not the logarithms themselves, that we subsequently continue meromorphically.  Let \(\Lambda\in\C^*\) be the scale parameter, let \(\Lambda^4=e^{-M}\), and fix
\begin{equation}
 \log(\Lambda^4)=-M+2\pi\mathrm{i},
 \qquad \log\Lambda=\frac{-M+2\pi\mathrm{i}}{4}.
\label{eq:Lambda-continuation-branch}
\end{equation}
Let
\begin{equation}
\begin{aligned}
 \gamma_{\eps_1,\eps_2}(x;\Lambda)
 :={}&
 \frac{-\frac16(x+s_\eps/2)^3+x^2\log\Lambda}
 {2\eps_1\eps_2}\\
 &+\sum_{m\geq1}
 \frac{e^{-mx}}
 {m(e^{m\eps_1}-1)(e^{m\eps_2}-1)},\\
 \widetilde\gamma_{\eps_1,\eps_2}(x;\Lambda)
 :={}&\gamma_{\eps_1,\eps_2}(x;\Lambda)
 +\frac{\pi^2x/6-\zeta(3)}{\eps_1\eps_2}\\
 &+\frac{s_\eps}{2\eps_1\eps_2}
 \left(x\log\Lambda+\frac{\pi^2}{6}\right)
 +\frac{\eps_1^2+\eps_2^2+3\eps_1\eps_2}
 {12\eps_1\eps_2}\log\Lambda.
\end{aligned}
\label{eq:NY-gamma}
\end{equation}

For clarity, denote the polynomial part of \(\widetilde\gamma\) by
\[
 P^{\mathrm{pert}}_{\eps_1,\eps_2}(x;\Lambda)
 :=\widetilde\gamma_{\eps_1,\eps_2}(x;\Lambda)
 -\sum_{m\geq1}\frac{e^{-mx}}
 {m(e^{m\eps_1}-1)(e^{m\eps_2}-1)}.
\]
Thus \(P^{\mathrm{pert}}\) is the explicit polynomial and constant expression in \eqref{eq:NY-gamma}, with \(\log\Lambda\) fixed.  On the initial domain,
\begin{equation}
 e^{-\widetilde\gamma_{\eps_1,\eps_2}(x;\Lambda)}
 =e^{-P^{\mathrm{pert}}_{\eps_1,\eps_2}(x;\Lambda)}
 (e^{-x-\eps_1-\eps_2};e^{-\eps_1},e^{-\eps_2})_\infty.
\label{eq:NY-exponentiated-gamma}
\end{equation}

\begin{definition}
\label{def:negative-root-continuation}
For \(a_1+a_2=0\), write \(a_{12}=a_1-a_2\) and \(a_{21}=-a_{12}\).  We choose opposite parameters
\[
 a_1=-\frac A2,\qquad a_2=\frac A2,\qquad u=e^{-A},
\]
so \(a_{12}=-A=\log u\) and \(a_{21}=A\).  For either root, \(e^{-\widetilde\gamma(\pm A;\Lambda)}\) means the right side of \eqref{eq:NY-exponentiated-gamma}, continued as a meromorphic product with the fixed additive lifts.  Base inversions use \eqref{eq:double-Pochhammer-base-inversion}; the connection formula and analytic hypotheses below specify the continuations used in the limit.  If a value of \(\widetilde\gamma\) itself is needed, choose a simply connected zero- and pole-free domain and the analytic logarithm agreeing with the initial series.  No single-valued meromorphic logarithm is asserted.  The lift \eqref{eq:Lambda-continuation-branch} is used for both roots throughout the specialization at \(u=0\).
\end{definition}

For \(\vec a=(a_1,a_2)\) with \(a_1+a_2=0\), let \(u_\alpha=e^{a_\alpha}\) and \(q_i=e^{-\eps_i}\).  The framed sheaf series is
\begin{equation}
 Z_1^{\mathrm{fr}}(\eps_1,\eps_2,\vec a;\Lambda)
 =\sum_{N\geq0}
 \bigl(\Lambda^4(q_1q_2)^{3/2}\bigr)^N
 \chi_{\mathbf T_{\mathrm{fr}}}^{\mathrm{loc}}
 \bigl(\M(2,N),\det\cV\bigr),
\label{eq:NY-ADHM-series}
\end{equation}
The full rank \(2\) series is
\begin{equation}
\begin{aligned}
 Z_1(\eps_1,\eps_2,\vec a;\Lambda)
 :={}&
 \exp\left[
 -\sum_{\alpha\ne\beta}
 \widetilde\gamma_{\eps_1,\eps_2}
 (a_\alpha-a_\beta;\Lambda)
 -\sum_{\alpha=1}^2
 \frac{a_\alpha^3}{6\eps_1\eps_2}\right]\\
 &\times Z_1^{\mathrm{fr}}
 (\eps_1,\eps_2,\vec a;\Lambda).
\end{aligned}
\label{eq:NY-full-partition}
\end{equation}

\begin{lemma}
\label{lem:NY-determinant-is-ADHM}
Let \(\ell_\infty\subset\mathbb P^2\) be the framing line, let \(\mathcal E\) be a family of framed sheaves on \(\mathbb P^2\times\M(2,N)\), and let \(\pi_2\) be the projection to \(\M(2,N)\).  Write
\begin{equation}
 \mathcal V_0
 =R^1\pi_{2*}\bigl(\mathcal E(-\ell_\infty)\bigr)
 \cong\cV.
\label{eq:NY-V0}
\end{equation}
This is the bundle denoted \(\mathcal V\) on \(\mathbb P^2\) in \cite{NakajimaYoshiokaPerverse}; on the blowup their notation \(\mathcal V_0\) uses the same cohomological definition.  Consequently, for every \(c\in\Z\), the insertion with exponent \(c\) is \(\ch((\det\cV)^c)\).
\end{lemma}

\begin{proof}
For a fibre \(E\) of \(\mathcal E\), the monad description gives
\[
 H^0(E(-\ell_\infty))=H^2(E(-\ell_\infty))=0,
 \qquad H^1(E(-\ell_\infty))\cong V.
\]
The relative monad gives the same in families, so \(R^1\pi_{2*}(\mathcal E(-\ell_\infty))\) is the tautological bundle of the principal \(GL(V)\)-bundle in \eqref{eq:ADHM-quotient}.  For determinant exponent \(c\), the integrand of Nakajima--Yoshioka is \(\operatorname{td}(\M(2,N))\exp(c\,c_1(\mathcal V_0))\) \cite{NakajimaYoshiokaPerverse}, and equivariant Hirzebruch--Riemann--Roch turns it into the localized Euler characteristic of \((\det\cV)^c\).
\end{proof}

Write the framing weights symmetrically,
\[
 q_i=e^{-\eps_i},\qquad u=v^2,\qquad
 (e^{a_1},e^{a_2})=(v,v^{-1}),
\]
and define \(\mathfrak p\) by
\begin{equation}
 \Lambda^4(q_1q_2)^{3/2}=v\mathfrak p.
\label{eq:NY-p-variable}
\end{equation}
Multiplying both framing weights by \(v\) changes them from \((v,v^{-1})\) to \((u,1)\); it leaves the tangent character unchanged and multiplies \(\det\cV\) on \(\M(2,N)\) by \(v^N\).  Let
\begin{equation}
\begin{aligned}
 \Lambda_{\mathrm c}(\mathfrak p,u;q_1,q_2)
 &:={}
 \left(\frac{u^{1/2}\mathfrak p}{(q_1q_2)^{3/2}}\right)^{1/4},\\
 \mathscr Z(\mathfrak p,u;q_1,q_2)
 &:={}
 Z_1\left(
 -\log q_1,-\log q_2,
 \left(\tfrac12\log u,-\tfrac12\log u\right);
 \Lambda_{\mathrm c}(\mathfrak p,u;q_1,q_2)
 \right).
\end{aligned}
\label{eq:NY-mathscrZ-definition}
\end{equation}
We regard \(\eps_i\) and \(\log u\) as formal additive parameters, or fix analytic branches, and choose the fourth root in \(\Lambda_{\mathrm c}\) once and for all.  By the change of framing, the framed sheaf series of \(\mathscr Z(\mathfrak p,u;q_1,q_2)\) is
\begin{equation}
 \sum_{N\geq0}\mathfrak p^N
 \chi_{\mathbf T_{\mathrm{fr}}}^{\mathrm{loc}}\bigl(\M(2,N),\det\cV\bigr)
\label{eq:NY-centered-ADHM}
\end{equation}
with framing weights \((u,1)\).

\begin{theorem}
\label{thm:NY-native}
For \(k\in\{0,1\}\) and \(d_{\mathrm{NY}}\in\Z_{\geq0}\), let
\[
 c_{k,d_{\mathrm{NY}}}
 =\frac{k^3-k-2d_{\mathrm{NY}}+2}{24},\qquad
 \ell_n=(n-k/2,-n+k/2),
\]
and
\begin{equation}
\begin{aligned}
 \widehat Z_{1,k,d_{\mathrm{NY}}}
 (\eps_1,\eps_2,\vec a;\Lambda)
 ={}&e^{c_{k,d_{\mathrm{NY}}}s_\eps}
 \sum_{n\in\Z}
 Z_1\!\left(
 \begin{gathered}
 \eps_1,\eps_2-\eps_1,\vec a+\eps_1\ell_n;\\[-2pt]
 \Lambda e^{\eps_1(2d_{\mathrm{NY}}+k-3)/8}
 \end{gathered}\right)\\
 &\hspace{22mm}\times
 Z_1\!\left(
 \begin{gathered}
 \eps_1-\eps_2,\eps_2,\vec a+\eps_2\ell_n;\\[-2pt]
 \Lambda e^{\eps_2(2d_{\mathrm{NY}}+k-3)/8}
 \end{gathered}\right).
\end{aligned}
\label{eq:NY-native-additive}
\end{equation}
Then
\begin{equation}
\begin{aligned}
 \widehat Z_{1,0,d_{\mathrm{NY}}}
 (\eps_1,\eps_2,\vec a;\Lambda)
 &=Z_1(\eps_1,\eps_2,\vec a;\Lambda),
 &&0\leq d_{\mathrm{NY}}\leq2,\\
 \widehat Z_{1,1,1}(\eps_1,\eps_2,\vec a;\Lambda)&=0.
\end{aligned}
\label{eq:NY-unity-vanishing}
\end{equation}
\end{theorem}

\begin{proof}
The \(K\)-theoretic blowup formula of Nakajima--Yoshioka \cite{NakajimaYoshiokaII,NakajimaYoshiokaPerverse} is the Atiyah--Bott--Lefschetz formula on the framed sheaf moduli space of the blowup of \(\mathbb P^2\), whose fixed loci are indexed by \(\ell_n\) and by two framed sheaf fixed loci at the two charts.  In their notation, take rank \(2\) and determinant power \(1\), and write \(\ell=(\ell_1,\ell_2)\) for a lattice vector.  The permitted vectors satisfy
\[
 \ell_1+\ell_2=0,\qquad
 \ell_\alpha\equiv-\frac k2\pmod{\Z},
\]
so \(\ell=\ell_n=(n-k/2,-n+k/2)\).  Their exponent multiplying \(\Lambda\) becomes
\[
 \frac14\left(d_{\mathrm{NY}}-\frac12+\frac k2-1\right)
 =\frac{2d_{\mathrm{NY}}+k-3}{8},
\]
which gives the two shifts in \eqref{eq:NY-native-additive}, and their scalar exponent becomes
\[
 -\frac{(4(d_{\mathrm{NY}}-\frac12+\frac k2)-2)}{48}
 +\frac{k^3}{24}=c_{k,d_{\mathrm{NY}}},
\]
since
\[
 -\frac{4d_{\mathrm{NY}}-4+2k}{48}+\frac{k^3}{24}
 =\frac{k^3-k-2d_{\mathrm{NY}}+2}{24}.
\]
Equation~(3.8) of \cite{NakajimaYoshiokaPerverse}, obtained from their Theorem~2.11(1) and equation~(3.6), identifies the blowup series with the framed sheaf series on \(\mathbb P^2\) for \(k=0\) and \(0\leq d_{\mathrm{NY}}\leq2\).  In their Theorem~2.11(2) and equation~(3.6), let \(C\) be the exceptional curve of \(\operatorname{Bl}_0\mathbb P^2\), let \(E\) be the framed sheaf in their notation, and take rank \(2\), their parameter \(a=1\), determinant power \(1\), \((c_1(E),[C])=1\) and \(d=1\).  The resulting blowup series is \(0\).  Let \(E'=E\otimes\cO(C)\); then \(\mathcal V_1(E)=\mathcal V_0(E')\), so the determinant lines agree, while
\[
 (c_1(E'),[C])=(c_1(E)+2[C],[C])=1+2C^2=-1.
\]
This is the \(k=1\) convention of \eqref{eq:NY-native-additive}.  Lemma~\ref{lem:NY-determinant-is-ADHM} identifies the inserted line with \(\det\cV\).
\end{proof}

In multiplicative form: for
\[
 (k,d_{\mathrm{NY}})\in
 \{(0,0),(0,1),(0,2),(1,1)\},
\]
and in the notation of \eqref{eq:NY-mathscrZ-definition},
\begin{equation}
\begin{aligned}
 &\sum_{n\in\Z}
 \frac{
 \mathscr Z(\mathfrak p q_1^{n-d_{\mathrm{NY}}-k},u q_1^{k-2n};
 q_1,q_2/q_1)}
 {\mathscr Z(\mathfrak p,u;q_1,q_2)}\\[-1mm]
 &\hspace{18mm}\times
 \mathscr Z(\mathfrak p q_2^{n-d_{\mathrm{NY}}-k},u q_2^{k-2n};
 q_1/q_2,q_2)\\
 &\qquad=
 \begin{cases}
 (q_1q_2)^{(1-d_{\mathrm{NY}})/12},
   &k=0,\ d_{\mathrm{NY}}=0,1,2,\\
 0,&(k,d_{\mathrm{NY}})=(1,1).
 \end{cases}
\end{aligned}
\label{eq:NY-multiplicative}
\end{equation}
Indeed, by \eqref{eq:NY-p-variable} on each chart, the shifts of \(\Lambda\) and \(\vec a\) in \eqref{eq:NY-native-additive} become
\[
 \mathfrak p_i=\mathfrak p q_i^{n-d_{\mathrm{NY}}-k},\qquad
 u_i=uq_i^{k-2n}.
\]
Divide \eqref{eq:NY-native-additive} by \(Z_1\) and use \eqref{eq:NY-unity-vanishing}; for \(k=0\), \(e^{-c_{0,d_{\mathrm{NY}}}s_\eps} =(q_1q_2)^{(1-d_{\mathrm{NY}})/12}\).

\subsection{Blowup equations}
\label{subsec:P2-equations-recursion}
If \(H\) is the class of a line, then \(-H\cdot[\mathbb P^2]=3\) in the total space.  The lattice sign in the equations below is derived from the continued degree \(0\) factors in Lemma~\ref{lem:two-chart-factor}.

Let
\begin{equation}
 \mathcal Z(Q;t,q)
 :=Z_X^{\PT,\mathrm{ref}}(-Q;t,q)
 =\sum_{d\geq0}A_d(t,q)Q^d,
 \qquad A_0=1,
\label{eq:P2-A-coefficients}
\end{equation}
and
\begin{equation}
 F_{\mathbb P^2}^{(0)}(T;\alpha,\beta)
 =\frac{T^3}{18\alpha\beta}
 -\frac{\pi^2T}{6\alpha\beta}+\frac{T}{12}
 +\frac{(\alpha+\beta)^2T}{24\alpha\beta}.
\label{eq:P2-degree-zero-polynomial}
\end{equation}
The normalized series is
\begin{equation}
 \widehat Z_X(T;\alpha,\beta)
 =e^{F_{\mathbb P^2}^{(0)}(T;\alpha,\beta)}
 \mathcal Z(e^{-T};e^\alpha,e^{-\beta}).
\label{eq:HSW-full-partition}
\end{equation}
For \(r=\pm1\), let
\begin{equation}
\begin{aligned}
 R_{r,n}&=3n+\frac r2,&
 m_r(n)&=\frac{n(3n+r)}2,\\
 \gamma_r(n)&=-\frac32n^3-\frac{3r}{4}n^2+\frac14n,
\end{aligned}
\label{eq:P2-unity-exponents}
\end{equation}
and for \(r=3\),
\begin{equation}
 R_{3,n}=3n+\frac32,\qquad
 m_3(n)=\frac{3n(n+1)}2,\qquad
 \gamma_3(n)=-\frac34n(n+1)(2n+1).
\label{eq:P2-vanishing-exponents}
\end{equation}
The multiplicative and additive equivariant variables are related by
\begin{equation}
 t=e^{\eps_1},\qquad q=e^{-\eps_2},\qquad
 q_1=e^{-\eps_1}=t^{-1},\qquad q_2=e^{-\eps_2}=q.
\label{eq:P2-blowup-variable-change}
\end{equation}
Let \(\mathfrak a=q_1q_2\), and write the framed sheaf variable recording \(c_2\) as
\begin{equation}
 \mathfrak p=-\mathfrak a^{3/2}z.
\label{eq:NY-z-variable}
\end{equation}
In the sign convention of \eqref{eq:P2-A-coefficients}, the comparison of Theorem~\ref{thm:P2-PT-ADHM} gives
\begin{equation}
 Q=-zu^2.
\label{eq:P2-plane-variable}
\end{equation}

\begin{lemma}
\label{lem:NY-chart-substitutions}
Take \((k,d_{\mathrm{NY}})\in\{(0,1),(0,2),(1,1)\}\) and replace \(n\) by \(-n\) in \eqref{eq:NY-multiplicative}.  On chart \(i=1,2\),
\begin{equation}
 \mathfrak p_i=\mathfrak p q_i^{-n-d_{\mathrm{NY}}-k},\qquad
 u_i=uq_i^{2n+k},\qquad
 \mathfrak a_i=\frac{\mathfrak a}{q_i},\qquad
 z_i=zq_i^{-n-d_{\mathrm{NY}}-k+3/2}.
\label{eq:NY-chart-substitutions}
\end{equation}
With \(Q_i=-z_i u_i^2\),
\begin{equation}
 \frac{Q_i}{Q}
 =q_i^{3n-d_{\mathrm{NY}}+k+3/2}
 =q_i^{R_{r,n}},
 \qquad
 R_{r,n}=3n+\frac r2,\qquad r=3-2d_{\mathrm{NY}}+2k.
\label{eq:NY-plane-shift}
\end{equation}
The three cases give \(r=1,-1,3\) respectively.
\end{lemma}

\begin{proof}
The first two equations are \eqref{eq:NY-multiplicative} after \(n\mapsto-n\).  The product of the two coordinate characters on chart \(i\) is \(\mathfrak a_i=\mathfrak a/q_i\), so \eqref{eq:NY-z-variable} gives
\[
 \frac{z_i}{z}
 =\frac{\mathfrak p_i}{\mathfrak p}
  \left(\frac{\mathfrak a}{\mathfrak a_i}\right)^{3/2}
 =q_i^{-n-d_{\mathrm{NY}}-k+3/2}.
\]
The sign in \eqref{eq:P2-plane-variable} cancels in the ratio, and \(Q_i/Q=z_iu_i^2/(zu^2)\) gives
\[
 \frac{Q_i}{Q}
 =q_i^{-n-d_{\mathrm{NY}}-k+3/2+4n+2k}
 =q_i^{3n-d_{\mathrm{NY}}+k+3/2}.
\]
\end{proof}

With the framing characters \((v,v^{-1})\), let
\begin{equation}
 u=v^2=e^{-A},\qquad Q=-zu^2=e^{-T}.
\label{eq:centered-variables}
\end{equation}
Together with \(\Lambda^4=e^{-M}\) and the branch \eqref{eq:Lambda-continuation-branch}, equations \eqref{eq:NY-p-variable} and \eqref{eq:NY-z-variable} give
\begin{equation}
 \Lambda^4=-vz,\qquad T=M+\frac32A.
\label{eq:centered-parameter-equality}
\end{equation}
The parameter \(M\) in \(\mathscr Z\) satisfies
\begin{equation}
 M=T-\frac32A.
\label{eq:u-zero-parameter-path}
\end{equation}
The coefficientwise specialization at \(u=0\) keeps \(Q=e^{-T}\) fixed, so \(M\) moves along \eqref{eq:u-zero-parameter-path} as \(A\to\infty\).

After replacing the Novikov coordinate by its negative, let
\begin{equation}
\begin{aligned}
 \mathscr A_{\mathrm c}(Q,u;t,q):={}&
 \exp\left[
 \sum_{m\geq1}
 \frac{(t^m+q^m)u^m+(-1)^m(tq)^{m/2}Q^mu^{-m}}
 {m(1-t^m)(1-q^m)}\right]\\
 &\times\iota_0 Z_{2,1}^{\mathrm{fr}}
 \left(\left(\frac qt\right)^{3/2}Qu^{-2};
 t^{-1},q;u,1\right).
\end{aligned}
\label{eq:normalized-u-series}
\end{equation}
\begin{lemma}
\label{lem:u-zero-regularity}
Assume Conjecture~\ref{conj:P2-vertex-comparison}.  Then \(\mathscr A_{\mathrm c}\in\K[[Q,u]]\), and
\begin{equation}
 \mathscr A_{\mathrm c}(Q,0;t,q)=\mathcal Z(Q;t,q).
\label{eq:u-zero-constant-term}
\end{equation}
\end{lemma}

\begin{proof}
Substitute \(-Q\) for \(Q\) in Lemma~\ref{lem:direct-ADHM-regrouping}.  If \(m=|\lambda|+|\mu|\) and \(\mathsf G_{\lambda\mu}(x)=\sum_{h=0}^m g_hx^h\), a summand is an element of \(\K\) times
\[
 Q^{|\nu|+h}u^{m-h}.
\]
Both exponents are nonnegative, and for fixed exponents the sizes of \(\lambda,\mu,\nu\) are bounded, so the sum is coefficientwise finite, while the constant term in \(u\) has \(h=m\); by Lemma~\ref{lem:two-leg-polynomiality} and Conjecture~\ref{conj:P2-vertex-comparison} it is \(Z_{X}^{\PT,\mathrm{ref}}(-Q;t,q)=\mathcal Z(Q;t,q)\).
\end{proof}

The analytic passage to \(u=0\) uses Conjecture~\ref{conj:P2-asymptotic-resummation} and the support estimates of Appendix~\ref{app:u-zero}.

\begin{lemma}
\label{lem:P2-perturbative-shift}
For any \(R\),
\begin{equation}
\begin{aligned}
 &F_{\mathbb P^2}^{(0)}
 (T+\eps_1R;\eps_1,\eps_2-\eps_1)
 +F_{\mathbb P^2}^{(0)}
 (T+\eps_2R;\eps_1-\eps_2,\eps_2)\\
 &\qquad-F_{\mathbb P^2}^{(0)}
 (T;\eps_1,\eps_2)\\
 &=-\left(\frac{R^2}{6}-\frac1{24}\right)T
 +(\eps_1+\eps_2)\left(-\frac{R^3}{18}+\frac R8\right).
\end{aligned}
\label{eq:P2-perturbative-shift}
\end{equation}
\end{lemma}

\begin{proof}
Substitute \eqref{eq:P2-degree-zero-polynomial}.  The coefficients of \(T^3\), \(T^2\) and \(\pi^2T\) cancel; the coefficient of \(T\) and the constant term are as displayed.
\end{proof}

\subsection{Continuation of the degree \(0\) factors}
Let \(D=\eps_1\eps_2\) and \(s=\eps_1+\eps_2\).  Define the polynomials
\begin{equation}
 P_{\pm A}(A,M;\eps_1,\eps_2)
 =\frac1D\left(
 \frac{A^3}{6}+\frac{MA^2}{4}-\frac{\pi^2A}{3}
 +\frac{\eps_1^2+\eps_2^2+3D}{12}A\right).
\label{eq:root-polynomial}
\end{equation}
The part depending only on \(M\) is
\begin{equation}
 P_M(M;\eps_1,\eps_2)
 =\frac1D\left(
 \frac{M^3}{9}
 +\left(\frac{\pi^2}{3}-\frac{s^2}{12}\right)M\right).
\label{eq:M-polynomial}
\end{equation}
For the empty partition factor, define
\begin{equation}
 W(Y)=-\frac1D\left(
 \frac{Y^3}{6}+
 \left(\frac{\pi^2}{6}-\frac{\eps_1^2+\eps_2^2}{24}\right)Y\right),
 \qquad -zu=e^Y,
\label{eq:empty-two-leg-polynomial}
\end{equation}
where \(Y\) is an additive logarithmic variable.  Since \(Q=-zu^2=e^{-T}\) and \(u=e^{-A}\),
\begin{equation}
 Y=-M-\frac A2,\qquad T=M+\frac32A=A-Y.
\label{eq:Y-A-M-T}
\end{equation}
With \(Y\) and \(T\) as in \eqref{eq:Y-A-M-T},
\begin{equation}
 P_{\pm A}(A,M;\eps_1,\eps_2)+W(Y)
 =F_{\mathbb P^2}^{(0)}(T;\eps_1,\eps_2)
  +P_M(M;\eps_1,\eps_2).
\label{eq:polynomial-decomposition}
\end{equation}
Indeed, \(A=-2M-2Y\) and \(T=-2M-3Y\), so
\[
 \frac{A^3}{6}+\frac{MA^2}{4}-\frac{Y^3}{6}
 =\frac{T^3}{18}+\frac{M^3}{9}.
\]
Substituting in the remaining terms gives
\[
 -\frac{\pi^2T}{6}+\frac{DT}{12}
 +\frac{s^2T}{24}
 +\left(\frac{\pi^2}{3}-\frac{s^2}{12}\right)M,
\]
and after division by \(D\) these are the linear terms of \eqref{eq:P2-degree-zero-polynomial} and \eqref{eq:M-polynomial}.  For additive parameters \(\alpha,\beta\), let
\begin{equation}
\begin{aligned}
 G_{\alpha,\beta}(A,M)
 :={}&\frac{2\zeta(3)}{\alpha\beta}
 -\frac{\pi\mathrm{i}(\alpha+\beta)A}{2\alpha\beta}\\
 &+\frac{\alpha^2+\beta^2+3\alpha\beta}
 {48\alpha\beta}(2\pi\mathrm{i}+M)
 +\frac{\alpha^3+\alpha^2\beta+\alpha\beta^2+\beta^3}
 {48\alpha\beta}.
\end{aligned}
\label{eq:gamma-normalization}
\end{equation}
We use the exponentiated continuation of Definition~\ref{def:negative-root-continuation}.  The following connection identity uses the double Pochhammer convention \eqref{eq:global-double-Pochhammer}.

\begin{lemma}
\label{lem:triple-sine-continuation}
Let
\[
\begin{aligned}
 \omega_1&=\eps_1,& \omega_2&=-\eps_2,&
 \omega_3&=2\pi\mathrm{i},\\
 \omega_{\Sigma}&=\omega_1+\omega_2+\omega_3,&
 z_Y&=\omega_{\Sigma}/2+Y,&&
\end{aligned}
\]
and \(\omega=(\omega_1,\omega_2,\omega_3)\).  Suppose first that the periods lie in a common open half-plane with boundary through the origin and that \(\operatorname{Im}(\omega_j/\omega_k)\ne0\) for all \(j\ne k\).  For
\[
 x_j=e^{2\pi\mathrm{i}z_Y/\omega_j},\qquad
 q_{kj}=e^{2\pi\mathrm{i}\omega_k/\omega_j},
\]
let \((x_j;(q_{kj})_{k\ne j})_\infty\) denote the double Pochhammer \((x_j;q_{k_1j},q_{k_2j})_\infty\), \(\{k_1,k_2\}=\{1,2,3\}\setminus\{j\}\), and let
\begin{equation}
 \Phi(Y)
 :=\prod_{j=1}^2
 \frac{(x_j^{-1};(q_{kj}^{-1})_{k\ne j})_\infty}
 {(x_j;(q_{kj})_{k\ne j})_\infty}.
\label{eq:modular-connection-factor}
\end{equation}
Then
\begin{equation}
 \frac{\displaystyle\prod_{j=1}^3
 (x_j;(q_{kj})_{k\ne j})_\infty}
 {\displaystyle\prod_{j=1}^3
 (x_j^{-1};(q_{kj}^{-1})_{k\ne j})_\infty}
 =e^{W(Y)},
\label{eq:triple-sine-connection}
\end{equation}
with \(W\) as in \eqref{eq:empty-two-leg-polynomial}.  If \(t=e^{\eps_1}\), \(q=e^{-\eps_2}\) and \(Y=\log(Q/u)\), the third numerator factor is
\begin{equation}
 \left(-\sqrt{tq}\,\frac Qu;t,q\right)_\infty
 =\exp\left[
 -\sum_{m\geq1}
 \frac{(-1)^m(tq)^{m/2}Q^mu^{-m}}
 {m(1-t^m)(1-q^m)}\right],
\label{eq:outgoing-empty-product}
\end{equation}
and the third denominator factor, in the bases \((t,q)\), is
\begin{equation}
 \left(-\sqrt{tq}\,\frac uQ;t,q\right)_\infty.
\label{eq:incoming-empty-product}
\end{equation}
Equivalently, for \(Y\) in the same branch,
\begin{equation}
 (-\sqrt{tq}\,e^Y;t,q)_\infty
 =e^{W(Y)}(-\sqrt{tq}\,e^{-Y};t,q)_\infty
  \Phi(Y).
\label{eq:exact-double-Pochhammer-connection}
\end{equation}
\end{lemma}

The identity holds meromorphically in \(Y\) on this initial period domain and persists on compatible continuations where they exist, with \(\omega_3=2\pi\mathrm i\) fixed.  Its use on the common domains for the three chart functions is part of the analytic hypotheses below; the lemma alone does not assert such a common continuation.

\begin{proof}
Let \(S_3(z\mid\omega_1,\omega_2,\omega_3)\) be the Barnes triple sine in Narukawa's convention, with multiple Bernoulli polynomial \(B_{33}\); on the open set of the statement, Narukawa's hypothesis \(\operatorname{Im}(\omega_j/\omega_k)\ne0\) holds, and his two product formulas \cite[Section~4]{Narukawa} give
\[
 S_3(z_Y\mid\omega)
 =e^{-\pi\mathrm{i}B_{33}(z_Y\mid\omega)/6}
   \prod_{j=1}^3(x_j;(q_{kj})_{k\ne j})_\infty
\]
and
\[
 S_3(z_Y\mid\omega)
 =e^{\pi\mathrm{i}B_{33}(z_Y\mid\omega)/6}
   \prod_{j=1}^3
   (x_j^{-1};(q_{kj}^{-1})_{k\ne j})_\infty.
\]
Narukawa's construction gives the meromorphic continuation of the products, so the quotient gives \eqref{eq:triple-sine-connection} as a meromorphic identity in \((\eps_1,\eps_2)\).  At \(z_Y=\omega_{\Sigma}/2+Y\),
\[
 B_{33}(\omega_{\Sigma}/2+Y\mid\omega)
 =\frac{Y^3-\frac14
  (\omega_1^2+\omega_2^2+\omega_3^2)Y}
 {\omega_1\omega_2\omega_3},
\]
and since \(\omega_1\omega_2\omega_3=-2\pi\mathrm{i}\eps_1\eps_2\) \(\pi\mathrm{i}B_{33}/3=W(Y)\).  Moreover
\[
 x_3=e^{z_Y}=-\sqrt{tq}\,e^Y
 =-\sqrt{tq}\,Q/u,
\]
which proves \eqref{eq:outgoing-empty-product}.  The rule for two inverted bases gives
\[
 (x_3^{-1};t^{-1},q^{-1})_\infty
 =(tqx_3^{-1};t,q)_\infty
 =\left(-\sqrt{tq}\,\frac uQ;t,q\right)_\infty.
\]
Solving \eqref{eq:triple-sine-connection} for the third numerator factor gives \eqref{eq:exact-double-Pochhammer-connection}.
\end{proof}

Define the linear \(M\)-term
\begin{equation}
\begin{aligned}
 L_{\eps_1,\eps_2}(M)
 &=\frac{\eps_1^2+\eps_2^2+3D}{48D}M,\\
 \widetilde P_M(M;\eps_1,\eps_2)
 &=P_M(M;\eps_1,\eps_2)
  +L_{\eps_1,\eps_2}(M).
\end{aligned}
\label{eq:NY-linear-M-term}
\end{equation}

Let
\begin{equation}
 Y_{\mathrm{edge}}=A-T=-M-A/2,
 \qquad
 Y_{\mathrm{root}}=A-\frac{\eps_1+\eps_2}{2}+\pi\mathrm{i}.
\label{eq:edge-root-arguments}
\end{equation}
These are the edge and negative-root arguments.  Define
\begin{equation}
 \mathcal R(A,T;\eps_1,\eps_2)
 =\left(-\sqrt{tq}\,\frac uQ;t,q\right)_\infty
   \Phi(Y_{\mathrm{edge}})\Phi(Y_{\mathrm{root}})^{-1}.
\label{eq:remaining-product-factor}
\end{equation}
Here \(t=e^{\eps_1}\), \(q=e^{-\eps_2}\), and \(\Phi\) uses the displayed equivariant parameters.  At a chart value all of these parameters are substituted simultaneously.

The analytic hypotheses below ask for representations of \(\mathcal R\) and its chart translates as finite products of factors
\begin{equation}
 \left(a_\rho e^{-\lambda_\rho A}Q^{d_\rho};
       r_\rho,s_\rho\right)_\infty^{\epsilon_\rho},
 \qquad \epsilon_\rho\in\{1,-1\},
\label{eq:angular-double-Pochhammer-factor}
\end{equation}
where \(a_\rho,r_\rho,s_\rho\) are meromorphic in \((\eps_1,\eps_2)\) and independent of \(A,T,Q\).  The coefficients \(a_\rho\) of a chart translate may also depend on its lattice index.  Existence of representations with only decaying arguments is an assumption, not a consequence of base inversion or of the connection identity.

For the unshifted factor and the two chart factors, the first two triple sine periods are
\[
 (\eps_1,-\eps_2),\qquad
 (\eps_1,\eps_1-\eps_2),\qquad
 (\eps_1-\eps_2,-\eps_2)
\]
respectively.  The permitted generating pairs in \eqref{eq:angular-double-Pochhammer-factor} form the finite set
\begin{equation}
 \begin{aligned}
 \mathcal P_{\mathrm{exp}}={}&\{(1,-1)\}\\
 &\cup\left\{
 \left(-\sigma\frac{2\pi\mathrm{i}}{\omega},
        \sigma\frac{2\pi\mathrm{i}}{\omega}\right),
 \left(-\sigma\frac{2\pi\mathrm{i}}{\omega},0\right):
 \omega\in\{\eps_1,\eps_1-\eps_2,-\eps_2\},\quad
 \sigma=\pm1\right\}.
 \end{aligned}
\label{eq:finite-exponent-pairs}
\end{equation}
A factor with period \(\omega\) and argument \(Y_{\mathrm{edge}}=A-T\) is a meromorphic coefficient times
\[
 e^{\sigma 2\pi\mathrm{i}A/\omega}
 Q^{\sigma 2\pi\mathrm{i}/\omega},
\]
while the same factor with argument \(Y_{\mathrm{root}}\) has no power of \(Q\).  Since
\begin{equation}
 \log\left(ae^{-\lambda A}Q^d;r,s\right)_\infty^\epsilon
 =-\epsilon\sum_{m\geq1}
 \frac{a^m e^{-m\lambda A}Q^{md}}
 {m(1-r^m)(1-s^m)},
\label{eq:double-Pochhammer-log-expansion}
\end{equation}
all exponent pairs in one factor are positive integral multiples of one member of \(\mathcal P_{\mathrm{exp}}\), and the number \(N\) of distinct generating pairs in such a representation is at most \(13\).

\subsection{Analytic hypotheses}
Fix the lift \eqref{eq:Lambda-continuation-branch} and one of
\[
 (k,d_{\mathrm{NY}})=(0,1),\ (0,2),\ (1,1),
 \qquad r=3-2d_{\mathrm{NY}}+2k.
\]
For a function \(f(A,T;\alpha,\beta)\), define its two chart substitutions, for every \(n\in\Z\), by
\begin{equation}
\begin{aligned}
 \tau_{1,n}f&=
 f\bigl(A+\eps_1(2n+k),T+\eps_1R_{r,n};
        \eps_1,\eps_2-\eps_1\bigr),\\
 \tau_{2,n}f&=
 f\bigl(A+\eps_2(2n+k),T+\eps_2R_{r,n};
        \eps_1-\eps_2,\eps_2\bigr).
\end{aligned}
\label{eq:analytic-chart-operators}
\end{equation}
In particular \(M=T-3A/2\) changes to
\(M+\eps_i(3-2d_{\mathrm{NY}}-k)/2\).
We apply \(\tau_{i,n}\) to \(\mathcal R\) and to
\(\mathscr A_{\mathrm c}(e^{-T},e^{-A};e^\alpha,e^{-\beta})\).
The latter series will retain its \(A\)-independent terms in \(Q\).

We say that the \emph{analytic hypotheses} hold for this pair if there exist a nonempty connected open set \(\mathscr U\subset\C^3\) with coordinates \((T,\eps_1,\eps_2)\), an open interval \(I^\circ\) on which \(\operatorname{Re}(e^{\mathrm i\theta})>0\), and \(R_0>0\), with
\[
 \mathscr D_\infty=
 \{(A,T,\eps_1,\eps_2):(T,\eps_1,\eps_2)\in\mathscr U,\
 |A|>R_0,\ \arg A\in I^\circ\},
\]
such that the following conditions hold.
\begin{enumerate}
\item[(A1)] The exponentiated degree \(0\) products of \(Z_1\) in \eqref{eq:NY-full-partition}, its unshifted framed sheaf series, and both chart substitutions for every \(n\in\Z\) have chosen analytic continuations from nonempty initial convergence domains to a common simply connected pole-free domain containing \(\mathscr D_\infty\).  These continuations keep the specified lifts, satisfy \eqref{eq:exact-double-Pochhammer-connection}, and agree with the original germs on their initial domains.

\item[(A2)] Fix one finite product representation \eqref{eq:angular-double-Pochhammer-factor} for \(\mathcal R\) at each of the three equivariant parameter pairs.  For every \(n\), its translated representation is obtained by \eqref{eq:analytic-chart-operators}, without changing the generating exponent pairs.  All generating pairs belong to \(\mathcal P_{\mathrm{exp}}\).  These representations are exact after the polynomial factors already displayed have been extracted; no further \(A\)- or \(T\)-dependent prefactor is suppressed.  In particular their formal constant term is \(1\).

\item[(A3)] For every generating pair \((\lambda_\rho,d_\rho)\) actually used in these representations, every compact \(K\subset\mathscr U\), and every closed interval \(J\subset I^\circ\),
\begin{equation}
 \inf_{\substack{(T,\eps_1,\eps_2)\in K\\ \theta\in J}}
 \operatorname{Re}(\lambda_\rho e^{\mathrm i\theta})>0.
\label{eq:P2-asymptotic-decay-condition}
\end{equation}
This condition applies to the nonconstant generators of the remaining products, not to their unit terms or to the \(A\)-independent \(Q\)-series in \(\mathscr A_{\mathrm c}\).  The bound uses the same finite list for all \(n\).

\item[(A4)] The product expansions in (A2), the regrouped \(Q,u\)-series \(\mathscr A_{\mathrm c}\) and their chart substitutions converge normally on \(\mathscr D_\infty\).  The same holds for the normalized bilateral sum
\begin{equation}
 \sum_{n\in\Z}(-1)^n Q^{m_r(n)}
 e^{(\eps_1+\eps_2)\gamma_r(n)}
 \prod_{i=1}^2
 \bigl(\tau_{i,n}\mathscr A_{\mathrm c}\bigr)
 \bigl(\tau_{i,n}\mathcal R\bigr).
\label{eq:analytic-normalized-lattice-sum}
\end{equation}
For each \(K,J\) as above, normal convergence here includes a summable majorant uniform for \((T,\eps_1,\eps_2)\in K\) and \(A=\rho e^{\mathrm i\theta}\), \(\rho\geq R_{K,J}\), \(\theta\in J\), for some \(R_{K,J}\geq R_0\).  This is asserted for the displayed normalized expressions, not for unremoved polynomial exponentials.

\item[(A5)] Applying the connection identity to each instanton coefficient, continuing it as in (A1), and regrouping it into the expressions of (A2)--(A4) commutes with summing the instanton and lattice series.  Thus the continued Nakajima--Yoshioka identities agree with these expansions after removal of their explicit common factors.  The expansions are unique as convergent series in \(e^{-\lambda A}Q^d\), with \(\operatorname{Log}Q=-T\), and their coefficient identities extend meromorphically in the equivariant parameters.
\end{enumerate}

\begin{conjecture}
\label{conj:P2-asymptotic-resummation}
The analytic hypotheses \textup{(A1)--(A5)} hold for each of
\[
 (k,d_{\mathrm{NY}})=(0,1),\ (0,2),\ (1,1).
\]
\end{conjecture}

The conjecture includes existence of the positive-exponent representations and compatibility with the entire bilateral sum.  Neither is proved by the degree \(0\) connection formula.  The support argument in Appendix~\ref{app:u-zero} proves the required formal finiteness once these hypotheses are granted.

\subsection{The three equations}
For \(s=\eps_1+\eps_2\), let
\begin{equation}
 \Lambda_{X,1}=e^{s/18},\qquad
 \Lambda_{X,-1}=e^{-s/18},\qquad
 \Lambda_{X,3}=0.
\label{eq:X-bilinear-coefficients}
\end{equation}

\begin{theorem}
\label{thm:P2-blowup-main}
Assume Conjectures~\ref{conj:P2-vertex-comparison} and \ref{conj:P2-asymptotic-resummation}.  For \(r=1,-1,3\),
\begin{equation}
\begin{aligned}
 &\sum_{n\in\Z}(-1)^n
 \widehat Z_X\left(
 T+\eps_1\left(3n+\frac r2\right);
 \eps_1,\eps_2-\eps_1\right)\\
 &\hspace{23mm}\times
 \widehat Z_X\left(
 T+\eps_2\left(3n+\frac r2\right);
 \eps_1-\eps_2,\eps_2\right)
 =\Lambda_{X,r}\widehat Z_X(T;\eps_1,\eps_2).
\end{aligned}
\label{eq:P2-main-blowup}
\end{equation}
The equality is interpreted after division by
\(e^{F_{\mathbb P^2}^{(0)}(T;\eps_1,\eps_2)}\).
For \(r=\pm1\), also cancel the nonzero scalar \(\Lambda_{X,r}\); the result is an identity in \(\K[[Q]]\), \(Q=e^{-T}\).
For \(r=3\), also divide the left side by its common monomial \(Q^{1/3}\) before taking coefficients; the right side is zero.  In each case every coefficient of the resulting series is a finite sum over \(n\).
\end{theorem}

\begin{proof}
The analytic \(u\to0\) specialization is Lemma~\ref{lem:bilateral-u-zero-limit} and Proposition~\ref{prop:u-zero-specialization} of Appendix~\ref{app:u-zero}.  Theorem~\ref{thm:P2-PT-ADHM} identifies the positive degree stable pair series with the coefficientwise limit of the framed sheaf series.  Apply the multiplicative formula \eqref{eq:NY-multiplicative} for
\[
 (k,d_{\mathrm{NY}})=(0,1),\qquad(0,2),\qquad(1,1).
\]
By Lemma~\ref{lem:NY-chart-substitutions} the chart translations are \(T\mapsto T+\eps_i(3n+r/2)\) with \(r=1,-1,3\).  By the appendix, the constant term at \(u=0\) commutes with the two chart substitutions, their product, and the bilateral lattice sum.  The three identities therefore give \eqref{eq:P2-main-blowup}, including the removal of \(Q^{1/3}\) in the vanishing case.
\end{proof}

\subsection{Coefficient recursion}
\begin{corollary}
\label{cor:P2-finite-coefficients}
Assume the two conjectures of Theorem~\ref{thm:P2-blowup-main}, and let \(x=t/q\).  For \(r=1,-1\) and every \(N\geq0\),
\begin{equation}
 \sum_{\substack{n\in\Z,\ a,b\geq0\\m_r(n)+a+b=N}}
 (-1)^n x^{\gamma_r(n)}
 t^{-R_{r,n}a}q^{R_{r,n}b}
 A_a(t,tq)A_b(tq,q)=A_N(t,q).
\label{eq:P2-unity-finite-coefficients}
\end{equation}
For \(r=3\) and every \(N\geq0\),
\begin{equation}
 \sum_{\substack{n\in\Z,\ a,b\geq0\\m_3(n)+a+b=N}}
 (-1)^n x^{\gamma_3(n)}
 t^{-R_{3,n}a}q^{R_{3,n}b}
 A_a(t,tq)A_b(tq,q)=0.
\label{eq:P2-vanishing-finite-coefficients}
\end{equation}
Both sums are finite.
\end{corollary}

\begin{proof}
The two charts replace \((t,q,Q)\) by
\[
 (t,tq,Qt^{-R_{r,n}}),\qquad
 (tq,q,Qq^{R_{r,n}})
\]
respectively.  By Lemma~\ref{lem:P2-perturbative-shift}, the ratio of the two shifted degree \(0\) exponentials to the unshifted one is
\begin{equation}
 Q^{R_{r,n}^2/6-1/24}
 x^{-R_{r,n}^3/18+R_{r,n}/8}.
\label{eq:P2-perturbative-ratio}
\end{equation}
For \(r=\pm1\),
\[
 \frac{R_{r,n}^2}{6}-\frac1{24}=m_r(n),\qquad
 -\frac{R_{r,n}^3}{18}+\frac{R_{r,n}}8-\frac r{18}
 =\gamma_r(n).
\]
After division by the unshifted degree \(0\) exponential and by \(x^{r/18}\), the blowup identity reads
\[
 \sum_{n\in\Z}(-1)^nQ^{m_r(n)}x^{\gamma_r(n)}
 \mathcal Z(Qt^{-R_{r,n}};t,tq)
 \mathcal Z(Qq^{R_{r,n}};tq,q)
 =\mathcal Z(Q;t,q),
\]
and its coefficient of \(Q^N\) is \eqref{eq:P2-unity-finite-coefficients}.  For \(r=3\),
\[
 \frac{R_{3,n}^2}{6}-\frac1{24}=\frac13+m_3(n),\qquad
 -\frac{R_{3,n}^3}{18}+\frac{R_{3,n}}8=\gamma_3(n),
\]
and coefficient extraction after removing \(Q^{1/3}\) gives \eqref{eq:P2-vanishing-finite-coefficients}.  Each \(m_r\) is a quadratic polynomial with finite sublevel sets on \(\Z\), so the sums are finite.
\end{proof}

For \(N>0\), let \(E_{r,N}(t,q)\) be the part of the left side of \eqref{eq:P2-unity-finite-coefficients}, or of \eqref{eq:P2-vanishing-finite-coefficients} when \(r=3\), with \(a<N\) and \(b<N\):
\begin{equation}
 E_{r,N}
 =\sum_{\substack{n\in\Z,\ a,b\geq0\\
                   m_r(n)+a+b=N\\a<N,\ b<N}}
 (-1)^n x^{\gamma_r(n)}
 t^{-R_{r,n}a}q^{R_{r,n}b}
 A_a(t,tq)A_b(tq,q).
\label{eq:P2-lower-degree-remainder}
\end{equation}
Then
\begin{equation}
 \begin{pmatrix}
 -1&t^{-N/2}&q^{N/2}\\
 -1&t^{N/2}&q^{-N/2}\\
 0&t^{-3N/2}-t^{3N/2}&q^{3N/2}-q^{-3N/2}
 \end{pmatrix}
 \begin{pmatrix}
 A_N(t,q)\\ A_N(t,tq)\\ A_N(tq,q)
 \end{pmatrix}
 =-
 \begin{pmatrix}
 E_{1,N}\\ E_{-1,N}\\ E_{3,N}
 \end{pmatrix}.
\label{eq:P2-recursion-matrix}
\end{equation}
The determinant of the matrix is
\begin{equation}
 \frac{(t^N-1)(q^N-1)(t^N-q^N)((tq)^N-1)}
 {(tq)^{3N/2}},
\label{eq:P2-recursion-determinant}
\end{equation}
which is nonzero in \(\K\).  The three blowup equations and \(A_0=1\) therefore determine \(\mathcal Z\) uniquely.

To see this, note that in the \(r=\pm1\) equations the only terms containing \(A_N\) have \(n=0\) and \((a,b)=(N,0)\) or \((0,N)\), so their coefficients give the first two rows of \eqref{eq:P2-recursion-matrix}; for \(r=3\), the zero set of \(m_3(n)\) is \(\{0,-1\}\), and the terms with \(n=0,-1\) give the third row.  Every other summand has \(a,b<N\) and lies in \(E_{r,N}\).  Direct expansion gives \eqref{eq:P2-recursion-determinant}.  In the induction step, \(A_a\in\K\) is known for \(a<N\), and induction on \(N\) computes each coefficient.  The same determinant argument in the smallest degree where two solutions differ proves uniqueness.

\subsection{First coefficients}
Let
\[
 s=\sqrt{tq},\qquad x=\frac tq,\qquad
 [j]_x=x^{-j}+x^{-j+1}+\cdots+x^j,
 \qquad j\in\tfrac12\Z_{\geq0}.
\]
Substituting \(N=1\) and \(N=2\) in \eqref{eq:P2-recursion-matrix} gives
\begin{equation}
 A_1(t,q)=
 \frac{t^2+tq+q^2}{\sqrt{tq}(1-t)(1-q)}
\label{eq:P2-A1-from-recursion}
\end{equation}
and
\begin{equation}
 A_2(t,q)=
 \frac{(t^2+tq+q^2)\mathcal P_2(t,q)}
 {t^2q^2(1-t)^2(1+t)(1-q)^2(1+q)},
\label{eq:P2-A2-from-recursion}
\end{equation}
where
\begin{equation}
\begin{aligned}
 \mathcal P_2(t,q)={}&t^5q^2+t^2q^5+t^4q+tq^4+t^3q^3+t^2q^2\\
 &+t^3+q^3-t^5-q^5.
\end{aligned}
\label{eq:P2-A2-numerator}
\end{equation}
With \(N=3\) as well, and \(t=sx^{1/2}\), \(q=sx^{-1/2}\),
\begin{equation}
\begin{aligned}
 A_1(t,q)&=s[1]_x
  +s^2\bigl([3/2]_x+[1/2]_x\bigr)+O(s^3),\\
 A_2(t,q)&=s[5/2]_x
  +s^2\bigl([3]_x+[2]_x+[1]_x\bigr)+O(s^3),\\
 A_3(t,q)&=[9/2]_x
  +s\bigl([5]_x+[4]_x+[3]_x\bigr)+O(s^2).
\end{aligned}
\label{eq:P2-first-three-recursive-expansions}
\end{equation}
These are obtained by exact rational simplification of the three linear systems \eqref{eq:P2-recursion-matrix}.

To compare with the Laurent polynomials tabulated by Choi--Katz--Klemm, let
\begin{equation}
 P_{n,d}(x)
 :=(-1)^{d+n}[s^n]
 A_d(sx^{1/2},sx^{-1/2}).
\label{eq:P2-CKK-comparison-definition}
\end{equation}
On the branch \eqref{eq:PT-refined-square-root-branches}, \(s=\sqrt{tq}=y\).  The replacement \(\mathcal Z(Q)=Z_X^{\PT,\mathrm{ref}}(-Q)\) contributes \((-1)^d\).  By definition,
\begin{equation}
 [s^n]A_d(sx^{1/2},sx^{-1/2})
 =(-1)^{d+n}P_{n,d}(x).
\label{eq:P2-CKK-comparison-transform}
\end{equation}
The recursion gives
\begin{equation}
\begin{aligned}
 P_{1,1}&=[1]_x,&
 P_{2,1}&=-[3/2]_x-[1/2]_x,\\
 P_{1,2}&=-[5/2]_x,&
 P_{2,2}&=[3]_x+[2]_x+[1]_x,\\
 P_{0,3}&=-[9/2]_x,&
 P_{1,3}&=[5]_x+[4]_x+[3]_x.
\end{aligned}
\label{eq:P2-CKK-comparison-values}
\end{equation}
The polynomials \(P_{1,1}\), \(P_{1,2}\), \(P_{0,3}\) and \(P_{1,3}\) agree with the Laurent polynomials of Choi--Katz--Klemm \cite[Sections~8--9]{ChoiKatzKlemm}.  A geometric identification of \(P_{n,d}\) with the symmetrized index of \cite[Sections~7--8]{ChoiKatzKlemm} would additionally require a comparison of the square root, the equivariant lift, the branch, and the variable conventions.

\section{The Huang--Sun--Wang conjecture}
\label{sec:general-HSW}

\subsection{Lattice data}
The Huang--Sun--Wang equation is determined by the intersection pairing between compact curves and compact divisors \cite[Introduction and Section~3]{HuangSunWang}.  Let \(Y\) be a nonsingular quasi-projective 3-fold with \(K_Y\cong\cO_Y\).

Let \(N_{1,c}(Y)_{\Z}\) be the group generated by proper algebraic \(1\)-cycles, modulo numerical equivalence and then modulo torsion, and define \(N^1_c(Y)_{\Z}\) in the same way from compact Cartier divisors; we assume these are free of finite ranks \(b\) and \(g\), and that their intersection pairing has rank \(g\).  Choose integral bases \(\beta_1,\ldots,\beta_b\) and \(D_1,\ldots,D_g\), and let
\begin{equation}
 C_{i\alpha}=-\beta_i\cdot D_\alpha,\qquad
 \mathbf C=(C_{i\alpha})\in\operatorname{Mat}_{b\times g}(\Z).
\label{eq:general-C-matrix}
\end{equation}
Let \(\mathbf t=(t_1,\ldots,t_b)^t\) be the logarithmic curve coordinates.  For \(\mathbf d=(d_1,\ldots,d_b)\), the Novikov monomial of the corresponding class is \(Q^{\mathbf d}=e^{-\mathbf d\cdot\mathbf t}\).  Choose a \(\Z\)-basis \(v_1,\ldots,v_{b-g}\) of \(\ker(\mathbf C^t:\Z^b\to\Z^g)\), let \(\mathbf M\in\operatorname{Mat}_{(b-g)\times b}(\Z)\) have rows \(v_a^t\), and let
\begin{equation}
 \mathbf m=\mathbf M\mathbf t.
 \label{eq:general-mass-coordinates}
\end{equation}
Then \(\mathbf m\) is unchanged by \(\mathbf t\mapsto\mathbf t+\mathbf C\mathbf x\), and \(\mathbf M\) and \(\mathbf m\) are empty when \(b=g\).  Following Huang--Sun--Wang, the components of \(\mathbf m\) are called mass coordinates.

Fix an additive degree
\begin{equation}
 \deg_H:\operatorname{NE}_c(Y)\longrightarrow\Z_{\geq0}
 \label{eq:general-degree-filtration}
\end{equation}
which is positive on every nonzero effective class and has finite sublevel sets.  All curve class series in this section are completed with respect to \(\deg_H\).  The positive degree series below has coefficients meromorphic in \((\eps_1,\eps_2)\) and lies in the augmentation ideal of the completion \eqref{eq:completed-monoid-algebra-general}; the degree \(0\) exponential is a meromorphic prefactor outside the Novikov completion.

\subsection{The refined partition function}
Let
\[
 q_L=e^{(\eps_1-\eps_2)/2},\qquad
 q_R=e^{(\eps_1+\eps_2)/2},\qquad
 \chi_j(x)=\frac{x^{2j+1}-x^{-2j-1}}{x-x^{-1}}.
\]
Fix a refined curve counting partition function for \(Y\) with degree \(0\) coefficient \(1\) and positive degree logarithm of the form \eqref{eq:general-refined-BPS-free-energy} below.  Fix an orientation, that is, a square root of the virtual canonical line with its equivariant lift.  We assume the logarithm of the positive degree part has a multiple cover expansion
\begin{equation}
\begin{aligned}
 F_Y^{>0}(\mathbf t;\eps_1,\eps_2)
 ={}&\sum_{\mathbf d,w,j_L,j_R}
 \frac{(-1)^{2j_L+2j_R}N^{\mathbf d}_{j_L,j_R}}{w}
 \frac{\chi_{j_L}(q_L^w)\chi_{j_R}(q_R^w)}
 {(e^{w\eps_1/2}-e^{-w\eps_1/2})}\\[-2pt]
 &\hspace{35mm}\times
 \frac{e^{-w\mathbf d\cdot\mathbf t}}
 {(e^{w\eps_2/2}-e^{-w\eps_2/2})},
\end{aligned}
\label{eq:general-refined-BPS-free-energy}
\end{equation}
with integers \(N^{\mathbf d}_{j_L,j_R}\), where \(\mathbf d\) runs over nonzero effective compact classes, \(w\geq1\), and \(j_L,j_R\in\tfrac12\Z_{\geq0}\).  Local finiteness means that for every \(N\) only finitely many triples \((\mathbf d,j_L,j_R)\) with \(\deg_H(\mathbf d)\leq N\) have nonzero multiplicity; when such an expansion exists, the integers are determined by character decomposition and M\"obius inversion.

Fix symmetric coefficients \(a_{ijk}\) and linear coefficients \(b_i,b_i^{\mathrm{NS}}\) for the degree \(0\) normalization of Huang--Sun--Wang.  With repeated indices summed, the degree \(0\) part is
\begin{equation}
\begin{aligned}
 F_Y^{(0)}(\mathbf t;\eps_1,\eps_2)
 ={}&\frac1{\eps_1\eps_2}
 \left(\frac16a_{ijk}t_it_jt_k
       +4\pi^2b_i^{\mathrm{NS}}t_i\right)
 +b_it_i\\
 &-\frac{(\eps_1+\eps_2)^2}{\eps_1\eps_2}
 b_i^{\mathrm{NS}}t_i.
\end{aligned}
\label{eq:general-HSW-perturbative}
\end{equation}
Choose a parity class \(\mathbf B\in\Z^b\) with
\begin{equation}
 N^{\mathbf d}_{j_L,j_R}\ne0
 \quad\Longrightarrow\quad
 2j_L+2j_R+1\equiv\mathbf B\cdot\mathbf d\pmod2,
\label{eq:general-parity-class}
\end{equation}
and let
\begin{equation}
\begin{aligned}
 \widehat F_Y(\mathbf t;\eps_1,\eps_2)
 &:=F_Y^{(0)}(\mathbf t;\eps_1,\eps_2)
 +F_Y^{>0}(\mathbf t+\pi\mathrm{i}\mathbf B;
                     \eps_1,\eps_2),\\
 \widehat Z_Y(\mathbf t;\eps_1,\eps_2)
 &:=\exp\widehat F_Y(\mathbf t;\eps_1,\eps_2).
\end{aligned}
\label{eq:general-twisted-partition}
\end{equation}
An integer vector \(\mathbf r\equiv\mathbf B\pmod{2\Z^b}\) is called an \(r\)-field.  For such \(\mathbf r\), let
\begin{equation}
 \mathbf R_{\mathbf r}(\mathbf n)
 =\mathbf C\mathbf n+\frac{\mathbf r}{2},
 \qquad \mathbf n\in\Z^g,
\label{eq:general-R-field}
\end{equation}
and let
\begin{equation}
 \mathcal P_{Y,\mathbf B}
 =\{\mathbf r\in\Z^b:\mathbf r\equiv\mathbf B\pmod2\}
   \big/2\mathbf C\Z^g.
 \label{eq:general-parity-quotient}
\end{equation}
Every lattice sum below is interpreted coefficientwise in the completed curve class ring: for a fixed \(r\)-field and curve degree coefficient, the sum is first taken on a nonempty connected open set, away from its poles, on which it converges normally, and the result is continued meromorphically; the initial open set may depend on the coefficient and the \(r\)-field, and the continuations agree under \(\deg_H\)-adic truncation.

\subsection{The conjecture}
\begin{conjecture}[Huang--Sun--Wang \cite{HuangSunWang}]
\label{conj:general-HSW}
There are finite disjoint subsets
\[
 \mathcal R_Y^{\mathrm{unity}},\mathcal R_Y^{\mathrm{vanishing}}
 \subseteq\mathcal P_{Y,\mathbf B},
 \qquad
 \mathcal R_Y^{\mathrm{unity}}\cup
 \mathcal R_Y^{\mathrm{vanishing}}\ne\varnothing,
\]
such that every associated lattice sum has the coefficientwise meromorphic continuation described above and, for a representative \(\mathbf r\) of each class in their union and a nonzero meromorphic function \(\Lambda_{\mathbf r}(\mathbf m;\eps_1,\eps_2)\) for each unity class,
\begin{equation}
\begin{aligned}
 &\sum_{\mathbf n=(n_1,\ldots,n_g)\in\Z^g}(-1)^{n_1+\cdots+n_g}
 \widehat Z_Y(\mathbf t+\eps_1\mathbf R_{\mathbf r}(\mathbf n);
                 \eps_1,\eps_2-\eps_1)\\
 &\hspace{25mm}\times
 \widehat Z_Y(\mathbf t+\eps_2\mathbf R_{\mathbf r}(\mathbf n);
                 \eps_1-\eps_2,\eps_2)\\
 &\qquad=
 \begin{cases}
 \Lambda_{\mathbf r}(\mathbf m;\eps_1,\eps_2)
 \widehat Z_Y(\mathbf t;\eps_1,\eps_2),
     &[\mathbf r]\in\mathcal R_Y^{\mathrm{unity}},\\
  0,&[\mathbf r]\in\mathcal R_Y^{\mathrm{vanishing}}.
 \end{cases}
\end{aligned}
\label{eq:general-HSW-blowup}
\end{equation}
\end{conjecture}

For \(\mathbf n_0\in\Z^g\), the vectors \(\mathbf r\) and \(\mathbf r+2\mathbf C\mathbf n_0\) give equivalent equations: the change is absorbed by translating \(\mathbf n\) and multiplying the sum by \((-1)^{(n_0)_1+\cdots+(n_0)_g}\).  The relevant classes are therefore those of \eqref{eq:general-parity-quotient}, and for unity classes
\[
 \Lambda_{\mathbf r+2\mathbf C\mathbf n_0}
 =(-1)^{(n_0)_1+\cdots+(n_0)_g}\Lambda_{\mathbf r}.
\]
For the resolved conifold, the compact divisor lattice is \(0\) and the sum has a single term.

For \(Y_\ell\), \(0\leq\ell\leq2\), in the integral coordinates \((T_F,T_B)\), equation \eqref{eq:Hirzebruch-integral-HSW-data} gives
\[
 \mathbf C=\widetilde{\mathbf C}_\ell
 =\begin{pmatrix}2\\2-\ell\end{pmatrix},\qquad
 \mathbf r=\widetilde{\mathbf r}_{\ell;j,d}=
 \begin{pmatrix}2j\\\ell+2-2d+2(1-\ell)j\end{pmatrix},
 \qquad \mathbf B=\widetilde{\mathbf B}_\ell
 =\begin{pmatrix}0\\\ell\end{pmatrix}\pmod2.
\]
Corollary~\ref{cor:Hirzebruch-exact-HSW} proves the corresponding equations, including the vanishing class represented by \((2,0)^t\) in the \((T_F,T_M)\) coordinates, whose integral representative is \((2,2-\ell)^t\).  For local \(\mathbb P^2\), \(\mathbf C=(3)\) and \(\mathbf B=(1)\), and Theorem~\ref{thm:P2-blowup-main} gives \(r=1,-1,3\) under its two hypotheses.

\section{The rational elliptic surface}
\label{sec:rational-elliptic-conjecture}

\subsection{Geometry}
Let \(x_1,\ldots,x_9\) be the base points of a general pencil of plane cubics, and let
\[
 S=\operatorname{Bl}_{x_1,\ldots,x_9}\mathbb P^2.
\]
We require the contraction sequences used below to leave points in del Pezzo general position, that is, to produce surfaces with ample anticanonical class; this holds for a general pencil.  Let \(h\) be the pullback of a line and \(e_i\) the exceptional classes.  The anticanonical system defines an elliptic fibration \(\pi:S\to\mathbb P^1\).  We fix the section and fibre classes
\begin{equation}
 b_0=e_9,
 \qquad
 f=3h-\sum_{i=1}^9e_i=-K_S,
\label{eq:rational-elliptic-section-fibre}
\end{equation}
so \(b_0^2=-1\), \(b_0\cdot f=1\) and \(f^2=0\).  Let
\begin{equation}
 L_{E_8}=(\Z b_0\oplus\Z f)^\perp\subset H_2(S,\Z),
 \qquad
 (\lambda,\mu)_{E_8}=-\lambda\cdot\mu.
\label{eq:rational-elliptic-E8-lattice}
\end{equation}
With this positive definite form, \(L_{E_8}\) is the \(E_8\) root lattice.  Since \(\Z b_0\oplus\Z f\) is unimodular, every class has a unique expression
\begin{equation}
 \beta=kb_0+d f+\lambda,
 \qquad k,d\in\Z,
 \quad \lambda\in L_{E_8},
\label{eq:rational-elliptic-class-decomposition}
\end{equation}
and \(k=\beta\cdot f\) is the degree of \(\beta\) over the base of \(\pi\).

Let \(X_S=\Tot_S K_S\), with \(\C^*\) scaling the fibres of \(K_S\), and let \(\kappa\) be the character of the fibre scaling.

\begin{lemma}
\label{lem:rational-elliptic-fixed-proper}
Let \(S_0\) be a nonsingular projective surface and \(X_{S_0}=\Tot_{S_0}K_{S_0}\), with \(\C^*\) scaling the fibres.  For every nonzero effective \(\beta\in H_2(S_0,\Z)\) and every \(n\in\Z\), the fixed locus \(P_n(X_{S_0},\beta)^{\C^*}\) is proper.
\end{lemma}

\begin{proof}
Let \(I\subset\cO_{X_{S_0}}\) be the ideal of the zero section; a compact \(\C^*\)-invariant subset of \(X_{S_0}\) is set-theoretically contained in the zero section, since the orbit of a point with nonzero fibre coordinate is not proper, and for a fixed stable pair \(\cO_{X_{S_0}}\to F\), the ideal \(I\) therefore acts nilpotently on \(F\).  It remains to bound the nilpotence index.

Choose an ample integral divisor \(H\) on \(S_0\); at the generic point of an irreducible component \(C\) of the support, the nilpotence index of \(I\) is at most the multiplicity of \(C\) in the fundamental cycle of \(F\), which is at most \(H\cdot\beta\), whence \(I^{H\cdot\beta}F\) has \(0\)-dimensional support, and since \(F\) is pure of dimension \(1\) this subsheaf is \(0\).  Every fixed pair therefore factors through the finite thickening
\[
 S_{0,\beta}=V(I^{H\cdot\beta})\subset X_{S_0},
\]
which is finite over \(S_0\), hence projective.  Pairs on \(S_{0,\beta}\) with fixed Hilbert polynomial live in a projective relative Quot scheme, and the \(\C^*\)-fixed locus is closed.
\end{proof}

\subsection{Localized invariants}
The moduli space itself need not be proper.  For each \((n,\beta)\), choose an equivariant orientation of the fixed obstruction theory, that is, an equivariant square root of \(K^{\mathrm{vir}}|_{\mathscr F}\) on every connected component \(\mathscr F\subset P_n(X_S,\beta)^{\C^*}\), and let
\[
 \widehat{\mathcal O}^{\mathrm{vir}}_{\mathscr F}
 =\mathcal O^{\mathrm{vir}}_{\mathscr F}
  \otimes(K^{\mathrm{vir}}|_{\mathscr F})^{1/2}.
\]
By Lemma~\ref{lem:rational-elliptic-fixed-proper} every fixed component is proper, and by Lemma~\ref{lem:fixed-locus-square-root-independence} the rationalized localized index does not depend on the square roots.  For a compact class \(\beta\ne0\), let
\begin{equation}
 \mathsf P_{n,\beta}(\kappa)
 =\sum_{\mathscr F\subset P_n(X_S,\beta)^{\C^*}}
 \chi_{\C^*}\left(
  \mathscr F,
  \frac{\widehat{\mathcal O}^{\mathrm{vir}}_{\mathscr F}}
       {\Lambda_{-1}((N_{\mathscr F}^{\mathrm{vir}})^\vee)}
 \right),
\label{eq:rational-elliptic-localized-PT}
\end{equation}
with \(N_{\mathscr F}^{\mathrm{vir}}\) the moving virtual normal class.  Let \(Q_{b_0}=\mathbf Q^{b_0}\) and \(Q_\tau=\mathbf Q^f\).  For each nonzero effective \(\beta\), let
\begin{equation}
 \mathcal P_\beta(y,\kappa)
 :=\sum_{n\in\Z}y^n\mathsf P_{n,\beta}(\kappa)
 \in\Q(\kappa^{1/2})((y)).
\label{eq:rational-elliptic-fixed-class-series}
\end{equation}
Whenever this series is the Laurent expansion at \(y=0\) of a rational function, we denote the function by
\begin{equation}
 \mathcal P_\beta^{\mathrm{rat}}(y,\kappa)
 \in\Q(\kappa^{1/2})(y).
\label{eq:rational-elliptic-rationality}
\end{equation}

Let \(p_i=e^{\eps_i}\) with \(p_i^{1/2}=e^{\eps_i/2}\).  The unshifted parameters and the two chart substitutions are
\begin{equation}
\begin{array}{c|cc}
 i&\kappa_i^{1/2}&y_i\\ \hline
 0&p_1^{1/2}p_2^{1/2}&-p_1^{1/2}p_2^{-1/2}\\
 1&p_2^{1/2}&-p_1p_2^{-1/2}\\
 2&p_1^{1/2}&-p_1^{1/2}p_2^{-1}.
\end{array}
\label{eq:rational-elliptic-refined-variables}
\end{equation}
Here \(i=1\) corresponds to \((\eps_1,\eps_2-\eps_1)\) and \(i=2\) to \((\eps_1-\eps_2,\eps_2)\).

Fix an ample integral divisor \(H\) on \(S\), and let \(\K_{\mathrm{ell}} =\Q(p_1^{1/2},p_2^{1/2})\).  The coefficient ring is
\begin{equation}
 \Lambda_{S,H}^{\mathrm{ell}}
 :=\left\{\sum_{\gamma\in H_2(S,\Z)}a_\gamma \mathbf Q^\gamma:
 a_\gamma\in\K_{\mathrm{ell}},\
 \#\{\gamma:a_\gamma\ne0,\ H\cdot\gamma\leq N\}<\infty
 \text{ for every }N\in\mathbb R\right\}.
\label{eq:rational-elliptic-Novikov-ring}
\end{equation}
Under \eqref{eq:rational-elliptic-class-decomposition} we write
\[
 \mathbf Q^\beta=Q_{b_0}^kQ_\tau^d e^\lambda,
 \qquad e^\lambda:=\mathbf Q^\lambda,
\]
where the \(e^\lambda\) are the basis elements of the group algebra of \(L_{E_8}\).

If, after cancellation, each substitution
\begin{equation}
 \mathcal P_\beta^{\mathrm{rat}}(y_i,\kappa_i),
 \qquad i=0,1,2,
\label{eq:rational-elliptic-specialization-regularity}
\end{equation}
has denominator nonzero in \(\K_{\mathrm{ell}}\), it defines an element of that field.  In this case, for \(i=0,1,2\) and \(k\geq0\), let
\begin{equation}
\begin{aligned}
 \mathsf Z_k^{(i)}
 :={}&\delta_{k0}
 +\sum_{\substack{\beta=kb_0+df+\lambda\ne0\ \mathrm{effective}}}
 Q_\tau^d e^\lambda
 \mathcal P_\beta^{\mathrm{rat}}(y_i,\kappa_i),\\
 \mathsf E_k^{(i)}
 :={}&(\mathsf Z_0^{(i)})^{-1}\mathsf Z_k^{(i)},
 \qquad \mathsf Z_k:=\mathsf Z_k^{(0)},\quad
 \mathsf E_k:=\mathsf E_k^{(0)}.
\end{aligned}
\label{eq:rational-elliptic-Ek-definition}
\end{equation}
These lie in \(\Lambda_{S,H}^{\mathrm{ell}}\).  The term \(\delta_{k0}\) is the empty stable pair, so \(\mathsf Z_0^{(i)}\) has constant term \(1\) and is invertible, and \(\mathsf E_0^{(i)}=1\).

\subsection{The conjecture}
We use the normalized odd theta series
\begin{equation}
 \widehat\Theta_1(Q_\tau,X)
 :=\sum_{r\in\Z}(-1)^r
 Q_\tau^{r(r+1)/2}X^r.
\label{eq:logarithmic-odd-theta}
\end{equation}
For \(X=e^\alpha p_1^a p_2^b\) with \(a,b\in\Z\), its terms lie in \(\Lambda_{S,H}^{\mathrm{ell}}\): the quadratic power of \(Q_\tau\) gives the support condition of \eqref{eq:rational-elliptic-Novikov-ring}.  For \(\alpha\in L_{E_8}\), define the group algebra automorphisms
\begin{equation}
 \mathsf T_{i,\alpha}(e^\lambda)
 =p_i^{(\lambda,\alpha)_{E_8}}e^\lambda,
 \qquad
 \mathsf T_{i,\alpha}(Q_{b_0})=Q_{b_0},\qquad
 \mathsf T_{i,\alpha}(Q_\tau)=Q_\tau
 \quad(i=1,2).
\label{eq:rational-elliptic-character-translation}
\end{equation}

\begin{conjecture}
\label{conj:rational-elliptic-theta}
There is a system of equivariant orientations as above for which every series \eqref{eq:rational-elliptic-fixed-class-series} has the rational form \eqref{eq:rational-elliptic-rationality} and all three substitutions \eqref{eq:rational-elliptic-specialization-regularity} are well defined.  With the definitions above, let \(\alpha\in L_{E_8}\) be a root, \((\alpha,\alpha)_{E_8}=2\).  For every \(k\geq0\),
\begin{equation}
\begin{aligned}
 &\sum_{\substack{k_1+k_2=k\\k_1,k_2\geq0}}
 p_1^{-k_1/2}p_2^{-k_2/2}
 \widehat\Theta_1\left(
 Q_\tau,e^\alpha p_1^{1-k_1}p_2^{1-k_2}\right)\\
 &\qquad\times
 \mathsf T_{1,\alpha}(\mathsf E_{k_1}^{(1)})
 \mathsf T_{2,\alpha}(\mathsf E_{k_2}^{(2)})
 =\widehat\Theta_1(Q_\tau,e^\alpha p_1p_2)
 \mathsf E_k,
\end{aligned}
\label{eq:rational-elliptic-unity}
\end{equation}
and
\begin{equation}
 \sum_{\substack{k_1+k_2=k\\k_1,k_2\geq0}}
 p_1^{k_1/2}p_2^{k_2/2}
 \widehat\Theta_1\left(
 Q_\tau,p_1^{k_1}p_2^{k_2}\right)
 \mathsf E_{k_1}^{(1)}\mathsf E_{k_2}^{(2)}=0.
\label{eq:rational-elliptic-vanishing}
\end{equation}
\end{conjecture}

These are identities in \(\Lambda_{S,H}^{\mathrm{ell}}\), with the \(E_8\)-characters formal and no convergence condition on the Novikov coordinates.  They have the theta function form of the root and zero vector identities of \cite[Section~3]{GuHaghighatKlemmSunWang}, whose convention is related to ours by \(Q_\tau=e^{2\pi\mathrm{i}\tau}\) and
\[
 \Theta_1(\tau,z)
 =Q_\tau^{1/8}e^{z/2}
  \widehat\Theta_1(Q_\tau,e^z).
\]
For \(\boldsymbol\mu\in L_{E_8}\otimes\C\), the evaluation \(e^\lambda\mapsto\exp((\lambda,\boldsymbol\mu)_{E_8})\) is applied to each finite \(H\)-degree truncation.  When the evaluated truncations converge normally on a nonempty open set, their sums define the specialization at \(\boldsymbol\mu\) and extend meromorphically.  Let
\begin{equation}
 \mathcal Z_S^{\mathrm{norm}}(Q_{b_0};p_1,p_2)
 :=\sum_{k\geq0}Q_{b_0}^k
 \mathsf E_k,
\label{eq:rational-elliptic-total-series}
\end{equation}
the stable pair series with every base degree coefficient divided by \(\mathsf Z_0\).

\subsection{Contraction of a \texorpdfstring{\((-1)\)-curve}{(-1)-curve}}
Let \(e\subset S\) be a \((-1)\)-curve with contraction \(\rho:S\to S'\), and let \(X_{S'}=\Tot_{S'}K_{S'}\).  For \(\gamma\in H_2(S',\Z)\), let \(\rho^!\gamma\) be the unique integral lift orthogonal to \(e\), so that
\[
 H_2(S,\Z)=\rho^!H_2(S',\Z)\oplus\Z e.
\]
Let \(Q_e=\mathbf Q^e\), and let \(\mathbf Q'^{\gamma}\) be the Novikov monomial of \(\gamma\) on \(S'\).  The change of Novikov monomials is
\begin{equation}
 \mathbf Q^{\rho^!\gamma+m e}
 \longmapsto \mathbf Q'^{\gamma}Q_e^m,
 \qquad \gamma\in H_2(S',\Z),\quad m\in\Z,
 \label{eq:contraction-curve-variable-map}
\end{equation}
extended coefficientwise to series satisfying the support condition of \eqref{eq:rational-elliptic-Novikov-ring}.  Let \(k(\gamma)=f\cdot\rho^!\gamma\).

Choose parity homomorphisms
\[
 \mathbf B_S:H_2(S,\Z)\longrightarrow\Z/2,
 \qquad
 \mathbf B':H_2(S',\Z)\longrightarrow\Z/2
\]
satisfying \eqref{eq:general-parity-class} for their local 3-folds and the compatibility
\begin{equation}
 \mathbf B_S(\rho^!\gamma+m e)
 =\mathbf B'(\gamma)+m\pmod2;
 \label{eq:contraction-parity-extension}
\end{equation}
in particular \(\mathbf B_S(e)=1\).  For a parity homomorphism \(\mathbf B\), the parity twist is
\begin{equation}
 \operatorname{tw}_{\mathbf B}(\mathbf Q^\beta)
 =(-1)^{\mathbf B(\beta)}\mathbf Q^\beta,
 \qquad \operatorname{tw}_{\mathbf B}(p_i^{1/2})=p_i^{1/2},
 \label{eq:parity-twist-homomorphism}
\end{equation}
extended linearly and continuously to \eqref{eq:rational-elliptic-Novikov-ring}; on \(S'\) it is \(\mathbf Q'^\gamma\mapsto(-1)^{\mathbf B'(\gamma)} \mathbf Q'^\gamma\).  Choose an integral basis \(\beta'_1,\ldots,\beta'_{b'}\) of \(H_2(S',\Z)\), \(b'=\rk H_2(S',\Z)\), and the adapted basis \(\rho^!\beta'_1,\ldots,\rho^!\beta'_{b'},e\) on \(S\).  The intersection matrices \eqref{eq:general-C-matrix} of \(X_S\) and \(X_{S'}\) are single columns with
\[
 \mathbf C_S=\begin{pmatrix}\mathbf C_{S'}\\1\end{pmatrix}.
\]
Deleting the last component of an \(r\)-field defines
\begin{equation}
 \pi_e:
 \mathcal P_{X_S,\mathbf B_S}\longrightarrow
 \mathcal P_{X_{S'},\mathbf B'},
 \qquad
 [(\mathbf r',r_e)]\longmapsto[\mathbf r'],
\label{eq:contraction-r-field-projection}
\end{equation}
which is well defined by \eqref{eq:contraction-parity-extension} and the relation between the two columns.  Let
\[
 \mathcal Z_S^{\mathrm{norm,tw}}
 =\operatorname{tw}_{\mathbf B_S}(\mathcal Z_S^{\mathrm{norm}}).
\]

Fix an ample integral divisor \(H'\) on \(S'\), and let \(\Lambda_{S',H'}^{\mathrm{ell}}\) be the ring \eqref{eq:rational-elliptic-Novikov-ring} with \(S,H,\mathbf Q^\gamma\) replaced by \(S',H',\mathbf Q'^\gamma\).  With \(w_e=Q_e^{-1}\), let
\begin{equation}
 \mathscr L_{S',H'}
 :=\left\{
  \sum_{\gamma\in H_2(S',\Z)}
  f_\gamma(w_e)\mathbf Q'^\gamma:
  \begin{array}{l}
   f_\gamma(w_e)\in\K_{\mathrm{ell}}((w_e)),\\
   \#\{\gamma:f_\gamma\ne0,\ H'\cdot\gamma\leq N\}<\infty\\
   \text{for every }N\in\mathbb R
  \end{array}
 \right\},
\label{eq:contraction-Laurent-ring}
\end{equation}
where \(R((w_e))=R[[w_e]][w_e^{-1}]\).  The support condition makes multiplication coefficientwise finite and defines the continuous map
\begin{equation}
 [w_e^0]:\mathscr L_{S',H'}
 \longrightarrow\Lambda_{S',H'}^{\mathrm{ell}}
\label{eq:contraction-constant-term-map}
\end{equation}
taking the coefficient of \(w_e^0\) in every \(f_\gamma\).

Choose equivariant orientations on the fixed loci of \(X_{S'}\), let \(\mathsf P^{S'}_{n,\gamma}(\kappa)\) be the localized index \eqref{eq:rational-elliptic-localized-PT} for \(S'\), and let
\[
 \mathcal P^{S'}_\gamma(y,\kappa)
 =\sum_{n\in\Z}y^n\mathsf P^{S'}_{n,\gamma}(\kappa).
\]
Assuming the rational continuation \(\mathcal P_{\gamma}^{S',\mathrm{rat}}\) exists and is regular at \((y_0,\kappa_0)\), let
\begin{equation}
\begin{aligned}
 \mathsf Z_{S',k}
 &:={\delta}_{k0}+
 \sum_{\substack{\gamma\ne0\ \mathrm{effective}\\ k(\gamma)=k}}
 \mathbf Q'^\gamma
 \mathcal P_{\gamma}^{S',\mathrm{rat}}(y_0,\kappa_0),\\
 \mathsf Z_{S',k}^{\mathrm{tw}}
 &:=\operatorname{tw}_{\mathbf B'}(\mathsf Z_{S',k}),\\
 \mathcal Z_{S'}^{\mathrm{norm,tw}}
 &:=(\mathsf Z_{S',0}^{\mathrm{tw}})^{-1}
   \sum_{k\geq0}\mathsf Z_{S',k}^{\mathrm{tw}}.
\end{aligned}
\label{eq:contracted-surface-series-definition}
\end{equation}
With \(p_i=e^{\eps_i}\) and \(p_i^{1/2}=e^{\eps_i/2}\), the exceptional curve factor is
\begin{equation}
 \mathcal U_e(Q_e;p_1,p_2)
 =\PE\left[
 \frac{Q_e}
 {(p_1^{1/2}-p_1^{-1/2})(p_2^{1/2}-p_2^{-1/2})}
 \right],
\label{eq:exceptional-curve-factor}
\end{equation}
which for \(|p_1|,|p_2|<1\) is the convergent product
\begin{equation}
 \mathcal U_e(Q_e;p_1,p_2)
 =\prod_{i,j\geq0}
  \left(1-Q_ep_1^{i+1/2}p_2^{j+1/2}\right)^{-1}.
 \label{eq:exceptional-curve-product}
\end{equation}
By \eqref{eq:contraction-parity-extension}, \(\operatorname{tw}_{\mathbf B_S}(\mathcal U_e(Q_e))= \mathcal U_e(-Q_e)\).  Take first \(0<p_1,p_2<1\).  Choose \(R_e>0\) and \(\theta_-<\theta_+\) so that the closure of
\[
 \mathscr S_e=\{Q_e:|Q_e|>R_e,\ \theta_-<\arg Q_e<\theta_+\}
\]
avoids the negative real axis.  Let \(\mathscr D_e\subset\C^*\) be a simply connected pole free domain containing \(\mathscr S_e\) and an open set on which \eqref{eq:exceptional-curve-product} converges, fix a branch of \(\log Q_e\) on it, and let \(\mathcal U_e^{\mathrm{tw,an}}\) be the analytic continuation of \(\mathcal U_e(-Q_e;p_1,p_2)\) to \(\mathscr D_e\); shifted and chart substituted factors are continued in the same way, on pole free domains meeting a common large \(\lvert Q_e\rvert\) sector and with compatible logarithm branches.  The identities are then continued meromorphically in \(p_1,p_2\).

We refer to the following as the \emph{contraction hypotheses}.
\begin{enumerate}
\item[(C1)] There is an equivariant orientation on \(S\) for which Conjecture~\ref{conj:rational-elliptic-theta} holds, and an orientation on \(S'\) compatible with it and with \eqref{eq:contraction-parity-extension}.

\item[(C2)] The series \(\mathcal P_{\gamma}^{S'}\) admit the rational continuations used in \eqref{eq:contracted-surface-series-definition}, regular at \((y_0,\kappa_0)\).

\item[(C3)] After the change of variables \eqref{eq:contraction-curve-variable-map}, each coefficient of \(\mathcal Z_S^{\mathrm{norm,tw}}\), viewed as a Laurent series in \(Q_e\) with finite principal part, converges normally on compact subsets of a punctured neighbourhood of \(Q_e=0\) and continues analytically to the common sector fixed above.

\item[(C4)] In the coordinate \(w_e=Q_e^{-1}\), the quotient
\begin{equation}
 (\mathcal U_e^{\mathrm{tw,an}})^{-1}
 \mathcal Z_S^{\mathrm{norm,tw}}
\label{eq:contraction-normalized-quotient}
\end{equation}
extends holomorphically to \(|w_e|<\delta\) for some \(\delta>0\), coefficientwise in the \(H'\)-adic completion.  Equivalently, after expansion at \(w_e=0\), every coefficient lies in \(\K_{\mathrm{ell}}[[w_e]]\), with the same finite \(H'\)-degree support condition as in \eqref{eq:contraction-Laurent-ring}.

\item[(C5)] Divide every shifted factor of a Huang--Sun--Wang equation on \(X_S\) by the corresponding shifted and chart substituted \(\mathcal U_e^{\mathrm{tw,an}}\), and the unity term by the unshifted exceptional factor.  Each normalized shifted factor extends holomorphically to \(w_e=0\), coefficientwise with the support condition of \textup{(C4)}, and both sides of the resulting equation are convergent there.  The value at \(w_e=0\) of each shifted normalized factor is the corresponding shifted series on \(X_{S'}\) under \eqref{eq:contraction-r-field-projection}.

\item[(C6)] For each root identity and the zero vector identity in \eqref{eq:rational-elliptic-unity}--\eqref{eq:rational-elliptic-vanishing}, there is a class \([\mathbf r_S]\in\mathcal P_{X_S,\mathbf B_S}\) whose Huang--Sun--Wang equation is that identity, such that \(\pi_e([\mathbf r_S])\) has the same unity or vanishing type on \(X_{S'}\), the normalized unity coefficient depends only on the projected class, and the cubic and linear terms of \eqref{eq:general-HSW-perturbative} restrict under \eqref{eq:contraction-curve-variable-map} to those of \(X_{S'}\).
\end{enumerate}

\begin{conjecture}
\label{conj:rational-elliptic-contraction}
Assume \textup{(C1)--(C6)}.  Then
\begin{equation}
 [w_e^0]\left((\mathcal U_e^{\mathrm{tw,an}})^{-1}
 \mathcal Z_S^{\mathrm{norm,tw}}\right)
 =\mathcal Z_{S'}^{\mathrm{norm,tw}}.
\label{eq:stable-pair-contraction-conjecture}
\end{equation}
\end{conjecture}

Under the hypotheses above, the expansions occurring in the divided root and zero vector equations are power series in \(w_e\).  Hence \([w_e^0]\) is evaluation at \(w_e=0\) and is multiplicative on every product appearing in those equations.  Applying this evaluation, using \textup{(C5)} and Conjecture~\ref{conj:rational-elliptic-contraction}, and then restoring \eqref{eq:general-HSW-perturbative}, gives the Huang--Sun--Wang equation \eqref{eq:general-HSW-blowup} on \(X_{S'}\).

\appendix
\section{The specialization at \texorpdfstring{\(u=0\)}{u=0}}
\label{app:u-zero}

The framing parameter enters the degree \(0\) factors of the Nakajima--Yoshioka series through infinite products whose region of convergence need not contain the limit \(u\to0\), and the lattice sum has infinitely many terms.  We use the meromorphic continuation of the degree \(0\) factors and a completion of the monoid generated by the monomials which decay as \(A\to\infty\); for fixed product expansions with this decay property, the connection formulas and support estimates below are unconditional.  Conjecture~\ref{conj:P2-asymptotic-resummation} supplies the expansions and allows the passage from individual coefficients of \(\Lambda^4\) to the full instanton and lattice sums; Conjecture~\ref{conj:P2-vertex-comparison} identifies the constant term with the local \(\mathbb P^2\) series.

We keep \(u=e^{-A}\), \(Q=e^{-T}\) and \(M=T-3A/2\) from \eqref{eq:centered-variables}--\eqref{eq:u-zero-parameter-path}.

\subsection{The formal series ring}
We use the factors \(\mathcal R\), their specified product representations, and the finite set \(\mathcal P_{\mathrm{exp}}\) defined in Section~\ref{sec:P2-blowup}.

A \emph{Stokes direction} is a ray of angle \(\theta\) with
\[
 \operatorname{Re}(\lambda e^{\mathrm{i}\theta})=0
 \quad\text{for some }(\lambda,d)\in\mathcal P_{\mathrm{exp}}.
\]
Under Conjecture~\ref{conj:P2-asymptotic-resummation}, choose a ray of angle \(\theta_*\) with \(\operatorname{Re}(e^{\mathrm{i}\theta_*})>0\) avoiding the finitely many Stokes directions, together with product representations supplied by the conjecture, and shrink to a closed interval \(I\) about \(\theta_*\) whose closure contains no Stokes direction and such that every exponent pair in those representations satisfies
\[
 \operatorname{Re}(\lambda e^{\mathrm{i}\theta})>0
 \qquad(\theta\in I).
\]
Choose \(R_0>0\) and let
\[
 \mathfrak S=\{re^{\mathrm{i}\theta}:r\geq R_0,\ \theta\in I\}.
\]
We let \(A\to\infty\) in \(\mathfrak S\) with \(T\) in a compact set.  Each product is expanded by \eqref{eq:double-Pochhammer-log-expansion}, and the resulting formal Laurent expansions are the ones used below.

Fix a generic value \(\eps^\circ=(\eps_1^\circ,\eps_2^\circ)\) for which the product representations are valid, away from the poles and from the proper analytic subsets on which distinct exponents coincide.  The exponent pairs and the real functions \(h,v\) below are evaluated at \(\eps^\circ\); the completed algebra is constructed separately at each generic value, and different values may give different monoids \(\Gamma_{\mathfrak S}\).  Let \((\lambda_\rho,d_\rho)\), \(1\leq\rho\leq N\), be the distinct members of \(\mathcal P_{\mathrm{exp}}\) occurring in the product representations on \(\mathfrak S\).  By \eqref{eq:double-Pochhammer-log-expansion}, the logarithm of each factor is a series in positive powers of a monomial
\begin{equation}
 a_\rho(\eps_1,\eps_2)e^{-\lambda_\rho A}Q^{d_\rho},
 \qquad
 \operatorname{Re}(\lambda_\rho e^{\mathrm{i}\theta_*})>0.
\label{eq:modular-factor-monomial}
\end{equation}
Here \(d_\rho\) may be nonintegral; fix \(\operatorname{Log}Q\) on the region considered, read \(Q^d=\exp(d\operatorname{Log}Q)\), and let \(\K_{\eps^\circ}\subset\C\) be the field generated by the values at \(\eps^\circ\) of \(\K\) and of the meromorphic coefficients of these products and all their specified chart translates.  The identities obtained extend meromorphically by the identity principle.

Let \(\Gamma_{\mathfrak S}\subset\C^2\) be the additive submonoid generated by the exponent pairs
\[
 (0,1),\qquad (1,-1),\qquad
 (\lambda_\rho,d_\rho)\quad(\rho=1,\ldots,N).
\]
The first two represent \(Q\) and \(u/Q\), so \((1,0)=(0,1)+(1,-1)\) represents \(u\).  We write
\[
 \mathbf e^{(\lambda,d)}=e^{-\lambda A}Q^d.
\]
Define the additive real functions
\[
h(\lambda,d)=\operatorname{Re}(\lambda e^{\mathrm{i}\theta_*}),
 \qquad v(\lambda,d)=\operatorname{Re}(d).
\]  Every generator other than \((0,1)\) has positive \(h\)-value.  Since the generators are finite in number, there is \(C>0\) with
\begin{equation}
 v(\gamma)\geq-Ch(\gamma)
 \qquad(\gamma\in\Gamma_{\mathfrak S}).
\label{eq:semigroup-cone-bound}
\end{equation}

\begin{definition}
\label{def:completed-semigroup-algebra}
Let \(\mathscr H_{\mathfrak S}\) be the set of formal sums
\begin{equation}
 \sum_{\gamma\in\Gamma_{\mathfrak S}}c_\gamma \mathbf e^\gamma,
 \qquad c_\gamma\in\K_{\eps^\circ},
\label{eq:completed-semigroup-algebra}
\end{equation}
such that for every \(H,V\in\mathbb R\) only finitely many nonzero terms satisfy \(h(\gamma)\leq H\) and \(v(\gamma)\leq V\).  For \(H\geq0\), let
\[
 F^{>H}\mathscr H_{\mathfrak S}
 =\left\{\sum_\gamma c_\gamma \mathbf e^\gamma:
 c_\gamma=0\text{ whenever }h(\gamma)\leq H\right\}.
\]
\end{definition}

The subsets \(F^{>H}\mathscr H_{\mathfrak S}\) form a neighbourhood basis of \(0\).  Let \(\mathfrak m_A=F^{>0}\mathscr H_{\mathfrak S}\).  Projection to \(h=0\) defines
\begin{equation}
 \operatorname{CT}_A:\mathscr H_{\mathfrak S}
 \longrightarrow\K_{\eps^\circ}[[Q]].
\label{eq:asymptotic-constant-term}
\end{equation}

\begin{lemma}
\label{lem:completed-semigroup-support}
Multiplication is well defined in \(\mathscr H_{\mathfrak S}\), each \(F^{>H}\mathscr H_{\mathfrak S}\) is an ideal, and the algebra is complete and separated.  Every element of \(1+\mathfrak m_A\) is invertible, and \(\operatorname{CT}_A\) is a continuous ring homomorphism.  For \(i\in\{1,2\}\), \(n\in\Z\), and
\[
 (k,d_{\mathrm{NY}})\in\{(0,1),(0,2),(1,1)\},
 \qquad r=3-2d_{\mathrm{NY}}+2k,
\]
the chart translations
\[
 A\longmapsto A+\eps_i(2n+k),\qquad
 T\longmapsto T+\eps_iR_{r,n}
\]
preserve the exponent pairs, hence \(h\), \(v\), and the estimate \eqref{eq:semigroup-cone-bound}.  They commute with extraction of the \(A\)-constant term, with the induced substitution \(Q\mapsto Qe^{-\eps_iR_{r,n}}\) on the target:
\[
 \operatorname{CT}_A(\tau_{i,n}f)
 =\tau_{i,n}^{\,0}(\operatorname{CT}_A f).
\]
Here \(\tau_{i,n}^{\,0}\) denotes this \(Q\)-substitution together with the corresponding change of equivariant parameters.  This is commutation, not invariance of the resulting \(Q\)-series.  The common monoid includes the generators for all three parameter pairs before the additive translations are applied.  The factor
\[
 \left(-\sqrt{tq}\,\frac uQ;t,q\right)_\infty
\]
lies in \(1+\mathfrak m_A\).  Under Conjecture~\ref{conj:P2-asymptotic-resummation}, the specified expansions of \(\mathcal R\) and every chart translate lie in \(1+\mathfrak m_A\).  Dividing by the corresponding incoming product shows that the induced expansions of \(\Phi(Y_{\mathrm{edge}})\Phi(Y_{\mathrm{root}})^{-1}\) and its chart translates do so as well.
\end{lemma}

\begin{proof}
Let \(S_1,S_2\) be two supports satisfying the condition of Definition~\ref{def:completed-semigroup-algebra}.  If \(\gamma_1+\gamma_2\) has \(h\)-value at most \(H\) and \(v\)-value at most \(V\), then \(h(\gamma_i)\leq H\) and
\[
 v(\gamma_i)
 =v(\gamma_1+\gamma_2)-v(\gamma_{3-i})
 \leq V+CH
\]
by \eqref{eq:semigroup-cone-bound}.  Each \(S_i\) then leaves only finitely many possibilities for \(\gamma_i\): every convolution coefficient is a finite sum, and the product again has locally finite support; since \(h\) is additive and nonnegative on \(\Gamma_{\mathfrak S}\), multiplication by any element sends \(F^{>H}\mathscr H_{\mathfrak S}\) into itself, and these are ideals and define a multiplicative topology.

The only elements of \(\Gamma_{\mathfrak S}\) with \(h=0\) are the nonnegative multiples of \((0,1)\), so the \(h=0\) part is \(\K_{\eps^\circ}[[Q]]\), and since the positive \(h\)-values of the finite generating set have a positive lower bound, for \(x\in\mathfrak m_A\) the partial sums of \(\sum_{k\geq0}(-x)^k\) stabilize below every fixed \(h\)-bound and define \((1+x)^{-1}\).  The same stabilization proves completeness.  An element in every \(F^{>H}\) has empty support, so the topology is separated; since \(\operatorname{CT}_A\) is the projection onto the \(h=0\) subalgebra, it is continuous and multiplicative.

The logarithm of the incoming product is a series in positive powers of \(u/Q\).  Assumptions (A2)--(A3) and \eqref{eq:double-Pochhammer-log-expansion} put the specified expansion of \(\mathcal R\) in \(1+\mathfrak m_A\).  The incoming product is a unit there, giving the assertion for the quotient of the \(\Phi\)-functions.  A chart translation multiplies an exponent monomial by a nonzero meromorphic coefficient and leaves the exponent pair unchanged, while the Stokes directions depend only on the \(\lambda_\rho\) and are preserved by the chart translations.
\end{proof}

\subsection{Asymptotic expansion}
\begin{lemma}
\label{lem:full-u-zero-expansion}
Assume Conjectures~\ref{conj:P2-vertex-comparison} and \ref{conj:P2-asymptotic-resummation}, and let \(A\to\infty\) in the closed angular region fixed above.  With
\[
 c(\eps_1,\eps_2)
 =-\frac{\pi\mathrm{i}(\eps_1^2+3\eps_1\eps_2+\eps_2^2)}
 {24\eps_1\eps_2},
\]
which is independent of \(A,M,T,Q\), one has
\begin{equation}
\begin{aligned}
 Z_1(\eps_1,\eps_2,(-A/2,A/2);\Lambda)
 ={}&e^{c(\eps_1,\eps_2)}
 e^{G_{\eps_1,\eps_2}(A,M)
       +\widetilde P_M(M;\eps_1,\eps_2)}\\
 &\times e^{F_{\mathbb P^2}^{(0)}
       (T;\eps_1,\eps_2)}
 \mathscr A_{\mathrm c}(Q,u;e^{\eps_1},e^{-\eps_2})
 \mathcal R(A,T;\eps_1,\eps_2)
\end{aligned}
\label{eq:full-u-zero-expansion}
\end{equation}
as meromorphic functions on that region, where \(u=e^{-A}\), \(T=M+3A/2\) with \(T\) fixed as \(A\to\infty\), and \(\mathcal R\) is \eqref{eq:remaining-product-factor}.

After multiplying the left side by
\[
 e^{-c(\eps_1,\eps_2)
    -G_{\eps_1,\eps_2}(A,M)
    -\widetilde P_M(M;\eps_1,\eps_2)
    -F_{\mathbb P^2}^{(0)}(T;\eps_1,\eps_2)},
\]
its formal Laurent expansion is \(\mathscr A_{\mathrm c}(Q,u;e^{\eps_1},e^{-\eps_2}) \mathcal R(A,T;\eps_1,\eps_2)\in\mathscr H_{\mathfrak S}\).  The factor \(\mathcal R\) lies in \(1+\mathfrak m_A\), and
\begin{equation}
 \operatorname{CT}_A\bigl(
 \mathscr A_{\mathrm c}(Q,u;t,q)\mathcal R(A,T;\eps_1,\eps_2)\bigr)
 =\mathscr A_{\mathrm c}(Q,0;t,q)=\mathcal Z(Q;t,q).
\label{eq:u-zero-order-zero}
\end{equation}
\end{lemma}

\begin{proof}
Combine the two exponentiated root factors in \eqref{eq:NY-full-partition}.  By \eqref{eq:NY-exponentiated-gamma}, after removing their explicit polynomial exponentials, their product is
\begin{equation}
 (ut^{-1}q;t^{-1},q)_\infty
 (u^{-1}t^{-1}q;t^{-1},q)_\infty
 =(qu;t,q)_\infty^{-1}(q/u;t,q)_\infty^{-1},
\label{eq:root-series-Pochhammer}
\end{equation}
where \(t=e^{\eps_1}\), \(q=e^{-\eps_2}\).  This is an identity of the exponentiated products, so no logarithmic continuation at \(-A\) is needed.
Indeed, before inverting the first base the two arguments are \(ut^{-1}q\) and \(u^{-1}t^{-1}q\) with bases \((t^{-1},q)\), and \((x;t^{-1},q)_\infty=(tx;t,q)_\infty^{-1}\) gives \eqref{eq:root-series-Pochhammer}.

Since \(q/u=-\sqrt{tq}\,e^{Y_{\mathrm{root}}}\), the inverse of \eqref{eq:exact-double-Pochhammer-connection} gives
\begin{equation}
 (q/u;t,q)_\infty^{-1}
 =e^{-W(Y_{\mathrm{root}})}(tu;t,q)_\infty^{-1}
  \Phi(Y_{\mathrm{root}})^{-1}.
\label{eq:root-outgoing-connection}
\end{equation}
The factors \((qu;t,q)_\infty^{-1}\) and \((tu;t,q)_\infty^{-1}\) satisfy
\begin{equation}
 (qu;t,q)_\infty^{-1}(tu;t,q)_\infty^{-1}
 =\PE\left[\frac{(t+q)u}{(1-t)(1-q)}\right].
\label{eq:root-one-leg-products}
\end{equation}
Combining the polynomial part of \eqref{eq:NY-gamma} with \(-W(Y_{\mathrm{root}})\), in the branch \eqref{eq:Lambda-continuation-branch},
\begin{equation}
\begin{aligned}
 &\exp[-\widetilde\gamma_{\eps_1,\eps_2}(A;\Lambda)
       -\widetilde\gamma_{\eps_1,\eps_2}(-A;\Lambda)]\\
 &\quad={}
 e^{c+G_{\eps_1,\eps_2}(A,M)+P_{\pm A}(A,M;\eps_1,\eps_2)
       +L_{\eps_1,\eps_2}(M)}
 \PE\left[
  \frac{(t+q)u}{(1-t)(1-q)}\right]
 \Phi(Y_{\mathrm{root}})^{-1}.
\end{aligned}
\label{eq:root-continuation}
\end{equation}
Indeed, the \(\pi\mathrm{i}\)-dependent part of \(-W(Y_{\mathrm{root}})\) contributes \(\pi\mathrm{i}A^2/(2D)\), while the lift \eqref{eq:Lambda-continuation-branch} changes \(\log\Lambda\) by \(\pi\mathrm{i}/2\) and the quadratic term of \eqref{eq:NY-gamma} contributes the opposite \(-\pi\mathrm{i}A^2/(2D)\).  The \(A\)-dependent polynomial left after removing \(G_{\eps_1,\eps_2}(A,M)\) is
\[
 \frac1D\left(
  \frac{A^3}{6}+\frac{MA^2}{4}-\frac{\pi^2A}{3}
  +\frac{\eps_1^2+\eps_2^2+3D}{12}A\right)
 =P_{\pm A}(A,M;\eps_1,\eps_2).
\]
The coefficient of \(M\) independent of \(A\) in the paired roots is
\[
 \frac{\eps_1^2+\eps_2^2+3D}{24D}M;
\]
half of it sits in \(G_{\eps_1,\eps_2}(A,M)\) and the other half is \(L_{\eps_1,\eps_2}(M)\).  The remaining \(A,M\)-independent terms are
\(-\pi\mathrm{i}(\eps_1^2+3D+\eps_2^2)/(24D)\), giving the stated \(c\).

Solving \eqref{eq:normalized-u-series} for the framed sheaf series cancels the one-leg factor in \eqref{eq:root-continuation} and leaves \(\left(-\sqrt{tq}\,Q/u;t,q\right)_\infty\).  Since \(Y_{\mathrm{edge}}=-M-A/2=A-T\), equation \eqref{eq:exact-double-Pochhammer-connection} rewrites it as
\[
 e^{W(Y_{\mathrm{edge}})}
 \left(-\sqrt{tq}\,\frac uQ;t,q\right)_\infty
 \Phi(Y_{\mathrm{edge}}).
\]
Multiplying by \(\Phi(Y_{\mathrm{root}})^{-1}\) from \eqref{eq:root-continuation} and using \eqref{eq:polynomial-decomposition} gives \eqref{eq:full-u-zero-expansion} coefficientwise in \(\Lambda^4\); Conjecture~\ref{conj:P2-asymptotic-resummation} assembles these identities in the chosen convergent product representation.

Assumptions (A2)--(A3) place the specified product expansion of \(\mathcal R\) in \(1+\mathfrak m_A\).  Since \(\mathscr A_{\mathrm c}\in\K_{\eps^\circ}[[Q,u]]\) after the equivariant specialization, continuity of \(\operatorname{CT}_A\) and Lemma~\ref{lem:u-zero-regularity} give \eqref{eq:u-zero-order-zero}.
\end{proof}

\begin{lemma}
\label{lem:two-chart-factor}
For the chart substitutions of Lemma~\ref{lem:NY-chart-substitutions}, let
\begin{equation}
 A_i=A+\eps_i(2n+k),\qquad
 M_i=M+\eps_i\eta_{k,d_{\mathrm{NY}}},\qquad
 \eta_{k,d_{\mathrm{NY}}}
 =\frac{3-2d_{\mathrm{NY}}-k}{2}.
\label{eq:centered-additive-chart-data}
\end{equation}
The quotient of the exponentials of \(G_{\alpha,\beta}(A,M)+\widetilde P_M(M;\alpha,\beta)\) at the two chart substitutions by the unshifted exponential is
\begin{equation}
 (-1)^n e^{-\pi\mathrm{i}k/2}
 \exp\left[
 \left(\frac1{12}-\frac{\eta_{k,d_{\mathrm{NY}}}^2}{3}\right)M
 -s\left(\frac{\eta_{k,d_{\mathrm{NY}}}^3}{9}
          +\frac1{24}\right)\right].
\label{eq:two-chart-factor}
\end{equation}
\end{lemma}

The exponential is independent of \(n\).  For \((k,d_{\mathrm{NY}})=(0,1)\) or \((0,2)\) it is also independent of \(M\), and the quotient is independent of \(u\).

\begin{proof}
Let \(\eta=\eta_{k,d_{\mathrm{NY}}}\), and let \(\Delta f\) denote the sum of the two chart values of \(f\) minus its unshifted value.  The \(A\)-dependent term of \(G_{\alpha,\beta}\) is \(-\pi\mathrm{i}(\alpha+\beta)A/(2\alpha\beta)\).  Substituting the two chart weights, including the shifts of \(M\) in \eqref{eq:centered-additive-chart-data},
\begin{equation}
 \Delta G
 =-\pi\mathrm{i}\left(n+\frac k2\right)
  +\frac{s(\eta-1)}{24}.
\label{eq:gamma-two-chart}
\end{equation}
The coefficients of the unshifted additive variables cancel by
\[
 \frac1{\eps_1(\eps_2-\eps_1)}
 +\frac1{(\eps_1-\eps_2)\eps_2}
 =\frac1{\eps_1\eps_2}.
\]
Direct expansion gives
\begin{equation}
\begin{aligned}
 &\widetilde P_M
 (M+\eps_1\eta;\eps_1,\eps_2-\eps_1)
 +\widetilde P_M
 (M+\eps_2\eta;\eps_1-\eps_2,\eps_2)\\
 &\qquad-\widetilde P_M
 (M;\eps_1,\eps_2)\\
 &=\left(\frac1{12}-\frac{\eta^2}{3}\right)M
 -s\left(\frac{\eta^3}{9}+\frac{\eta}{12}\right)
 +\frac{\eta s}{24}.
\end{aligned}
\label{eq:M-polynomial-transformation}
\end{equation}
Combining \eqref{eq:gamma-two-chart} and \eqref{eq:M-polynomial-transformation} gives \eqref{eq:two-chart-factor}.

For \((k,d_{\mathrm{NY}})=(0,1)\), \(\eta=1/2\) and the coefficient of \(M\) vanishes:
\begin{equation}
 \Delta G=-\pi\mathrm{i}n-\frac{s}{48},\qquad
 \Delta\widetilde P_M=-\frac{5s}{144},
 \qquad
 \exp(\Delta G+\Delta\widetilde P_M)
 =(-1)^ne^{-s/18}.
\label{eq:unity-common-factor}
\end{equation}
For \((k,d_{\mathrm{NY}})=(1,1)\), \(\eta=0\) and the exponential in \eqref{eq:two-chart-factor} is
\begin{equation}
 \exp\left(\frac{M}{12}-\frac{s}{24}\right).
\label{eq:zero-common-M-factor}
\end{equation}
It depends on \(M\), hence on \(A\) and \(T\) along \(M=T-3A/2\), but it is independent of \(n\) and common to every summand of the vanishing equation, so we divide by it before expanding in \(\mathscr H_{\mathfrak S}\) and applying \(\operatorname{CT}_A\); the right side stays \(0\).  The constants in \eqref{eq:full-u-zero-expansion} contribute the further quotient
\[
 \exp\bigl(c(\eps_1,\eps_2-\eps_1)
 +c(\eps_1-\eps_2,\eps_2)-c(\eps_1,\eps_2)\bigr),
\]
independent of \(n,A,M,T\).  Substitution of the explicit \(c\) in Lemma~\ref{lem:full-u-zero-expansion} shows that its exponent is zero, so this quotient is \(1\).  The lift \eqref{eq:Lambda-continuation-branch} contributes \(1\) when \(k=0\) and one further nonzero constant when \(k=1\), which is removed together with \eqref{eq:zero-common-M-factor}.
\end{proof}

Recall \(R_{r,n}\), \(m_r(n)\) and \(\gamma_r(n)\) from \eqref{eq:P2-unity-exponents} and \eqref{eq:P2-vanishing-exponents}.

\begin{lemma}
\label{lem:bilateral-u-zero-limit}
Assume Conjectures~\ref{conj:P2-vertex-comparison} and \ref{conj:P2-asymptotic-resummation}, and fix the angular region \(\mathfrak S\) above.  For
\[
 (k,d_{\mathrm{NY}})=(0,1),\qquad (0,2),\qquad (1,1),
\]
remove the factors of Lemma~\ref{lem:two-chart-factor} from the bilateral sum over \(n\in\Z\); in the last case also remove \(Q^{1/3}\), the factor \eqref{eq:zero-common-M-factor}, and the constants described in the proof of that lemma.  The resulting sums lie in \(\mathscr H_{\mathfrak S}\), and \(\operatorname{CT}_A\) may be applied term by term.
\end{lemma}

\begin{proof}
Substituting \(R=3n+\tfrac12,3n-\tfrac12,3n+\tfrac32\) in \eqref{eq:P2-degree-zero-polynomial} gives the initial \(Q\)-exponents
\[
 m_1(n)
 \quad\text{for }(k,d_{\mathrm{NY}})=(0,1),\qquad
 m_{-1}(n)
 \quad\text{for }(k,d_{\mathrm{NY}})=(0,2),
\]
and
\[
 \frac13+m_3(n)
 \quad\text{for }(k,d_{\mathrm{NY}})=(1,1).
\]
In the last case, first remove \(Q^{1/3}\).  By Lemma~\ref{lem:completed-semigroup-support}, every translated \(\mathcal R\)-factor and every translated \(\mathscr A_{\mathrm c}\)-factor lies in \(\mathscr H_{\mathfrak S}\), and all three functions \(m_r\) have finite sublevel sets on \(\Z\).  Let \(\mathbf e^\gamma\) be a monomial of the product of the two translated \(\mathcal R\)-factors, with \(v(\gamma)=\delta\), and write the two monomials selected from the translated \(\mathscr A_{\mathrm c}\)-factors as \(Q^{a_1}u^{b_1}\) and \(Q^{a_2}u^{b_2}\), \(a_i,b_i\geq0\); their chart translations change only the coefficients.  A term with total \(Q\)-exponent \(d_{\mathrm{tot}}\) satisfies
\[
 m_r(n)+a_1+a_2+\delta=d_{\mathrm{tot}},
\]
and since \(\delta\geq-Ch(\gamma)\),
\[
 m_r(n)\leq d_{\mathrm{tot}}+Ch(\gamma).
\]
To check the support condition of Definition~\ref{def:completed-semigroup-algebra}, fix bounds \(h\leq H\) and \(v\leq V_0\).  The \(h\)-value of the term is
\[
 h(\gamma)+(b_1+b_2)h(1,0)
\]
and its \(v\)-value is
\[
 v(\gamma)+m_r(n)+a_1+a_2.
\]
Since \(h(1,0)>0\), the first bound leaves finitely many \(b_1+b_2\), hence finitely many \((b_1,b_2)\).  It also gives \(h(\gamma)\leq H\), while the second gives \(v(\gamma)\leq V_0\) since \(m_r(n),a_1,a_2\geq0\), and the common finitely generated monoid has only finitely many such \(\gamma\), independently of \(n\): the \(h\)-bound bounds all positive-\(h\) generator multiplicities, and the \(v\)-bound then bounds the multiplicity of \((0,1)\).  For each \(\gamma\), \eqref{eq:semigroup-cone-bound} gives
\[
 m_r(n)\leq V_0-v(\gamma)\leq V_0+CH.
\]
Only finitely many \(n\) occur, and once \(\gamma,n,b_1,b_2\) are fixed, the inequality \(a_1+a_2\leq V_0-v(\gamma)-m_r(n)\) leaves finitely many \((a_1,a_2)\).  The bilateral sum therefore has locally finite support and lies in \(\mathscr H_{\mathfrak S}\).  This is a coefficientwise sum of a locally finite family; it need not be the limit of its partial sums in the \(h\)-filtration, since the \(Q\)-series have \(h=0\).

By Lemma~\ref{lem:completed-semigroup-support}, the chart translations preserve \(h,v\) and commute with \(\operatorname{CT}_A\) with the induced \(Q\)-substitution.  The coefficientwise finiteness just proved allows \(\operatorname{CT}_A\) to be applied term by term; its \(h=0\) part selects the constant term \(1\) of each \(\mathcal R\)-factor and the constant term in \(u\) of each \(\mathscr A_{\mathrm c}\)-factor, which is the shifted series \(\mathcal Z\) by \eqref{eq:u-zero-order-zero}.
\end{proof}

\begin{proposition}
\label{prop:u-zero-specialization}
Assume Conjectures~\ref{conj:P2-vertex-comparison} and \ref{conj:P2-asymptotic-resummation}, fix the branch \eqref{eq:Lambda-continuation-branch} and the angular region \(\mathfrak S\), and let
\[
 (k,d_{\mathrm{NY}})=(0,1),\ (0,2),\text{ or }(1,1),
 \qquad r=3-2d_{\mathrm{NY}}+2k.
\]
After the \(n\)-independent factors of Lemma~\ref{lem:two-chart-factor} are removed from the corresponding Nakajima--Yoshioka identity, its Laurent expansion lies in \(\mathscr H_{\mathfrak S}\), and applying \(\operatorname{CT}_A\) gives
\begin{equation}
\begin{aligned}
 &\sum_{n\in\Z}(-1)^n
 \widehat Z_X
 (T+\eps_1R_{r,n};\eps_1,\eps_2-\eps_1)\\
 &\hspace{24mm}\times
 \widehat Z_X
 (T+\eps_2R_{r,n};\eps_1-\eps_2,\eps_2)\\
 &\qquad=\Lambda_{X,r}\widehat Z_X(T;\eps_1,\eps_2),
\end{aligned}
\label{eq:u-zero-specialization}
\end{equation}
with \(\Lambda_{X,r}\) as in \eqref{eq:X-bilinear-coefficients}.  For \((k,d_{\mathrm{NY}})=(1,1)\), the removed factors also include \(Q^{1/3}\), the factor \eqref{eq:zero-common-M-factor}, and the constants of the proof of Lemma~\ref{lem:two-chart-factor}.
\end{proposition}

The formal identity \eqref{eq:u-zero-specialization} does not depend on the angular region; the lift \eqref{eq:Lambda-continuation-branch} remains fixed.

\begin{proof}
Apply Theorem~\ref{thm:NY-native} with the chart variables of Lemma~\ref{lem:NY-chart-substitutions}, and insert \eqref{eq:full-u-zero-expansion} for the unshifted function and the two chart functions.  For \(k=0\), Lemma~\ref{lem:two-chart-factor} leaves the sign \((-1)^n\) and a factor independent of \(n\) and \(T\).  For \((k,d_{\mathrm{NY}})=(1,1)\), division by \eqref{eq:zero-common-M-factor} and by the nonzero constants of that proof again leaves only the \(n\)-dependence \((-1)^n\).

By Conjecture~\ref{conj:P2-asymptotic-resummation}, the continued coefficient identities assemble into the chosen normally convergent product representations; Lemmas~\ref{lem:full-u-zero-expansion} and \ref{lem:completed-semigroup-support} place their Laurent expansions in \(\mathscr H_{\mathfrak S}\), and Lemma~\ref{lem:bilateral-u-zero-limit} allows \(\operatorname{CT}_A\) to be applied term by term to the lattice sum.  Equation~\eqref{eq:u-zero-order-zero} replaces each \(\mathscr A_{\mathrm c}\)-factor by the corresponding shifted series \(\mathcal Z\), which together with the polynomial factor gives \(\widehat Z_X\).  In the two unity cases, \eqref{eq:P2-perturbative-shift} at \(R=1/2\) and \(R=-1/2\) gives \(e^{(\eps_1+\eps_2)/18}=\Lambda_{X,1}\) and \(e^{-(\eps_1+\eps_2)/18}=\Lambda_{X,-1}\).  The divided vanishing equation gives \(\Lambda_{X,3}=0\).

After \(\operatorname{CT}_A\) is applied, the expression involves only \(\mathcal Z\), the polynomial factor and the chart shifts, none of which depends on \(\mathfrak S\).  Any angular region satisfying the conditions above therefore gives the same identity.
\end{proof}

\end{document}